\documentclass [leqno, spanish, 10pt]{amsart}
\usepackage[T1]{fontenc}
\usepackage[utf8]{inputenc}

\usepackage{etoolbox} 	
\usepackage{graphicx}
\usepackage{epstopdf}

\usepackage[all]{xy}
\usepackage{amssymb, amsmath, amsfonts, mathtools, amsthm, amsbsy, amscd,stmaryrd}
\usepackage[bbgreekl]{mathbbol} 
\usepackage[mathscr]{euscript}
\usepackage{esint}
\usepackage{upgreek}
\usepackage{xypic}
\usepackage{mathcomp}

\usepackage{url}
\usepackage[linktocpage]{hyperref}

\usepackage{etoolbox}
\apptocmd{\sloppy}{\hbadness 10000\relax}{}{}

\usepackage{thmtools} 
\newtheorem{theorem}{Theorem}[section]
\newtheorem{corollary}[theorem]{Corollary}
\newtheorem{lemma}[theorem]{Lemma}

\theoremstyle{definition}
\newtheorem{example}[theorem]{Example}
\newtheorem{remark}[theorem]{Remark}

\numberwithin{equation}{section}

\def\cprime{$'$}

\newcommand{\fili}{\mathbb{F}}
\newcommand{\rvf}{H}
\newcommand{\livf}{L}

\newcommand{\lin}{\ell}
\renewcommand{\S}{\mathbb{S}}

\newcommand{\cent}{\mathbb{c}}

\newcommand{\mult}{\mathscr{M}}
\newcommand{\mnil}{\mathbb{M}}
\newcommand{\rnil}{\mathbb{R}}
\newcommand{\lnil}{\mathbb{L}}
\newcommand{\triv}{\mathbb{T}}

\newcommand{\cayn}{\mathbb{C}_{n}}

\newcommand{\dev}{\mathsf{dev}}
\newcommand{\hol}{\mathsf{h}}
\newcommand{\trans}{\mathbb{t}}

\newcommand{\amc}{\mathcal{H}}

\newcommand{\std}{\mathbb{V}^{\ast}}

\newcommand{\Om}{\Omega}

\newcommand{\amg}{\mathscr{A}}

\newcommand{\stw}{\mathbb{W}}

\newcommand{\imt}{\iota}

\newcommand{\A}{\mathsf{A}}
\newcommand{\B}{\mathsf{B}}

\newcommand{\kc}{\mathsf{c}}

\newcommand{\fie}{\mathbb{k}}

\newcommand{\lc}{\Sigma}

\renewcommand{\H}{\mathsf{H}}

\newcommand{\alg}{\mathbb{A}}

\newcommand{\balg}{\mathbb{B}}

\newcommand{\mlt}{\star}

\newcommand{\hess}{\operatorname{Hess}}

\newcommand{\vol}{\mathsf{vol}}

\renewcommand{\part}{\vdash}

\newcommand{\rad}{\mathbb{e}}

\newcommand{\Id}{\text{Id}}

\newcommand{\dum}{\,\cdot\,\,}
\newcommand{\Ga}{\Gamma}

\newcommand{\nm}{\mathsf{W}}

\newcommand{\I}{\mathbb{I}}
\newcommand{\J}{\mathbb{J}}

\newcommand{\la}{\lambda}
\newcommand{\ep}{\epsilon}

\newcommand{\zmodtwo}{\mathbb{Z}/2\mathbb{Z}}
\newcommand{\reat}{\mathbb{R}^{\times}}

\newcommand{\reap}{\mathbb{R}^{+}}

\newcommand{\cinf}{C^{\infty}}

\newcommand{\eno}{\text{End}}

\newcommand{\si}{\sigma}
\newcommand{\pr}{\partial}

\newcommand{\sign}{\operatorname{sgn}}

\newcommand{\im}{\operatorname{Im}}
\newcommand{\Aff}{\mathbb{Aff}}
\newcommand{\aff}{\mathbb{aff}}

\newcommand{\lb}{\langle}
\newcommand{\ra}{\rangle}

\newcommand{\ste}{\mathbb{V}}

\newcommand{\al}{\alpha}
\newcommand{\be}{\beta}
\newcommand{\ga}{\gamma}

\newcommand{\hnabla}{\widehat{\nabla}}
\newcommand{\tnabla}{\tilde{\nabla}}
\newcommand{\gl}{\mathfrak{gl}}

\newcommand{\g}{\mathfrak{g}}

\newcommand{\n}{\mathfrak{n}}

\newcommand{\ad}{\text{ad}}
\newcommand{\Ad}{\text{Ad}}

\newcommand{\tensor}{\otimes}

\newcommand{\rea}{\mathbb R}

\newcommand{\tr}{\operatorname{\mathsf{tr}}}

\newcommand{\parab}{\mathbb{Par}}

\begin{document}
\title{Left-symmetric algebras and homogeneous improper affine spheres}
\author{Daniel J.~F. Fox} 
\address{Departamento de Matemática Aplicada a la Ingeniería Industrial \\ Escuela Técnica Superior de Ingeniería y Diseño Industrial\\ Universidad Politécnica de Madrid\\Ronda de Valencia 3\\ 28012 Madrid España}
\email{daniel.fox@upm.es}

\begin{abstract}
The nonzero level sets in $n$-dimensional flat affine space of a translationally homogeneous function are improper affine spheres if and only if the Hessian determinant of the function is equal to a nonzero constant multiple of the $n$th power of the function. 
The exponentials of the characteristic polynomials of certain left-symmetric algebras yield examples of such functions whose level sets are analogues of the generalized Cayley hypersurface of Eastwood-Ezhov. There are found purely algebraic conditions sufficient for the characteristic polynomial of the left-symmetric algebra to have the desired properties. Precisely, it suffices that the algebra has triangularizable left multiplication operators and the trace of the right multiplication is a Koszul form for which right multiplication by the dual idempotent is projection along its kernel, which equals the derived Lie subalgebra of the left-symmetric algebra.  
\end{abstract}

\maketitle
\setcounter{tocdepth}{1}  
\tableofcontents

\section{Introduction}
Let $\hnabla$ be the standard flat affine connection on $\rea^{n+1}$ and fix a $\hnabla$-parallel volume form $\Psi$.  Define $\H(F)$ by $\det \hess F = \H(F) \Psi^{\tensor 2}$. An immersed hypersurface $\Sigma$ in $\rea^{n+1}$ is \textit{nondegenerate} if its second fundamental form with respect to $\hnabla$ is nondegenerate. In this case, there is a distinguished equiaffinely invariant transverse vector field defined along $\Sigma$, the \textit{equiaffine normal}. A nondegenerate connected hypersurface $\Sigma$ is an \textit{improper affine sphere} if its equiaffine normals are parallel. 
By Theorem \ref{ahtheorem}, the level sets of a smooth translationally homogeneous function $F$ on $\rea^{n+1}$ satisfying 
\begin{align}\label{ma}
\H(F) = \kc F^{n+1}
\end{align}
for some $\kc \neq 0$ are improper affine spheres. That $F$ be translationally homogeneous means that there is a constant vector $\rad^{i} \in \rea^{n+1}$ such that $F(x + t \rad) = e^{\la t}F(x)$ for all $t \in \rea$ and $x \in \rea^{n+1}$.

The main result reported here is Theorem \ref{triangularizabletheorem0}, which yields translationally homogeneous solutions of \eqref{ma} having the form $F = e^{P}$ where $P$ is a weighted homogeneous polynomial arising as the characteristic polynomial of the left-symmetric algebra associated with the prehomogeneous action of a simply-connected solvable Lie group. These solutions of \eqref{ma} have the nice properties that they are defined and nonvanishing on all of $\rea^{n+1}$ and their level sets are connected, everywhere nondegenerate graphs that are homogeneous with respect to the action of a group of affine transformations. 

An \textit{algebra} $(\alg, \mlt)$ means a finite-dimensional vector space $\alg$ with a bilinear product (multiplication) $\mlt: \alg \times \alg \to \alg$ that need not be either unital or associative. Here mainly Lie algebras and left-symmetric algebras are considered, although other algebras are mentioned occasionally. A \textit{left-symmetric algebra}\footnote{Left-symmetric algebras are also called \textit{pre-Lie algebras}, \textit{Vinberg algebras}, \textit{Koszul-Vinberg algebras}, and \textit{chronological algebras}, and some authors prefer to work with the opposite category of \textit{right-symmetric algebras}.}  (abbreviated \textit{LSA}) $(\alg, \mlt)$ is a vector space $\alg$ equipped with a multiplication $\mlt$ such that the associated skew-symmetric bracket $[a, b] = a\mlt b - b\mlt a$ satisfies the Jacobi identity, so makes $\alg$ into a Lie algebra, and such that the left regular representation $L:\alg \to \eno(\alg)$ defined by $L(a)b = a\mlt b$ is a Lie algebra representation, meaning $[L(a), L(b)] = L([a, b])$. Equivalently, the right regular representation $R:\alg \to \eno(\alg)$ defined by $R(a)b = b\mlt a$ satisfies
\begin{align}\label{rlsa}
R(x\mlt y) - R(y)R(x) = [L(x), R(y)].
\end{align}
By \eqref{rlsa} the \textit{trace form} $\tau$ defined on $\alg$ by $\tau(x, y) = \tr R(x)R(y) = \tr R(x\mlt y)$ is symmetric. An LSA is \textit{incomplete} if the linear form $\tr R$ is not zero. For an incomplete LSA with nondegenerate trace form $\tau$, the unique element $r \in \alg$ such that $\tr R(x) = \tau(r, x)$ for all $x \in \alg$, is an idempotent called the \textit{right principal idempotent}. An LSA $(\alg, \mlt)$ defined over a field $\fie$ of characteristic zero is \textit{triangularizable} if there is a basis of $\alg$ with respect to which every $L(x)$ is triangular. By Lemma \ref{cslemma} this is equivalent to the condition that the underlying Lie algebra $(\alg, [\dum, \dum])$ is solvable and for every $x \in \alg$ the eigenvalues of $L(x)$ are contained in $\fie$.

\begin{theorem}\label{triangularizabletheorem0}
Let $(\alg, \mlt)$ be a triangularizable $n$-dimensional LSA over a field of characteristic zero and having nondegenerate trace form $\tau$ and codimension one derived Lie subalgebra $[\alg, \alg]$. Let $G$ be the simply-connected Lie group with Lie algebra $(\alg, [\dum, \dum])$. There are a nonzero constant $\kc$ and a closed unimodular subgroup $H \subset G$ having Lie algebra $[\alg, \alg]$, such that the characteristic polynomial $P(x) = \det(I + R(x))$ of $(\alg, \mlt)$ solves $\H(e^{P}) = \kc e^{nP}$, and the level sets of $P$ are improper affine spheres homogeneous for the action of $H$ and having affine normals equal to a constant multiple of the right principal idempotent $r$. 
\end{theorem}
The translational homogeneity of $e^{P}$ is equivalent to the identity $P(x + tr) = P(x) + t$, while the weighted homogeneity of $P$ is the statement that $dP(E) = P$ where $E$ is the vector field $E_{x} = r + r\mlt x$. 

LSAs were introduced by Vinberg in \cite{Vinberg} as a tool in the classification of homogeneous convex cones. 
In \cite{Vinberg}, a triangularizable LSA with a Koszul form for which the associated metric is positive definite is called a \textit{clan}; see also \cite{Shima-homogeneoushessian}, where there is studied the more general class of triangularizable LSAs equipped with a positive definite Hessian metric (the definitions are recalled in section \ref{hessiansection}). In \cite{Vinberg}, Vinberg showed that the automorphism group of a homogeneous convex cone contains a triangularizable solvable subgroup acting simply transitively on the cone, and established a bijective correspondence between clans and homogeneous convex cones. Although it has not been completely developed, there should be a correspondence similar to that for homogeneous convex cones relating a different sort of prehomogeneous actions of solvable Lie groups with the domains bounded by homogeneous improper affine spheres. 
In the special case of a convex cone that is a component of the complement of the zero set of the fundamental relative invariant of a real form of an irreducible regular prehomogeneous vector space, the relative invariant $Q$ of the prehomogeneous vector space is among the relative invariants of this triangular subgroup. If the underlying space has dimension $n$, $Q$ solves an equation of the form $\H(Q) = \kc Q^{m}$ where $m = n(k-2)/k$ and $k = \deg Q$. The equation $\H(P) = \kc P^{n}$ results formally when the homogeneity degree $k$ tends to $\infty$, so in some formal sense the analogue for this equation of degree $k$ homogeneous polynomial solutions of $\H(Q) = \kc Q^{m}$ should be functions that somehow can be regarded as polynomials homogeneous of infinite degree. The conclusion of Theorem \ref{triangularizabletheorem0} shows that this makes sense if one regards a translationally homogeneous exponential of a weighted homogeneous polynomial as having infinite homogeneity degree.
The point relevant here is that the $P$ of Theorem \ref{triangularizabletheorem0} is relatively invariant for the action of the simply-connected Lie group corresponding to the Lie algebra underlying the LSA, so that Theorem \ref{triangularizabletheorem0} fits the case of improper affine spheres in a common framework with the case of proper affine spheres studied in \cite{Fox-prehom}.

Section \ref{affinespheresection} presents the needed background on improper affine spheres. Theorem \ref{ahtheorem} shows that the level sets of a translationally homogeneous function are improper affine spheres if and only if the function solves a Monge-Ampère equation of the form \eqref{ma}. Lemma \ref{improperlemma} shows the equivalence of \eqref{ma} to various other equations of Monge-Ampère type; these alternative formulations are used in the proof of Theorem \ref{triangularizabletheorem0}.

Section \ref{impropersection} reviews background on left-symmetric algebras, affine actions, and completeness. Although most of this material can be found in other sources, it is recalled here to have in one place all that is needed in subsequent sections. Following H. Shima, an LSA is Hessian if it admits a nondegenerate symmetric bilinear form (a metric) satisfying the compatibility condition \eqref{hessianmetric}. Section \ref{hessiansection} treats Hessian LSAs. The technical Lemma \ref{principalidempotentlemma} generalizes to indefinite signature Hessian LSAs results obtained for clans by H. Shima and E.~B. Vinberg. Theorem \ref{lsacptheorem} gives conditions on a Hessian LSA that in conjunction with Theorem \ref{ahtheorem} guarantee that the level sets of its characteristic polynomial are improper affine spheres. 

There are many notions of nilpotence used in studying LSAs and section \ref{nilpotencesection} discusses the interrelationships between those most relevant here. Some of the results obtained have purely algebraic interest. Theorem \ref{trivalgtheorem} shows that a finite-dimensional LSA over a field of characteristic zero is nilpotent if and only if it is right nilpotent with nilpotent underlying Lie algebra. The reader should see section \ref{nilpotencesection} for the definitions because terminology related to notions of nilpotent varies with the source; here the conventions follow those standard in the study of nonassociative algebras, so that an algebra is \textit{nilpotent} if the associative multiplication algebra generated by all left and right multiplication operators is nilpotent. 

By Lemma \ref{rightnilpotentlemma} such a right nilpotent LSA with nilpotent underlying Lie algebra is triangularizable, and there results the following corollary of Theorem \ref{triangularizabletheorem}.

\begin{corollary}\label{triangularizabletheorem2}
Let $(\alg, \mlt)$ be an $n$-dimensional LSA over a field of characteristic zero that is right nilpotent with nilpotent underlying Lie algebra. Suppose the trace-form $\tau$ is nondegenerate and the derived Lie subalgebra $[\alg, \alg]$ has codimension one. Let $G$ be the simply-connected Lie group with Lie algebra $(\alg, [\dum, \dum])$. There are a nonzero constant $\kc$ and a closed unimodular subgroup $H \subset G$ having Lie algebra $[\alg, \alg]$, such that the characteristic polynomial $P(x) = \det(I + R(x))$ of $(\alg, \mlt)$ solves $\H(e^{P}) = \kc e^{nP}$, and the level sets of $P$ are improper affine spheres homogeneous for the action of $H$ and having affine normals equal to a constant multiple of the right principal idempotent $r$. 
\end{corollary}
By Theorem \ref{trivalgtheorem} the nilpotency hypothesis of Corollary \ref{triangularizabletheorem2} can be restated simply as that $(\alg, \mlt)$ be nilpotent.

Theorem \ref{triangularizabletheorem} gives a sort of weight space decomposition of an LSA as in Theorem \ref{triangularizabletheorem0} that is useful in constructing examples. Although this is not developed systematically, section \ref{examplesection} concludes with some illustrative examples obtained by applying Theorem \ref{triangularizabletheorem0}. 

A motivating example, treated in Example \ref{cayleyexample}, is given by the $n$th generalized \textit{Cayley hypersurface}, defined by M. Eastwood and V. Ezhov in \cite{Eastwood-Ezhov} as the zero level set of the polynomial
\begin{align}\label{eepolynomials}
\Phi_{n}(x_{1}, \dots, x_{n}) = \sum_{i = 1}^{n}(-1)^{i}\frac{1}{i}\sum_{j_{1} + \dots + j_{i} = n}x_{j_{1}}\dots x_{j_{i}} = \sum_{\la \part n}(-1)^{|\la|}\frac{c_{\la}}{|\la|}x_{(\la)},
\end{align}
where the second sum is over all partitions $\la$ of $n$; $|\la|$ is the length of the partition $\la$; $x_{(\la)} = x_{1}^{m_{1}}\dots x_{n}^{m_{n}}$, where $m_{i}$ is the multiplicity of $i$ in $\la$; and $c_{\la}$ is the number of integer compositions of $n$ determining the partition $\la$. Eastwood and Ezhov prove that the Cayley hypersurface is an improper affine sphere admitting a transitive abelian group of affine motions and whose full symmetry group has one-dimensional isotropy. They ask if these properties characterize these hypersurfaces, and with the additional assumption that the domain above the hypersurface is homogeneous this was proved by Y. Choi and H. Kim in \cite{Choi-Kim} using the theory of LSAs. Relations between homogeneous improper affine spheres, LSAs, and Monge-Ampère equations like that in Theorem \ref{triangularizabletheorem0} and Theorems \ref{lsacptheorem} and \ref{triangularizabletheorem} in section \ref{impropersection} have been studied by Choi and Kim and K. Chang in the papers \cite{Choi-domain}, \cite{Choi-Chang}, and \cite{Choi-Kim}, that address a characterization of the generalized Cayley hypersurfaces conjectured in \cite{Eastwood-Ezhov}. Their work, as well as that of H. Shima \cite{Shima-homogeneoushessian} and A. Mizuhara \cite{Mizuhara, Mizuhara-solvable}, provided motivation for the material described here. In example \ref{cayleyexample} there is constructed for each positive integer $n$ an LSA $(\cayn, \mlt)$ that satisfies the hypotheses of Theorem \ref{triangularizabletheorem0} and, by Lemma \ref{cayleypolynomiallemma}, has the polynomial $P_{n} = 1 - n\Phi_{n}$ as its characteristic polynomial. This gives an alternative demonstration that the Cayley hypersurfaces are homogeneous improper affine spheres with the properties demonstrated in \cite{Eastwood-Ezhov}. A consequence, also proved in Lemma \ref{cayleypolynomiallemma}, of the realization of $1 - n\Phi_{n}$ as a determinant, is the recursive formula 
\begin{align}\label{cayleyrecursion2}
\Phi_{n}(x_{1}, \dots, x_{n}) = - x_{n} + \sum_{i = 1}^{n-1}(\tfrac{i}{n} - 1)x_{i}\Phi_{n-i}(x_{1}, \dots, x_{n-i}),
\end{align}
determining $\Phi_{n}$ (where $\Phi_{1}(x) = -x$). 

\section{Improper affine spheres as level sets}\label{affinespheresection}
This section gives the background on improper affine spheres and translationally homogeneous functions needed to understand the statement and proof of Theorem \ref{triangularizabletheorem0}. 
The reader primarily interested in left-symmetric algebras can skip directly to section \ref{impropersection}.

The group $\Aff(n+1, \rea)$ of affine transformations of $\rea^{n+1}$ comprises the automorphisms of the standard flat affine connection $\hnabla$ on $\rea^{n+1}$. Elements of its subgroup preserving the tensor square $\Psi^{2}$ of a fixed $\hnabla$-parallel volume form $\Psi$ are called \textit{unimodular affine} or \textit{equiaffine}. 

Let $\Sigma$ be a connected coorientable nondegenerate immersed hypersurface in $\rea^{n+1}$. Via the splitting $T\rea^{n+1} = T\Sigma \oplus \lb N\ra$ determined by a vector field $N$ transverse to $\Sigma$, the connection $\hnabla$ induces on $\Sigma$ a connection $\nabla$, a symmetric covariant two tensor $h$ representing the second fundamental form, a shape operator $S \in \Ga(\eno(T\Sigma))$, and the connection one-form $\tau \in \Ga(T^{\ast}\Sigma)$; these are defined by $\hnabla_{X}Y = \nabla_{X}Y + h(X, Y)N$ and $\hnabla_{X}N = -S(X) + \tau(X)N$, where $X$ and $Y$ are tangent to $\Sigma$. Here, as in what follows, notation indicating the restriction to $\Sigma$, the immersion, the pullback of $T\rea^{n+1}$, etc. is omitted. As generally in what follows, when indices are used the abstract index and summation conventions are employed and indices are to be understood as labels indicating valence and symmetries. Tensors on $\Sigma$ are labeled using capital Latin abstract indices. That $\Sigma$ be \textit{nondegenerate} means that the second fundamental form of $\Sigma$, equivalently $h_{IJ}$, is nondegenerate. Since by assumption $\Sigma$ is cooriented, it is orientable, and the interior multiplication $\imt(N)\Psi$ is a volume form on $\Sigma$. Since $\hnabla \Psi = 0$, for $X$ tangent to $\Sigma$, $\nabla_{X} \imt(N)\Psi = \tau(X)\imt(N)\Psi$. Let $\vol_{h} = q\imt(N)\Psi$ be the volume form induced on $\Sigma$ by $h$ and the orientation consistent with $\imt(N)\Psi$. Since $\vol_{h}^{2} = |\det h|$,
\begin{equation}\label{deth}
h^{PQ}\nabla_{I}h_{PQ} = 2\vol_{h}^{-1}\nabla_{I}\vol_{h}  = 2\left(q^{-1}dq_{I} + \tau_{I} \right).
\end{equation}
Any other transversal to $\Sigma$ has the form $\tilde{N} = a(N + Z)$ for a nowhere vanishing function $a$ and a vector field $Z$ tangent to $\Sigma$. The second fundamental form $\tilde{h}$, connection $\tnabla$, and connection one-form $\tilde{\tau}$ determined by $\tilde{N}$ and $\hnabla$ are related to $h$, $\nabla$, and $\tau$ by
\begin{align}\label{transform}
&\tilde{h}_{IJ} = a^{-1}h_{IJ},& &\tnabla = \nabla - h_{IJ}Z^{K}, & &\tilde{\tau}_{I} = \tau_{I} + a^{-1}da_{I} + h_{IP}Z^{P}.
\end{align}
It follows from \eqref{deth} and \eqref{transform} that 
\begin{equation}\label{normalize}
n\tilde{\tau}_{I} + \tilde{h}^{PQ}\tnabla_{I}\tilde{h}_{PQ} = n \tau_{I} + h^{PQ}\nabla_{I}h_{PQ} + (n+2)Z^{P}h_{IP}, 
\end{equation}
where $h^{IJ}$ and $\tilde{h}^{IJ}$ are the symmetric bivectors inverse to $h_{IJ}$ and $\tilde{h}_{IJ}$. Since \eqref{normalize} does not depend on $a$, the span of $\tilde{N}$ is determined by requiring $n\tilde{\tau}_{I} =- \tilde{h}^{PQ}\tnabla_{I}\tilde{h}_{PQ}$, so that, by \eqref{deth} and \eqref{normalize},
\begin{equation}\label{zdet}
Z^{P}h_{PI} =  -\tfrac{1}{n+2}\left(n\tau_{I} + h^{PQ}\nabla_{I}h_{PQ}\right) = -\tau_{I} - \tfrac{2}{n+2}q^{-1}dq_{I}= -\tfrac{1}{2}h^{PQ}\nabla_{I}h_{PQ} + \tfrac{1}{n+2}q^{-1}dq_{I}.
\end{equation}
Whatever is $a$, the resulting transversal $\tilde{N}$ is called an \textit{affine normal}, and the line field it spans is called the \textit{affine normal distribution of $\Sigma$}. Since $\det \tilde{h} = a^{-n}\det h$, the \textit{equiaffine normal} $\nm = a(N + Z)$ is determined up to sign by requiring $|\vol_{\tilde{h}}| = |\imt(\nm)\Psi|$, which forces $q = |a|^{(n+2)/2}$. 
By \eqref{transform}, the connection one-form associated with the equiaffine normal vanishes. Once a coorientation has been fixed, let $\nabla$, $h$, and $S$ be determined by the cooriented equiaffine normal. The pseudo-Riemannian metric $h_{IJ}$ is called the \textit{equiaffine metric}. The \textit{equiaffine mean curvature} is $\amc = n^{-1}S_{I}\,^{I}$. 

A coorientable nondegenerate connected hypersurface $\Sigma$ is an \textit{improper affine sphere} if its equiaffine normals are parallel. It is straightforward to check that $\Sigma$ is an improper affine sphere if and only if the shape operator determined by any affine normal vanishes identically. 

The definition of a connected affine sphere does not require a choice of coorientation, but some coherence condition on coorientations is necessary when there are multiple connected components. The convention used in this paper is the following. A smoothly immersed hypersurface having more than one connected component is an improper affine sphere if each connected component is an improper affine sphere and the affine normal lines of the different components are all parallel and there is a choice of coorientations of the components so that for the equiaffine normal consistent with this choice the signatures modulo $4$ of the equiaffine metrics of the different components are all the same. Note that if a disconnected hypersurface is an affine sphere with respect to a given choice of coorientations of the components, it is an affine sphere with respect to the opposite choice of coorientations, but with respect to no other choice of coorientations. In this sense, the definition is consistent with the definition for a connected hypersurface. 

Let $\Omega \subset \rea^{n+1}$ be an open domain (a nonempty open subset). For $F \in C^{k}(\Om)$ let $F_{i_{1}\dots i_{k}} = \hnabla_{i_{1}}\dots\hnabla_{i_{k-1}}dF_{i_{k}}$, and let $g_{ij} = (\hess F)_{ij} = F_{ij} = \hnabla_{i}dF_{j}$ be the \textit{Hessian} of $F$. As $\det \hess F$ and the tensor square $\Psi^{2}$ are $2$-densities, it makes sense to define the \textit{Hessian determinant} $\H(F)$ of a $C^{2}$ function $F$ by $\det \hess F = \H(F)\Psi^{2}$. If $x^{1}, \dots, x^{n+1}$ are coordinate functions such that $dx^{1}, \dots, dx^{n+1}$ is a $\hnabla$-parallel coframe and $\Psi = dx^{1}\wedge \dots \wedge dx^{n+1}$, then $\H(F) = \det \tfrac{\pr^{2}F}{\pr x^{i}\pr x^{j}}$. 

Where $\H(F)$ is nonzero, $g_{ij}$ is a pseudo-Riemannian metric with inverse symmetric bivector $g^{ij}$. In this case, indices are raised and lowered using $g_{ij}$ and $g^{ij}$, so, for example, $F^{i} = g^{ip}F_{p}$. There is written $|dF|_{g}^{2} = F^{p}F_{p}$, although this need not be positive when $g_{ij}$ is not positive definite. 

Let $\lin:\Aff(n+1, \rea) \to GL(n+1, \rea)$ be the projection onto the linear part. Because of the identity $g\cdot \H(F) = \det{}^{2}\lin(g) \H(g\cdot F)$, the equation 
\begin{align}\label{mai}
&\H(F) = \phi(F)
\end{align}
is affinely covariant in the sense that $F$ solves \eqref{mai} for some function $\phi$ if and only if $g\cdot F$ solves \eqref{mai} with a positive constant multiple of $\phi$ in place of $\phi$. In particular, it is natural to consider solutions of \eqref{mai} up to unimodular affine equivalence. Moreover, the affine covariance suggests also that properties of the equations \eqref{mai} should be reflected in the unimodular affine geometry of the level sets of $F$. 

An interesting general problem is the determination up to affine equivalence of all sufficiently smooth solutions of \eqref{mai} on a domain $\Om \subset \rea^{n+1}$ for some particular choice of $\phi$, e.g. when $\phi$ is a power or an exponential, and for particular choices of $\Om$. Of particular interest are cases of \eqref{mai} that admit solutions that are \textit{entire}, meaning defined on all of $\rea^{n+1}$, or \textit{polynomial}, meaning that $F_{i_{1}\dots i_{k}} = 0$ for some $k$.

Here the interest is in equations of the form \eqref{mai} whose solutions have level sets that are improper affine spheres. Requiring some kind of homogeneity property of the function $F$ restricts the possible forms of $\phi$ in \eqref{mai}. In particular, here there will be considered functions $F$ that are translationally homogeneous in the sense explained next, and that such a function solve an equation of the form \eqref{mai} forces $\phi$ to be a polynomial. The precise statement is Theorem \ref{ahtheorem}.

Let $\Omega \subset \rea^{n+1}$ be an open domain. For $F \in C^{0}(\Omega)$ and $r \in \rea$, let $\lc_{r}(F, \Omega) = \{x \in \Omega: F(x) = r\}$. For $\la\in \rea$ define $\amg^{\la}(\Omega)$ to comprise those $F \in C^{0}(\Omega) \cap \cinf(\Omega \setminus \lc_{0}(F, \Omega))$ for which there exists a parallel vector field (that is, a constant vector) $\rad^{i} \in \rea^{n+1}$ such that $F(x + t \rad) = e^{\la t}F(x)$ for all $t \in \rea$ and $x \in \Om$ such that $x + t \rad \in \Omega$.
An element of $\amg^{\la}$ is \textit{$\la$-translationally (affinely) homogeneous} with \textit{axis} $\rad^{i}$. For $F \in C^{0}(\Om)$ and $g \in \Aff(n+1, \rea)$ define $g\cdot F \in C^{0}(g \Om)$ by $(g \cdot F)(x) = F(g^{-1}x)$. Translational affine homogeneity is an affinely invariant condition in the sense that $F$ is translationally homogeneous if and only if $g\cdot F$ is translationally homogeneous for all $g \in \Aff(n+1, \rea)$. 
\begin{lemma}\label{affhomlemma}
A function $F \in \cinf(\Omega \setminus \lc_{0}(F, \Omega)) \cap C^{0}(\Omega) $ is in $\amg^{\la}(\Om)$ if and only if there is $\rad^{i} \in \rea^{n+1}$ such that $\rad^{p}F_{p} = \la F$. 
\end{lemma}
\begin{proof}
First suppose $F \in \amg^{\la}(\Om)$. If $x \in \Om$ and $F(x) \neq 0$ then, for any $t \in \rea$ such that $x + t\rad \in \Om$, $F(x + t\rad) = e^{\la t}F(x) \neq 0$, so $x + t\rad \notin \lc_{0}(F, \Om)$. Since $\Om$ is open, there is some small interval $I \subset \rea$ containing $0$ such that $x + t \rad \in \Om$ for $t \in I$. Hence $\rad^{p}F_{p}(x) = \tfrac{d}{dt}\big|_{t = 0}F(x + t\rad)  = \tfrac{d}{dt}\big|_{t = 0}\left(e^{\la t}F(x)\right) = \la F(x)$. Now suppose $F \in \cinf(\Om\setminus \lc_{0}(F, \Omega)) \cap C^{0}(\Omega)$ satisfies $\rad^{p}F_{p} = \la F$ for some fixed $\rad^{i} \in \rea^{n+1}$. Then $f(t) = F(x + t\rad)$ solves the initial value problem $f(0) = F(x)$ and $\tfrac{d}{dt}f(t) = \la f(t)$, so $F(x + t\rad) = f(t) = e^{\la t}F(x)$ for $t$ such that $x + t\rad \in \Om$.
\end{proof}

Let $\reat = GL(1, \rea)$ be the group of nonzero real numbers.
\begin{lemma}\label{homognondegenlemma}
Suppose given an open domain $\Omega \subset \rea^{n+1}$ and $F \in \amg^{\la}(\Om)$ for $\la \in \reat$. By Lemma \ref{affhomlemma} there is a vector field $\rad^{i} \in \rea^{n+1}$ such that $\rad^{p}F_{p} = \la F$. For $r \in \reat$, the level set  $\lc_{r}(F, \Omega)$ is smoothly immersed and transverse to $\rad^{i}$, and $\lc_{r}(F, \Omega)$ is nondegenerate if and only if $\H(F)$ does not vanish on $\lc_{r}(F, \Omega)$, in which case $dF$ and $|dF|^{2}_{g}$ do not vanish on $\lc_{r}(F, \Omega)$, and $\lc_{r}(F, \Om)$ is coorientable with equiaffine normal 
\begin{equation}\label{nm2}
\nm^{i}  =  -\la^{-1}(n+2)^{-1}\left|F \H(F)\right|^{1/(n+2)}\left(F^{-1}\rad^{i} + \la \H(F)^{-1}\H(F)^{i}\right).
\end{equation}
\end{lemma}

\begin{proof}
Since $\hnabla_{i}\rad^{j} = 0$, differentiating $\rad^{p}F_{p} = \la F$ yields $\rad^{p}F_{pi_{1}\dots i_{k}} = \la F_{i_{1}\dots i_{k}}$. Hence
\begin{align}\label{hompol1}
&\la F^{i}= \rad^{i},&
&|dF|^{2}_{g}  = F.
\end{align}
Tracing $\rad^{p}F_{ijp} = \la F_{ij}$ and combining the result with \eqref{hompol1} yields
\begin{align}\label{hompol2}
&\rad^{p}\H(F)_{p} = \H(F)\rad^{p}F_{pq}\,^{q} = \la(n+1)\H(F), &
\end{align}
Since for $x \in \lc_{r}(F, \Omega)$, $\rad^{i}F_{i}(x) = \la r  \neq 0$, $dF$ does not vanish on $\lc_{r}(F, \Omega)$ and so the level set $\lc_{r}(F, \Omega)$ is smoothly immersed; moreover, $\rad^{i}$ is transverse to $\lc_{r}(F, \Omega)$. Let $h_{IJ}$ be the corresponding second fundamental form. 
The restrictions $F_{IJ}$, $F_{Ip}\rad^{p}$, and $F_{I}$ satisfy
\begin{align}\label{hdl}
&F_{IJ} = -\la Fh_{IJ},& &F_{Ip}\rad^{p} = \la F_{I} = 0,& & F_{pq}\rad^{p}\rad^{q} = \la^{2}F,
\end{align}
along $\lc_{r}(F, \Omega)$. By \eqref{hdl}, $h_{IJ}$ is nondegenerate along $\lc_{r}(F, \Omega)$ if and only $\H(F)$ does not vanish along $\lc_{r}(F, \Omega)$. In this case, it follows from \eqref{hompol1} and $\la r \neq 0$ that neither $dF$ nor $|dF|^{2}_{g}$ vanishes along $\lc_{r}(F, \Omega)$.

The equiaffine normal of $\lc_{r}(F, \Om)$ has the form $\nm^{i} = a(\rad^{i} + Z^{i})$ where $Z^{p}F_{p} = 0$. Let $\vol_{h} = q\imt(\rad)\Psi$. By \eqref{deth}, $2q^{-1}q_{I} = h^{PQ}\nabla_{I}h_{PQ}$ and so, since $\hnabla_{i}\rad^{j} = 0$, it follows from \eqref{normalize} that $\la^{-1}F^{-1}(n+2)Z^{P}g_{IP} = -(n+2)Z^{P}h_{IP} = 2q^{-1}dq_{I}$. On the other hand, it follows from \eqref{hdl} that $q = |\la|^{-(n+2)/2}|F|^{-(n+1)/2}|\H(F)|^{1/2}$. Hence, by \eqref{hompol1} and \eqref{hompol2}, $Z^{i} = (n+2)^{-1}\la F(\H(F)^{-1}\H(F)^{i} -(n+1)F^{-1}F^{i})$ is tangent to $\lc_{r}(F, \Om)$, and $|a| = |\la|^{-1}|F|^{-(n+1)/(n+2)}|\H(F)|^{1/(n+2)}$. With the coorientation convention these formulas combine to yield \eqref{nm2}.
\end{proof}

Let $\sign:\reat \to \zmodtwo$ be the sign homomorphism $\sign(r) = r|r|^{-1}$. Define the \textit{standard coorientation} of a connected component of a level set of $F$ to be that consistent with the vector field $-\sign(|dF|^{2}_{g})F^{i}$, where $F^{i} = g^{ip}F_{p}$. That under the hypotheses of Theorem \ref{ahtheorem} this vector field is nonzero follows from Lemma \ref{homognondegenlemma}. 
Theorem \ref{ahtheorem} shows that the nonzero level sets of a translationally homogeneous solution of \eqref{ma} on $\rea^{n+1}$ are improper affine spheres. 

\begin{theorem}\label{ahtheorem}
Let $\la \in \reat$. Let $\Omega \subset \rea^{n+1}$ be a nonempty open subset, and let $I \subset \reat$ be a nonempty, connected, open subset. For $F \in \amg^{\la}(\Omega)$, let $\Omega_{I} = F^{-1}(I)\cap \Omega$. Let $\rad^{i} \in \rea^{n+1}$ be the axis of $F$. The following are equivalent.
\begin{enumerate}
\item \label{aht2} There is a nonvanishing function $\phi:I \to \rea$ such that $F$ solves $\H(F) = \phi(F)$ on $\Omega_{I}$.
\item \label{aht1improper} For all $r \in I$ each level set $\lc_{r}(F, \Omega_{I})$, equipped with the coorientation of its components consistent with $-\sign(|dF|^{2}_{g})F^{i}$, is an improper affine sphere with equiaffine normal equal to $c\rad^{i}$ for a constant $c$ depending only on $r$ (and not the connected component).
\end{enumerate}
When these conditions hold, there is a nonzero constant $\kc$ such that $\phi$ has the form $\phi(r) = \kc r^{n+1}$.
\end{theorem}

\begin{proof}
Suppose there holds \eqref{aht1improper}. That is, $F \in \amg^{\la}(\Omega)$ and there is a connected open interval $I \subset \rea \setminus\{0\}$ such that for all $r \in I$ each level set $\lc_{r}(F, \Omega_{I})$, equipped with the coorientation of its components consistent with $-\sign(|dF|^{2}_{g})F^{i}$, is an affine sphere with affine normal parallel to a fixed vector $\rad$. 
Because, by assumption, each connected component of $\lc_{r}(F, \Om_{I})$ is nondegenerate, Lemma \ref{homognondegenlemma} implies that neither $\H(F)$ nor $|dF|^{2}_{g}$ vanishes on $\lc_{r}(F, \Omega_{I})$. A posteriori, this justifies assigning to each component the coorientation given by $-\sign(|dF|^{2}_{g})F^{i}$. By assumption, the equaffine normal $\nm$ satisfies $\nm \wedge \rad = 0$ along $\lc_{r}(F, \Om_{I})$. Comparing with \eqref{nm2} shows that $d_{i}\log\H(F) = c\rad_{i} = c\la F_{i}$ for some $c$ locally constant on $\lc_{r}(F, \Om_{I})$. Contracting with $\rad^{i} = \la F^{i}$ and using \eqref{hompol2} yields $(n+1)\la = c\la^{2}F$, so that $d_{i}\log\H(F) = (n+1)F^{-1}F_{i}$. Hence $F^{-n-1}\H(F)$ is locally constant on $\lc_{r}(F, \Om_{I})$.  By assumption the signatures of the second fundamental forms of the connected components of $\lc_{r}(F, \Om_{I})$ are the same modulo $4$, and by \eqref{hdl} this implies that the signatures of $\hess F$ on the different connected components are the same modulo $4$, and so the signs of $\H(F)$ on the different connected components must be the same. This means that $|\H(F)|$ can be replaced by one of $\pm \H(F)$ coherently on all of $\lc_{r}(F, \Om_{I})$. Since by assumption there is $\rad^{i} \in \rea^{n+1}$ such that $\nm^{i} = c\rad^{i}$ for some $c$ depending only on $r$ and not the connected component of $\lc_{r}(F, \Om_{I})$, it follows from \eqref{nm2} that $\H(F)$ is constant on $\lc_{r}(F, \Om_{I})$. This is true for each $r \in I$, and so there is a function $\phi$ defined on $I$ such that $\H(F) = \phi(F)$ for $x \in \Omega_{I}$. This shows \eqref{aht1improper}$\implies$\eqref{aht2}. 

The implication \eqref{aht2}$\implies$\eqref{aht1improper} is proved as follows. If $F \in \amg^{\la}(\Omega)$ solves $\H(F) = \phi(F)$ on $\Om_{I}$ for some nonvanishing function $\phi:I \to \rea$, then, by Lemma \ref{homognondegenlemma}, each level set $\lc_{r}(F, \Omega_{I})$ is nondegenerate  and $dF$ and $|dF|^{2}_{g}$ do not vanish on $\lc_{r}(F, \Om_{I})$. In particular, the equiaffine normal $\nm^{i}$ is defined on $\Omega_{I}$. Since $\H(F)$ is constant on $\lc_{r}(F, \Omega_{I})$, $d\log\H(F) \wedge dF = 0$ on $\Omega_{I}$. Hence, by \eqref{hompol1}, $\H(F)^{-1}\H(F)^{i}$ is a multiple of $F^{i} = \la^{-1}\rad^{i}$. In \eqref{nm2} this shows that $\nm^{i}$ is a multiple of $\rad^{i}$, so that the connected components of $\lc_{r}(F, \Omega_{I})$ are affine spheres with affine normals parallel to $\rad^{i}$. Since by assumption $\H(F) = \phi(F)$ depends only on $r$, and not on the component, it follows that the equiaffine mean curvatures of different components of $\lc_{r}(F, \Om_{I})$ are the same. In both cases, from the constancy of $\H(F)$ on each $\lc_{r}(F, \Om_{I})$ and \eqref{hdl} it follows that the signatures of the distinct connected components of $\lc_{r}(F, \Om_{I})$ are the same modulo $4$. 

Suppose given $F \in \amg^{\la}(\Omega)$, an open interval $I \subset \rea \setminus\{0\}$, and a function $\phi$ defined on $I$ such that $\H(F) = \phi(F)$ for $x \in \Omega_{I}$. Since, by \eqref{hompol2}, $\H(F)$ has positive homogeneity $(n+1)\la$, there holds $\phi(e^{\la t}r) = e^{(n+1)\la t}\phi(r)$ for $r \in I$ and $t$ sufficently small. In particular, this shows that $\phi$ is continuous on $I$. Similarly, setting $h(t) = (e^{\la t} - 1) r $,
\begin{align}
\lim_{t \to 0} \tfrac{\phi(r + h(t)) - \phi(r)}{h(t)} 
= \lim_{t \to 0} \tfrac{\phi(e^{\la t}r ) - \phi(r)}{\la r t} = \lim_{t \to 0}\tfrac{(e^{(n+1)\la t} - 1)}{\la r}\phi(r)= \tfrac{(n+1) }{r }\phi(r),
\end{align}
so that $\phi$ is differentiable at $r$ and $\phi^{\prime}(r) =  \tfrac{(n+1)}{r}\phi(r)$. The general solution has the form $\phi(r) = \kc r^{n+1}$ for some $\kc  \neq 0$. 
\end{proof}

\begin{lemma}\label{twisteddetlemma}
If $F \in \cinf(\rea^{n+1})$ satisfies $F(x + t\rad) = F(x) + \la t$ for some $0 \neq \rad^{i} \in \rea^{n+1}$ and some $\la \in \rea$, then $\det(\hess F + cdF\tensor dF) = c\det(\hess F + dF \tensor dF)$ for all $c \in \cinf(\rea^{n+1})$.
\end{lemma}
\begin{proof}
By assumption $\rad^{i}F_{i} = \la$ and $\rad^{i}F_{ij} = 0$. Fix a unimodular basis $e_{1}, \dots, e_{n+1}$ such that $e_{n+1} = v$. Writing $\hess F + cdF\tensor dF$ as a matrix with respect to this basis, there results
\begin{align}
\begin{split}
\det(\hess F + cdF\tensor dF) & = \begin{vmatrix} F_{IJ} & 0 \\ 0 & c\la^{2}\end{vmatrix} 
= c\begin{vmatrix} F_{IJ} & 0 \\ 0 & \la^{2}\end{vmatrix} = c \det(\hess F + dF \tensor dF)
\end{split}
\end{align}
where the indices $I$ and $J$ run over $\{1, \dots, n\}$. 
\end{proof}

\begin{lemma}\label{improperlemma}
Let $\ste$ be an $(n+1)$-dimensional real vector space equipped with the standard equiaffine structure $(\nabla, \Psi)$, where $\Psi$ is given by the determinant. Let $\stw \subset \ste$ be a codimension one subspace, let $0 \neq \rad \in \ste$ be a vector transverse to $W$, and let $\mu \in \std$ be such that $\mu(\rad) = 1$ and $\ker \mu = \stw$. 
Equip $\stw$ with the induced affine structure and the parallel volume form $\mu = \imt(\rad)\Psi$ and define the operator $\H$ with respect to this induced equiaffine structure. Let $\pi:\ste \to \stw$ be the projection along the span $\lb \rad \ra$ of $\rad$. The following are equivalent.
\begin{enumerate}
\item\label{grp1} The \emph{graph of $f \in \cinf(\stw)$ along $\rad$}, $\{(w, t\rad) \in \stw \oplus \lb \rad \ra: t = f(w)\}$, is an improper affine sphere with affine normals parallel to $\rad$.
\item\label{grp2} There is $\kc \in \reat$ such that $\H(f) = \kc$ on $\stw$.
\item\label{grp3} There is $\kc \in \reat$ such that $F = \mu - f\circ \pi$ solves $\det(\hess F + dF \tensor dF) = (-1)^{n}\kc \Psi^{2}$ on $\ste$.
\item\label{grp4} There is $\kc \in \reat$ such that $G = \exp(\mu - f\circ \pi)$ solves $\H(G) = (-1)^{n}\kc G^{n+1}$ on $\ste$.
\item\label{grp5} There is $\kc \in \reat$ such that $\phi = \log(\mu - f\circ \pi)$ solves $\H(\phi) = (-1)^{n+1}\kc e^{-(n+2)\phi}$ on $\{x \in \ste: \mu(x) > f\circ \pi(x)\}$.
\end{enumerate}
\end{lemma}

\begin{proof}
Routine computations show the first two equalities of 
\begin{align}\label{improperequals}
\begin{split}
 (-1)^{n}\H(f)\circ \pi =\Psi^{-2}\tensor (\det(\hess F + dF \tensor dF)) = G^{-n-1}\H(G) = -e^{(n+2)\phi}\H( \phi),
\end{split}
\end{align}
while the third equality follows from Lemma \ref{twisteddetlemma}. From \eqref{improperequals} the equivalence of \eqref{grp2}-\eqref{grp5} is immediate. Since $G$ is by definition translationally homogeneous in the $\rad$ direction, the equivalence of \eqref{grp1} and \eqref{grp4} follows from Theorem \ref{ahtheorem}.
\end{proof}

\begin{remark}
After an equiaffine transformation, $\ste$, $\stw$, $\rad$, and the associated connections and volume forms can always be put in the following standard form. Let $\ste = \rea^{n+1}$ be equipped with its standard equiaffine structure $(\nabla, \Psi)$, where $\Psi = dx^{1}\wedge \dots \wedge dx^{n+1}$, and regard $\rea^{n}$ as the equiaffine subspace $\stw = \{x \in \ste: x^{n+1} = 0\}$ with the induced connection, also written $\nabla$, and the volume form $\mu =dx^{1}\wedge \dots \wedge dx^{n}$. Here $\mu = dx^{n+1}$, and the relation between $f$ and $F$ is $F(x_{1}, \dots, x_{n+1}) = x_{n+1} - f(x_{1}, \dots, x_{n})$.
\end{remark}

\begin{remark}
Examples of solutions of $\H(f) = \kc$ abound.
\begin{enumerate}
\item If $\kc < 0$, any function of the form $f(x_{1}, \dots, x_{n+1}) = (-\kc)^{1/2}x_{1}x_{n+1} + q(x_{1}) + \tfrac{1}{2}\sum_{i = 2}^{n}x_{i}^{2}$ with $q \in C^{2}(\rea)$ solves $\H(f) = \kc$ on all of $\rea^{n+1}$. 
This gives an infinitude of affinely inequivalent solutions to $\H(f) = \kc$ for $\kc < 0$, and so, by Lemma \ref{improperlemma}, an infinitude of affinely inequivalent entire solutions of $\H(G) = (-1)^{n}\kc G^{n+1}$ with $\kc < 0$. 

\item Let $\ste$ be an $n$-dimensional vector space. Let $\Phi:\ste \to \ste$ be a $C^{1}$ diffeomorphism and write $\Phi_{i}\,^{j} = \tfrac{\pr}{\pr x^{i}}\Phi(x)^{j}$. Define $f:\ste \times \std \to \rea$ by $f(x, y) = y_{p}\Phi(x)^{p}$. Let $\mu$ be the standard parallel volume form on $\ste$ and let $\Psi$ be the parallel volume form on $\ste \times \std$ determined by $\mu$ and the dual volume form on $\std$. A straightforward computation shows that $\H(f) = (-1)^{n}(\det \Phi_{i}\,^{j})^{2}$, where $\H$ is defined with respect to $\Psi$ and $\phi^{\ast}(\mu) = (\det \Phi_{i}\,^{j})\mu$. In particular, if $\Phi$ has constant Jacobian, $\det \Phi_{i}\,^{j} = \kc$, then $\H(f) = (-1)^{n}\kc^{2}$. In this case $\hess F$ has split signature.
\end{enumerate}

In section \ref{impropersection} it is shown how to construct many more solutions of the equations in Lemma \ref{improperlemma} from prehomogeneous actions of solvable Lie groups.
Some simple, but typical, examples obtained in this way are given by the translationally homogeneous (in the $z$-direction) functions
\begin{align}\label{expcayley}
&F(x, y,z) = e^{z - x^{2} - y^{2}},& &G(x, y, z) = e^{x^{3}/3 - xy + z},
\end{align}
that solve $\H(F) = 4F^{3}$ and $\H(G) = -G^{3}$, respectively. By Lemma \ref{improperlemma} their level sets are improper affine spheres. 
\end{remark}

\section{Left-symmetric algebras, affine actions, and completeness}\label{impropersection}
This section reviews basic material on LSAs in a form adequate for later applications. Part of the material presented is an amalgamation of background taken from  \cite{Goldman-Hirsch-orbits}, \cite{Helmstetter}, \cite{Kim-completeleftinvariant}, \cite{Kim-lsa},  \cite{Kim-developingmaps}, \cite{Segal-lsa}, \cite{Vinberg}. In particular, many results from J. Helmstetter's \cite{Helmstetter} are used. 

Although all LSAs considered in this section have finite dimension over a field of characteristic zero, these conditions are sometimes repeated so that the statements of lemmas and theorems are self-contained. The base field is usually supposed to be $\rea$, but this is stated explicitly where it is necessary, and many claims remain true over an arbitrary field $\fie$ of characteristic zero.

For a connected Lie group $G$ with Lie algebra $\g$, the isomorphism classes of the following structures are in pairwise bijection: 
\begin{itemize}
\item Left-invariant flat torsion-free affine connections on a Lie group $G$. 
\item Left-symmetric structures on the Lie algebra $\g$ of $G$ compatible with the Lie bracket on $\g$.
\item Étale affine representations of the Lie algebra $\g$. 
\end{itemize}
This section begins by sketching the constructions of the bijections. While this is explained elsewhere, for example in Proposition $1.1$ of \cite{Kim-completeleftinvariant}, \cite{Burde-etale}, or \cite{Kang-Bai}, it is recalled here to fix terminology and notation for later use.

A map $f:\A \to \B$ between affine spaces $\A$ and $\B$ is affine if there is a linear map $\lin(f):\ste \to \stw$, between the vector spaces $\ste$ and $\stw$ of translations of $\A$ and $\B$, such that $f(q) - f(p) = \lin(f)(q - p)$ for all $p, q \in \A$. With the operation of composition the bijective affine maps of $\A$ to itself form the Lie group $\Aff(\A)$ of affine automorphisms of $\A$, and $\lin:\Aff(\A) \to GL(\ste)$ is a surjective Lie group homomorphism. 
A one-parameter subgroup through the identity in $\Aff(\A)$ has the form $\Id_{A} + t\phi + O(t^{2})$ for some affine map $\phi:\A \to \ste$. Consequently, the Lie algebra $\aff(\A)$ of $\Aff(\A)$ is the vector space of affine maps from $\A$ to $\ste$ equipped with the bracket $[f, g] = \lin(f)\circ g - \lin(g) \circ f$. An \textit{affine representation} of a Lie algebra $(\g, [\dum, \dum])$ on the affine space $\A$ is a Lie algebra homomorphism $\rho:\g \to \aff(\A)$, that is, a linear map satisfying $\rho([a, b]) = \lin(\rho(a))\rho(b) - \lin(\rho(b))\rho(a)$ for all $a,b\in \g$. Any choice of fixed point (origin) $a_{0} \in \A$ determines a projection $\trans:\aff(\A) \to \ste$ onto the translational part defined by $\trans(f) = f(a_{0}) $, so that $f(a) = \lin(f)(a - a_{0}) + \trans(f)$ for any $a \in \A$. The Lie bracket on $\aff(\A)$ can be transported to $\eno(\ste) \oplus \ste$ via the resulting linear isomorphism $\lin \oplus \trans: \aff(\A) \to \eno(\ste) \oplus \ste$. 

Let $\rho:\g \to \aff(\A)$ be an affine representation of the Lie algebra $(\g, [\dum, \dum])$ on the affine space $\A$ with translations $\ste$. The representation $\rho$ is \textit{faithful} if it is injective, it is \textit{prehomogeneous at $x_{0}$} if there exists $x_{0} \in \A$ such that the map $\g \to \ste$ defined by $a \to \rho(a)x_{0}$ is a linear surjection, and it is \textit{étale at $x_{0}$} if there exists $x_{0} \in \A$ such that the map $\g \to \ste$ defined by $a \to \rho(a)x_{0}$ is a linear isomorphism. (If it is not important what $x_{0}$ is, $\rho$ is simply said to be \textit{prehomogeneous} or \textit{étale}.) An affine representation is étale if and only if it is faithful and prehomogeneous.

Let $(\alg, \mlt)$ be a finite-dimensional LSA. 
Let $\lin:\aff(\alg) \to \eno(\alg)$ be the projection onto the linear part and let $\trans:\aff(\alg) \to \alg$ be the projection, $\trans(\phi) = \phi(0)$, corresponding to the origin $0 \in \alg$. That $\mlt$ be left-symmetric is equivalent to the requirement that the map $\phi = L \oplus I:\alg \to \eno(\alg) \oplus \alg \simeq \aff(\alg)$ be an affine Lie algebra representation, where $I$ is the identity endomorphism of $\alg$ and the isomorphism $ \eno(\alg) \oplus \alg \simeq \aff(\alg)$ is that inverse to $\lin \oplus \trans$. Since $\trans \circ \phi = I$, that is $\phi(a)0 = a$, $\phi$ is étale. The map $\phi$ is the \textit{canonical affine representation} of the LSA $(\alg, \mlt)$.

In the other direction, given an étale affine representation $\rho: \g \to \aff(\A)$, for $a, b \in \g$ there exists a unique $a \mlt b \in \g$ such that $\lin(\rho(a))\rho(b)x_{0} = \rho(a \mlt b)x_{0}$. From the fact that $\rho$ is a Lie algebra representation it follows that $\mlt$ is a left-symmetric multiplication on $\g$ with underlying Lie bracket $[\dum, \dum]$. The special case where $\g = \alg = \A$ is a vector space with origin $x_{0} = 0$  and $\trans \circ \rho = I$ yields the affine representation $\rho:\alg \to \aff(\alg)$ of the Lie algebra $(\alg, [\dum, \dum])$ and the compatible left-symmetric multiplication $a\mlt b = \rho(a\mlt b)0 = \lin(\rho(a))b = \rho(a)(0 + b) - \rho(a)0 = \rho(a)b - a$ with left multiplication operator $L = \ell \circ \rho$.

The étale affine representation $\rho:\alg \to \aff(\alg)$ determined by a left-symmetric multiplication $\mlt$ extends to a faithful linear representation $\hat{\rho}:\alg \to \gl(\alg \oplus \rea)$ defined by $\hat{\rho}(a)(b, t) = (\lin(\rho(a))b + t\trans(\rho(a)), 0)$. Consequently, an $n$-dimensional Lie algebra that admits no faithful linear representation of dimension $n+1$ admits no compatible left-symmetric multiplication. The first such example, with $n = 11$, was constructed by Y. Benoist in \cite{Benoist-nilvariete}.

Given an LSA $(\alg, \mlt)$, let $G$ be the simply-connected Lie group with Lie algebra $(\alg, [\dum, \dum])$. For $a \in \alg$, the vector field $\livf^{a}_{g} = \tfrac{d}{dt}\big|_{t = 0}g\cdot \exp_{G}(ta)$ generated on $G$ by right multiplication by $\exp_{G}(ta)$ is left-invariant. The relations defining an LSA mean that the left-invariant connection $\nabla$ on $G$ defined by $\nabla_{\livf^{a}}\livf^{b} = \livf^{a\mlt b}$ is torsion-free and flat. Conversely, given a flat torsion-free left-invariant connection $\nabla$ on a Lie group $G$, defining $a\mlt b$ by $\nabla_{\livf^{a}}\livf^{b} = \livf^{a\mlt b}$ for $a$ and $b$ in the Lie algebra $\g$ of $G$, makes $\g$ into an LSA. 

Suppose $G$ is simply-connected with identity element $e$ and let $\dev:G \to \alg$ be the developing map such that $\dev(e) = 0$.  
Let $R_{g}$ be the operator on $G$ of right multiplication by $g$. By the left invariance of $\nabla$ there exists a unique $\hol(g) \in \Aff(\alg)$ such that $\dev \circ R_{g} = \hol(g)\circ \dev$. The map $\hol:G \to \Aff(\alg)$ is a group homomorphism. Although the homomorphism $\hol$ depends on the choice of origin $0$, another such choice leads to a homomorphism conjugate to $\hol$ by an element of $\Aff(\alg)$.

Since $\dev$ is an open map, the image $\dev(G)$ is an open subset of $\alg$ on which $\hol(G)$ acts transitively with discrete isotropy group. This is the canonical affine representation of $G$ on $\alg$. Since the isotropy group of $\hol(G)$ is discrete, the corresponding affine representation $T\hol(e):\g \to \aff(\alg)$ is étale. Lemma \ref{infhollemma} shows that $T\hol(e)$ is the affine representation determined by the original LSA and shows how the affine representation $\hol$ intertwines the exponential maps on $G$ and $\Aff(\alg)$. It is equivalent to results in \cite{Kim-developingmaps}.
\begin{lemma}[cf. \cite{Kim-developingmaps}]\label{infhollemma}
Let $G$ be the simply-connected Lie group with Lie algebra the underlying Lie algebra of an LSA $(\alg, \mlt)$, and let $\nabla$ be the flat torsion-free left-invariant affine connection determined on $G$ by the multiplication $\mlt$. The differential $T\hol(e)$ at the identity $e \in G$ of the affine representation $\hol:G \to \Aff(\alg)$ corresponding to the developing map $\dev:G \to \alg$ such that $\dev(e) = 0$ is equal to the canonical affine representation $\phi = L \oplus I:\alg \to \aff(\alg)$, where $L$ is the operator of left multiplication in $(\alg, \mlt)$. 
Precisely, for $a, b \in \alg$,
\begin{align}\label{holexp0}
&\tfrac{d}{dt}\big|_{t = 0}\hol(\exp_{G}(ta))\cdot b  = T\hol(e)a \cdot b = \phi(a)b = a\mlt b + a = (I + R(b))a,\\
\label{holexp}
&\hol(\exp_{G}(a)) = \exp_{\Aff(\alg)}\circ \phi(a) = (e^{L(a)}, E(a)),
\end{align}
where $E(a) = \sum_{k \geq 1}\tfrac{1}{k!}L(a)^{k-1}a$. The differential $T\dev_{g}:T_{g}G \to \alg$ of the developing map at $g \in  G$ equals $(I + R(\dev(g)))\circ TR_{g^{-1}}$. For $a \in \alg$ regard the map $x \to (I + R(x))a = \rvf^{a}_{x}$ as a vector field on $\alg$. The flow of $\rvf^{a}$ is $\phi^{a}_{t}(x) = \hol(\exp_{G}(ta))\cdot x = e^{tL(a)}x + E(ta)$.
\end{lemma}
Note that, by \eqref{holexp}, 
\begin{align}
\dev(\exp_{G}(a)) = \hol(\exp_{G}(a))0 = (e^{L(a)}, E(a))0 = E(a). 
\end{align}
The identity \eqref{holexp} appears in an equivalent form in section II.$1$. of \cite{Vinberg}; see also \cite{Choi-domain}. 

\begin{proof}
It follows from the definition of the exponential map of a Lie group that $\hol \circ \exp_{G} = \exp_{\Aff(\alg)}\circ T\hol(e)$, where $T\hol(e):\alg \to \aff(\alg)$ is the differential of $\hol$ at $e$. Precisely, for $a \in \g$, $\si(t) = \exp_{\Aff(\alg)}(tT\hol(e)a)$ is by definition the unique one-parameter subgroup of $\Aff(\alg)$ such that $\si(0) = e_{\Aff(\alg)}$ and $\si^{\prime}(0) = T\hol(e)a$; since $\hol\circ \exp_{G}(ta)$ is a one-parameter subgroup of $\Aff(\alg)$ satisfying the same initial conditions, it equals $\si(t)$.

By the definition of $\dev$, the image $\dev(\ga(t))$ of the $\nabla$-geodesic $\ga(t)$ such that $\ga(0) = e$ and $\dot{\ga}(0) = a$ is a straight line in $\alg$ tangent to $a$ at $0 \in \alg$, so $T\dev(e)a = \tfrac{d}{dt}\big|_{t = 0}\dev(\ga(t)) = a$. Since $\ga(t)$ is tangent at $e$ to the curve $\exp_{G}(ta)$, 
\begin{align}\label{atdevea}
\begin{split}
a &= \tfrac{d}{dt}\big|_{t = 0}\dev(\ga(t)) = T\dev(e)a = \tfrac{d}{dt}\big|_{t = 0}\dev(\exp_{G}(ta)) \\
&=  \tfrac{d}{dt}\big|_{t = 0}\hol(\exp_{G}(ta))\dev(e) = T\hol(e)a \cdot \dev(e)= \trans(T\hol(e)a),
\end{split}
\end{align}
so $\trans \circ T\hol(e) = I$. By the definition of $\dev$, $\nabla$ is the pullback via $\dev$ of the flat connection $\pr$ on $\alg$ determined by the flat affine structure on $\alg$ whose geodesics are the lines in $\alg$. 
Let $Y^{a}_{\dev(g)} = \tfrac{d}{dt}\big|_{t = 0}\dev(g\cdot \exp_{G}(ta)) = T\dev(g)(\livf^{a}_{g})$. The integral curve through $e$ of $\livf^{a}$ is $\exp_{G}(ta)$ and the integral curve through $\dev(e)$ of $Y^{a}$ is $\dev(\exp_{G}(ta))$. Combining the preceding observations with \eqref{holexp0} and \eqref{atdevea} yields
\begin{align}
\begin{split}
a\mlt b& = T\dev(e)(\livf^{a\mlt b}_{e}) = T\dev(e)((\nabla_{\livf^{a}}\livf^{b})_{e}) = (\pr_{Y^{a}}Y^{b})_{e}\\
& = \tfrac{d}{dt}\big|_{t = 0}Y^{b}_{\dev(\exp_{G}(at))} = \tfrac{d}{dt}\big|_{t = 0}\tfrac{d}{ds}\big|_{s = 0} \dev(\exp_{G}(ta)\exp_{G}(sb))\\
& = \tfrac{d}{dt}\big|_{t = 0}\tfrac{d}{ds}\big|_{s = 0} \hol(\exp_{G}(ta))\dev(\exp_{G}(sb))\\
& = \tfrac{d}{dt}\big|_{t = 0}\tfrac{d}{ds}\big|_{s = 0}\exp_{\Aff(\alg)}(T\hol(e)(ta))\dev(\exp_{G}(sb))\\
& = \tfrac{d}{dt}\big|_{t = 0}\lin(\exp_{\Aff(\alg)}(tT\hol(e)a))b = \lin(T\hol(e)a) \cdot b.
\end{split}
\end{align}
This shows $\lin (T\hol(e)a) = L(a)$, and so $T\hol(e) = \phi$. The exponential map $\exp_{G}$ of $G$ can be described explicitly for $a \in \alg$ near $0$ by representing $\aff(\alg)$ as a subgroup of $\eno(\alg \oplus \rea)$ and exponentiating the image $\phi(a)$. There results \eqref{holexp}. Finally, for $a \in \alg$ and $g \in G$,
\begin{align}
\begin{split}
T\dev_{g}\circ TR_{g}(a) & = \tfrac{d}{dt}\big|_{t = 0}\dev(\exp_{G}(ta)g) =  \tfrac{d}{dt}\big|_{t = 0}\hol(\exp_{G}(ta))\dev(g) \\&= T\hol(e)a \cdot \dev(g) = (I + R(\dev(g))a = \rvf^{a}_{\dev(g)},
\end{split}
\end{align}
by \eqref{atdevea}. There remains to show that the flow of $\rvf^{a}$ is $\phi^{a}_{t}(x) = e^{tL(a)}x + E(ta)$. Because
\begin{align}
\tfrac{d}{dt}E(ta) = \tfrac{d}{dt}\sum_{k\geq 1}\tfrac{t^{k}}{k!}L(a)^{k-1}a = \sum_{k \geq 1}\tfrac{t^{k-1}}{(k-1)!}L(a)^{k-1}a  = e^{L(a)}a,
\end{align}
there holds
\begin{align}
\tfrac{d}{dt}\phi^{a}_{t}(x) = \tfrac{d}{dt}(e^{tL(a)}x + E(ta)) = L(a)e^{tL(a)}x + e^{tL(a)}a.
\end{align}
On the other hand,
\begin{align}
\rvf^{a}_{\phi^{a}_{t}(x)} = (I + R(\phi^{a}_{t}(x))a = a + L(a)e^{tL(a)}x + L(a)E(ta)  = L(a)e^{tL(a)}x + e^{tL(a)}a.
\end{align}
This shows the final claim.
\end{proof}
By \eqref{holexp0}, for $a, x \in \alg$, the differential at $e \in G$ of the map $g \to \hol(g)\cdot x$ is $I + R(x)$, and the tangent space $T_{x}\hol(G)\cdot x$ at $x$ of the orbit $\hol(G)\cdot x$ is the image of $I + R(x)$, which is spanned by the vector fields $\rvf^{a}$.

Following Helmstetter in \cite{Helmstetter}, define the \textit{characteristic polynomial} of the LSA by $P(x) = \det(I + R(x))$. Since $R(x)$ is linear in $x$, $\deg P \leq \dim \alg$.

\begin{lemma}[J. Helmstetter \cite{Helmstetter}] \label{charpolylemma}
Let $G$ be the simply-connected Lie group with Lie algebra the underlying Lie algebra of the finite-dimensional LSA $(\alg, \mlt)$. The characteristic polynomial $P(x) = \det(I + R(x))$ of $(\alg, \mlt)$ satisfies 
\begin{align}\label{holgp}
(\hol(g)\cdot P)(x) = P(\hol(g^{-1})x) = \chi(g^{-1})P(x) 
\end{align}
for $g \in G$ and $x \in \alg$, where $\chi:G \to \reap$ is the group character satisfying $\chi\circ \exp_{G}(a) = e^{\tr R(a)}$ for $a \in \alg$.
\end{lemma}
\begin{proof}
Proposition $3.1$ of H. Kim's \cite{Kim-developingmaps} gives a proof of \eqref{holgp} in terms of the developing map. The proof given here, based on \eqref{holexp0}, is similar to Kim's. By \eqref{holexp0}, for $g \in G$ and $a, x \in \alg$,
\begin{align}
\begin{split}
(I + R(\hol(g)x))a & = \tfrac{d}{dt}\big|_{t = 0}\hol(\exp_{G}(ta))\hol(g)x =  \tfrac{d}{dt}\big|_{t = 0}\hol(g\exp_{G}(t\Ad(g^{-1})a))x.
\end{split}
\end{align}
Suppose that $g = \exp_{G}(b)$. Then, using \eqref{holexp0} and \eqref{holexp},
\begin{align}
\begin{split}
(I &+ R(\hol( \exp_{G}(b))x))a   =  \tfrac{d}{dt}\big|_{t = 0}\hol(\exp_{g}(b)\exp_{G}(t\Ad(\exp_{G}(-b))a)x\\
& =  \tfrac{d}{dt}\big|_{t = 0}\left(e^{L(b)}\hol(\exp_{G}(t\Ad(\exp_{G}(-b))a))x  + E(b)\right)
 = e^{L(b)}(I + R(x))\Ad(\exp_{G}(-b))a.
\end{split}
\end{align}
Consequently $P(\hol(\exp_{G}(b))x) = e^{(\tr L(b) - \tr \ad(b))}P(x) = e^{\tr R(b)}P(x)$. Since $G$ is generated by any given open neighborhood of the identity, in particular by $\exp_{G}(\alg)$, corresponding to the infinitesimal character $\tr R$ there is a unique group character $\chi:G \to \reat$ satisfying $\chi\circ \exp_{G}(b) = e^{\tr R(b)}$ for all $b \in \alg$. Then $P(\hol(g)x) = \chi(g)P(x)$ holds for $g \in \exp_{G}(\alg)$. Since  $\exp_{G}(\alg)$ generates $G$, this implies that equality holds for all $g \in G$. 
\end{proof}

A consequence of Lemma \ref{charpolylemma} is the result of \cite{Goldman-Hirsch-orbits} (see section $1A.8$) that the orbit $\hol(G) 0 = \hol(G)\dev(e) = \dev(G)$ is the connected component containing $0$ of the complement of the zero set of $P$ (see also Proposition $3.2$ of \cite{Kim-developingmaps}).
\begin{corollary}[\cite{Goldman-Hirsch-orbits}, \cite{Kim-developingmaps}, \cite{Helmstetter}]\label{orbitcorollary}
Let $G$ be the simply-connected Lie group with Lie algebra the underlying Lie algebra of the finite-dimensional LSA $(\alg, \mlt)$. The orbit $\hol(G) 0 = \hol(G)\dev(e) = \dev(G)$ is the connected component $\Om$ containing $0$ of the complement of the zero set $\{x \in \alg: P(x) = 0 \}$ of the characteristic polynomial $P$ and $\dev:G \to \Om$ is a diffeomorphism.
\end{corollary}
\begin{proof}
Taking $x = 0$ in \eqref{holgp} shows that
\begin{align}\label{pdev}
  P(\dev(g))= P(\hol(g)0) =\chi(g),
\end{align}
for $g \in G$, so that $P$ is nonvanishing on $\dev(G)$. 
By \eqref{holgp}, the action of $\hol(G)$ preserves $\{x \in \alg: P(x) = 0 \}$. Since $G$ is path-connected, $\hol(g)\Om = \Om$ for all $g \in G$, so $\hol(G)\Om = \Om$. Consequently $\dev(G) = \hol(G)0 \subset \hol(G)\Om = \Om$. 
By Lemma \ref{infhollemma}, the differential of $\dev$ at $g \in G$ equals $I + R(\dev(g))$, and so is a linear isomorphism, since $P(\dev(g)) \neq 0$. Hence $\dev:G \to \Om$ is a local diffeomorphism, so an open map, and $\dev(G)$ is connected open subset of $\Om$, hence equals $\Om$. Since an open map is proper, $\dev$ is a proper local diffeomorphism and so a covering map. Since $G$ is simply-connected, $\dev$ must be a diffeomorphism.
\end{proof}

Let $(\alg, \mlt)$ be a finite-dimensional LSA. The \textit{trace form} $\tau$ of the LSA $(\alg, \mlt)$ is the symmetric bilinear form defined by 
\begin{align}\label{traceformdefined}
\tau(x, y) = \tr R(x)R(y) = \tr R(x\mlt y)
\end{align}
(the second equality follows from \eqref{rlsa}).

Differentiating $\hol(\exp_{G}(-ta))\cdot P$ using \eqref{holexp} and \eqref{holgp} yields
\begin{align}\label{dptra}
\begin{split}
dP_{x}((I + R(x))a) & = \tfrac{d}{dt}\big|_{t = 0}P(\hol(\exp_{G}(ta))x) = \tfrac{d}{dt}\big|_{t = 0}e^{t\tr R(a)}P(x) = P(x)\tr R(a).
\end{split}
\end{align}
In particular, $dP_{0} = d\log P_{0} = \tr R$. 

Since $R(x)$ is linear in $x$, differentiating $R(x)$ in the direction of $b \in \alg$ yields $R(b)$. Writing \eqref{dptra} as $d\log P_{x}((I + R(x))a) = \tr R(a)$ and taking the derivative in the $b$ direction yields
\begin{align}\label{hesslogp}
(\hess \log P)_{x}((I + R(x))a,  b) = - d\log P_{x}(a \mlt b) = -\tr R((I + R(x))^{-1}(a\mlt b)).
\end{align}
In particular, $\tau(a, b) = -(\hess \log P)_{0}(a,  b)$, so that nondegeneracy of the trace form is the algebraic condition corresponding to nondegeneracy of the Hessian of $\log P$, or, what is essentially the same, the nondegeneracy of the level sets of $P$. Since the eventual goal here is to find characteristic polynomials $P$ solving $\H(e^{P}) = \kc e^{nP}$ with $\kc \neq 0$, the focus is LSAs for which $\tau$ is nondegenerate. Differentiating \eqref{hesslogp} yields
\begin{align}\label{cubiclogp}
(\pr^{2}d\log P)_{x}((I + R(x))a, b, c) =  - (\hess\log P)_{x}(a\mlt b, c) -(\hess \log P)_{x}(b, a \mlt c),
\end{align}
which combined with \eqref{hesslogp} shows that 
\begin{align}\label{taucubic}
(\pr^{2}d\log P)_{0}(a, b, c) = \tau(a\mlt b, c) + \tau(b, a\mlt c).
\end{align} 
More generally there holds the following identity due to Helmstetter.
\begin{lemma}[J. Helmstetter; $(5)$ of \cite{Helmstetter}]
The characteristic polynomial $P$ of an LSA $(\alg, \mlt)$ satisfies
\begin{align}\label{nablaklogp2}
(\pr^{k}d\log &P)_{0}(a_{0}, a_{1}, \dots, a_{k}) = (-1)^{k}\sum_{\si \in S_{k}}\tr\left(R(a_{\si(1)})\dots R(a_{\si(k)})R(a_{0}) \right),
\end{align}
for all $a_{0}, \dots, a_{k} \in \alg$.
\end{lemma}
\begin{proof}
The identity \eqref{nablaklogp2} is proved by induction on $k$.
The base case $k = 1$ is true by \eqref{hesslogp}. Repeated differentiation shows that, for $k \geq 1$, 
\begin{align}\label{nablakdlogp}
\begin{split}
(\pr^{k}d\log &P)_{x}((I + R(x))a_{0}, a_{1}, \dots, a_{k}) = -\sum_{i = 1}^{k}(\pr^{k-1}d\log P)_{x}(a_{1}, \dots, a_{i-1}, a_{0}\mlt a_{i}, a_{i+1}, \dots, a_{k}).
\end{split}
\end{align}
The inductive step is proved using \eqref{nablakdlogp}, \eqref{rlsa}, and that $\sum_{\si \in S_{k+1}}[A_{\si(1)}\dots A_{\si(k)}, A_{\si(k+1)}]  = 0$
for any $A_{1}, \dots, A_{k+1} \in \eno(\alg)$.
\end{proof}

An LSA $(\alg, \mlt)$ is \textit{right nil} (\textit{left nil}) if $R(a)$ (respectively $L(a)$) is nilpotent for every $a \in \alg$. 
A finite-dimensional LSA $(\alg, \mlt)$ is \textit{complete} if its characteristic polynomial is nowhere vanishing. Equivalently, $I + R(a)$ is invertible for all $a \in \alg$. The LSA is \textit{incomplete} otherwise. If $R(a)$ is nilpotent, then $I + R(a)$ is invertible, so $P(a) \neq 0$. Hence a right nil LSA is complete. The contrary claim, that a complete LSA is right nil, is due to D. Segal. This and other equivalent characterizations of completeness are explained in Lemma \ref{completenesslemma}.

\begin{lemma}\label{completenesslemma}
Let $(\alg, \mlt)$ be a finite-dimensional real LSA $(\alg, \mlt)$ with characteristic polynomial $P$. Let $\nabla$ be the flat left-invariant affine connection determined by $\mlt$ on the simply-connected Lie group $G$ with Lie algebra $(\alg, [\dum, \dum])$. The following are equivalent.
\begin{enumerate}
\item\label{cmp1} $\nabla$ is complete.
\item\label{cmp5} $R(a)$ is nilpotent for all $a \in \alg$; that is, $(\alg, \mlt)$ is right nil.
\item\label{cmp6} $\tr R(a) = 0$ for all $a \in \alg$. 
\item\label{cmp7} The trace form $\tau$ defined in \eqref{traceformdefined} is identically zero.
\item\label{cmp8} $(\pr^{N}d\log P)_{0} = 0$ for some $N \geq 1$.
\item\label{cmp3} $P$ is constant and equal to $1$.
\item\label{cmp4} $P$ has a critical point at $0 \in \alg$.
\item\label{cmp2} $P$ is nowhere vanishing on $\alg$; that is, $(\alg, \mlt)$ is complete.
\end{enumerate}
\end{lemma}
\begin{proof}
\eqref{cmp1}$\iff$\eqref{cmp2}.  By Corollary \ref{orbitcorollary}, $\dev$ is a diffeomorphism onto the connected component containing $0$ of the complement in $\alg$ of the zero set of $P$. The affine connection $\nabla$ is complete if and only if $\dev$ is a diffeomorphism from $G$ onto $\alg$. Hence $\nabla$ is complete if and only if $P$ is nowhere vanishing.

\eqref{cmp5}$\implies$\eqref{cmp2} and \eqref{cmp6}.  If $R(a)$ is nilpotent, then $I + R(a)$ is invertible and $\tr R(a) = 0$.

\eqref{cmp2}$\implies$\eqref{cmp6}. It is equivalent to show that $I + R(a)$ is invertible for all $a \in \alg$ if and only if $(\alg, \mlt)$ is right nil. This is proved as Theorem $1$ of D. Segal's \cite{Segal-lsa} (an alternative algebraic proof is given in \cite{Elduque-Myung-lsas}). The proof uses facts about dominant morphisms and, in particular, the fact that the orbits of a morphic action of a unipotent algebraic group are closed. Segal's proof shows the stronger statement that for a base field of characteristic zero, an LSA is complete if and only if the LSA obtained by extension of scalars to the algebraic closure of the base field is also complete. In the present context, this amounts to showing that the characteristic polynomial of the complexified LSA is nowhere vanishing; since the nonvanishing of the characteristic polynomial implies that no $R(b)$ has a nonzero eigenvalue in the base field, this suffices to show that all $R(b)$ are trace free. 

\eqref{cmp3}$\iff$\eqref{cmp5}. If $P$ is equal to $1$ then \eqref{nablaklogp2} with $a_{0} = a_{1} = \dots = a_{k} = a$ implies $\tr R(a)^{k+1} = 0$ for all $k \geq 0$. Hence $R(a)$ is nilpotent for all $a \in \alg$.
If $R(a)$ is nilpotent for all $a \in \alg$, then $I + R(a)$ is unipotent for all $a \in \alg$, so $P(a) = \det (I + R(a)) = 1$.

\eqref{cmp3}$\iff$\eqref{cmp6}. If $P$ is nowhere vanishing then $I+R(x)$ is invertible for all $x \in \alg$. Associate with $a \in \alg$ the vector field $\rvf^{a}_{x} = (I + R(x))a$. If $P$ is nonvanishing, and $\{a_{1}, \dots, a_{n}\}$ is a basis of $\alg$, then $\rvf^{a_{1}}_{x}, \dots, \rvf^{a_{n}}_{x}$ are linearly independent for all $x \in \alg$. By \eqref{dptra}, $d\log P_{x}(\rvf^{a_{i}}_{x}) = \tr R(a_{i})$. If $P$ is constant, the constant must be $1$ because $P(0) = 1$, and so $\tr R(a) = d\log P_{x}(\rvf^{a}_{x})) = 0$ for all $a \in \alg$. On the other hand, if $\tr R = 0$, then $d\log P_{x}(\rvf^{a_{i}}_{x}) = 0$ for $1 \leq i \leq n$, so $\log P$ is constant.

\eqref{cmp4}$\iff$\eqref{cmp6}. Since $P(0) = 1$ and $dP_{0} = d\log P_{0} = \tr R$, $P$ has a critical point at $0 \in \alg$ if and only if $\tr R = 0$.
	
\eqref{cmp6}$\iff$\eqref{cmp7}. If there holds \eqref{cmp6}, by \eqref{traceformdefined}, $\tau(x, y) = \tr R(x \mlt y) = 0$ for all $x, y \in \alg$. Suppose there holds \eqref{cmp7}. Since $\tr R$ is linear there is $r \in \alg^{\ast}$ such that $\tr R(x) = r(x)$. Since $P(0) = 1$ and $P$ is a polynomial, $\log P$ is real analytic in a neighborhood of $0$. By \eqref{nablakdlogp} with $x = 0$, if there holds $(\pr^{k-1}d\log P)_{0} = 0$ then $(\pr^{k}d\log P)_{0} = 0$. By \eqref{hesslogp}, that $\tau = 0$ means $(\pr d\log P)_{0} = 0$, so $(\pr^{k}d\log P)_{0} = 0$ for $k \geq 2$. Hence for $x$ sufficiently near $0$, $\log P(x) = r(x)$, or $P(x) = e^{r(x)}$. Differentiating this equality $k$ times yields $(\pr^{k}P)_{0} = r^{\tensor k}$, the $k$fold tensor product of $r$ with itself. Since $P$ is a polynomial, this vanishes for $k > \deg P$, and hence $r$ must be $0$.

\eqref{cmp8}$\iff$\eqref{cmp3}. That \eqref{cmp3}$\implies$\eqref{cmp8} is immediate, so suppose there holds \eqref{cmp8}. Since $P(0) = 1$ and $P$ is a polynomial, $\log P$ is real analytic in a neighborhood $U$ of $0$. By \eqref{nablakdlogp}, that $(\pr^{N}d\log P)_{0} = 0$ means that $(\pr^{k}d\log P)_{0} =0$ for all $k \geq N$. Since $\log P$ is real analytic, this implies it equals some degree $N-1$ polynomial $Q$ in $U$. That $P$ and $e^{P}$ both be polynomials means that both are constant, so $P$ is constant, equal to $1$.
\end{proof}

\begin{remark}
It may be tempting to believe that \eqref{cmp2} implies \eqref{cmp4} for any real polynomial, that is, that a nonvanishing real polynomial must have a critical point, but this is false. For example the polynomial $Q(x, y) = (1 + x +x^{2}y)^{2} + x^{2}$ takes on all positive real values and has nonvanishing gradient (this example is due to \cite{fedja}). This means that this $Q$ cannot be the characteristic polynomial of a real finite-dimensional LSA. This suggests the problem: \textit{Characterize those polynomials that occur as the characteristic polynomials of some real finite-dimensional LSA}.
\end{remark}

If $(\alg, \mlt)$ is incomplete, there is $a \in \alg$ such that $\tr R(a) \neq 0$ and so $dP_{0}$ cannot be identically $0$ and $\ker \tr R$ has codimension one.

\begin{lemma}\label{cpafflemma}
If $\Psi:(\alg, \mlt) \to (\balg, \circ)$ is an isomorphism of LSAs, the characteristic polynomials $P_{\mlt}$ and $P_{\circ}$ satisfy $P_{\circ} \circ \Psi = P_{\mlt}$.
\end{lemma}
\begin{proof}
The right multiplication operators satisfy $R_{\circ}(\Psi(x)) \circ \Psi = \Psi \circ R_{\mlt}(x)$ for all $x \in \alg$. The claim follows from the definition of the characteristic polynomial.
\end{proof}

\begin{remark}
By Lemma \ref{cpafflemma}, isomorphic LSAs have affinely equivalent characteristic polynomials. The converse is trivially false, because by Lemma \ref{completenesslemma} any two complete LSAs have affinely equivalent characteristic polynomials. Example \ref{negeigexample} exhibits nonisomorphic LSAs having the same nondegenerate characteristic polynomial.
\end{remark}

\section{Hessian LSAs, Koszul forms, idempotents, and the trace-form}\label{hessiansection}
A symmetric bilinear form $h$ on an LSA $(\alg, \mlt)$ is \textit{Hessian} if its \textit{cubic form} $C_{h}$ defined by $C_{h}(x, y, z) = h(x\mlt y, z) + h(y, x\mlt z)$ is completely symmetric, meaning
\begin{align}\label{hessianmetric}
0 = C_{h}(x, y, z) - C_{h}(y, x, z) = h([x, y], z) - h(x, y\mlt z) + h(y, x \mlt z),
\end{align} 
for all $x, y, z\in \alg$.  A \textit{metric} is a nondegenerate symmetric bilinear form $h$. A \textit{Hessian LSA} is a triple $(\alg, \mlt, h)$ comprising an LSA $(\alg, \mlt)$ and a Hessian metric $h$. The terminology \textit{Hessian} follows \cite{Shima-homogeneoushessian}. It follows from \eqref{rlsa} that the cubic form $C = C_{\tau}$ of the trace-form $\tau$ of an LSA is completely symmetric, so $\tau$ is Hessian. Hence, if $\tau$ is nondegenerate then it is a Hessian metric. Moreover, the identity \eqref{taucubic} shows that $C(a, b, c) = (\pr^{2}d\log P)_{0}(a, b, c)$. 

A \textit{Koszul form} on the LSA $(\alg,\mlt)$ is an element $\la \in \alg^{\ast}$ such that $h(x, y) = \la(x \mlt y)$ is a metric. 
The assumed symmetry of $h$ implies $[\alg, \alg] \subset \ker \la$, while the left symmetry of $\mlt$ implies that $h$ is a Hessian metric.
If an LSA admits a Koszul form, then its left regular representation $L$ is faithful, for, if $L(x) = 0$, then $h(x, y) = \la(L(x)y) = 0$ for all $y \in \alg$ implies $x = 0$. 

An algebra $(\alg, \mlt)$ is \textit{perfect} if $\alg \mlt \alg = \alg$. 
An algebra $(\alg, \mlt)$ is \textit{simple} if its square $\alg^{2}$ is not zero, and $\alg$ contains no nontrivial two-sided ideal. A nontrivial ideal in an LSA $(\alg, \mlt)$ is a nontrivial ideal in the underlying Lie algebra $(\alg, [\dum, \dum])$. Consequently, if the Lie algebra underlying an LSA is simple then the LSA is simple.

Since a simple algebra is perfect, if $(\alg, \mlt)$ is a simple LSA then two Koszul forms inducing the same metric must be equal.

Let $(\alg, \mlt)$ be a finite-dimensional LSA. If the trace form $\tau$ defined in \eqref{traceformdefined} is nondegenerate, then $\tr R$ is a Koszul form. 

By \eqref{cmp7} of Lemma \ref{completenesslemma}, if $(\alg, \mlt)$ is complete, then $\tau$ vanishes identically, so an LSA with nontrivial trace form $\tau$ is necessarily incomplete.

Given a Koszul form $\la$ on $(\alg, \mlt)$, the unique element $u \in \alg$ such that $\la(a) = \la(a \mlt u)$ for all $a \in \alg$ is idempotent. By \eqref{hessianmetric}, $h(R(u)x, y) - h(x, R(u)y) = h(u, [x,y]) = \la([x, y]) = 0$. Consequently, $h(u\mlt u, x) = h(u, x\mlt u) = \la(x\mlt u) = h(x, u)$ for all $x \in \alg$, so $u\mlt u = u$. The element $u \in\alg$ is the \textit{idempotent associated with $\la$}. 

Write $\lb x_{1}, \dots, x_{k} \ra$ for the span of $x_{1}, \dots, x_{k} \in \alg$ and  $V^{\perp}$ for the orthogonal complement of the subspace $V \subset \alg$ with respect to a given Hessian metric $h$.
For $A \in \eno(\alg)$, let $A^{\ast} \in \eno(\alg)$ be the \textit{$h$-adjoint} of $A$ defined by $h(Ax, y) = h(x, A^{\ast}y)$.

\begin{lemma}\label{principalidempotentlemma}
For an $n$-dimensional LSA $(\alg, \mlt)$ with a Koszul form $\la$ with associated Hessian metric $h$, the idempotent $u$ associated with $\la$ has the following properties.
\begin{enumerate}
\item\label{lurupre} $L(u)$ and $R(u)$ preserve the Lie ideal $\ker \la = \lb u \ra^{\perp}$.
\item\label{lladj} $L(u) + L(u)^{\ast} = R(u) + I$.
\item\label{rsa} $R(u)$ is $h$-self-adjoint, $R(u) = R(u)^{\ast}$.
\item\label{trru} $R(u) - R(u)^{2}$ is nilpotent, and there is an integer $k$ such that $1 \leq \tr R(u) = \dim \alg_{1} = k \leq n$ and $(R(u) - R(u)^{2})^{\max\{k, n-k\}} = 0$.
\item \label{a0l} $\sum_{k \geq 1} \ker R(u)^{k} \subset \ker \la$.
\item\label{rclsa6b} No nontrivial left ideal of $(\alg, \mlt)$ is contained in $\ker \la$.
\item\label{kos8} $\ker R(u) \subset \ker \la$ is a Lie subalgebra of $\alg$ on which $L(u)$ acts as a Lie algebra derivation.
\end{enumerate}
\end{lemma}

\begin{proof}
If $x \in \ker \la$ then $\la(R(u)x) = \la(L(u)x) = h(u, x) = \la(x) = 0$, so $L(u)\ker \la \subset \ker \la$ and $R(u)\ker \la \subset \ker \la$. This shows \eqref{lurupre}. Taking $y = u$ and $h = \tau$ in \eqref{hessianmetric} yields \eqref{lladj} and so also \eqref{rsa}. Let $N = R(u) - R(u)^{2} = [L(u), R(u)]$. Then $\tr N^{k+1} = \tr [L(u), R(u)]N^{k} = \tr L(u)[R(u), N^{k}] = 0$ for all $k \geq 0$. This implies $N = R(u) - R(u)^{2}$ is nilpotent.

Next it is claimed that the Fitting decomposition $\alg = \alg_{0} \oplus \alg_{1}$ of $\alg$ with respect to $R(u)$ is an $h$-orthogonal $L(u)$-invariant direct sum, that is $\alg_{0}^{\perp} = \alg_{1}$ and $\alg_{1}^{\perp} = \alg_{0}$ and $L(u)$ preserves $\alg_{0}$ and $\alg_{1}$, and, moreover, the Fitting decomposition $\alg = \bar{\alg}_{0} \oplus \bar{\alg}_{1}$ of $\alg$ with respect to $I - R(u)$ satisfies $\bar{\alg}_{0} = \alg_{1}$ and $\bar{\alg}_{1} = \alg_{0}$. By definition of the Fitting decomposition $\alg_{0} = \sum_{k \geq 1} \ker R(u)^{k}$ and $\alg_{1} = \cap_{k \geq 1}R(u)^{k}\alg$. Because $\alg$ is finite-dimensional there is a minimal $p$ such that $\alg_{0} = \ker R(u)^{p}$ and $\alg_{1} = R(u)^{p}\alg$. If $z\in \alg_{1}$ there is $y \in \alg$ such that $z = R(u)^{p}y$ and so for all $x \in \alg_{0}$, $h(z, x) = h(R(u)^{p}y, x) = h(y, R(u)^{p}x) = 0$, showing $\alg_{1} \subset \alg_{0}^{\perp}$. If $z \in \alg_{1}^{\perp}$, then $h(R(u)^{p}z, x) = h(z, R(u)^{p}x) = 0$ for all $x \in \alg$, so $R(u)^{p}z = 0$ and $z \in \alg_{0}$. Thus $\alg_{1}^{\perp} \subset \alg_{0}$. This shows $\alg_{0}^{\perp} = \alg_{1}$ and $\alg_{1}^{\perp} = \alg_{0}$. It is easily checked, by induction on $k$, that $[R(u)^{k},L(u)] = k(R(u)^{2} -R(u))R(u)^{k-1}$ for $k \geq 1$. If $x \in \alg_{0}$ then $R(u)^{p}L(u)x = L(u)R(u)^{p}x + p(R(u)^{p+1} -R(u)^{p})x = 0$, so $L(u)x \in \alg_{0}$. If $x \in \alg_{1}$ then there is $y \in \alg$ such that $x = R(u)^{p}y$ and $L(u)x = L(u)R(u)^{p}y = R(u)^{p}(L(u) + p(I - R(u)))y \in \alg_{1}$. This shows that the Fitting decomposition is $L(u)$ invariant. If $(I - R(u))^{m}x = 0$ then 
\begin{align}\label{xru}
x = -\sum_{j = 1}^{m}\binom{m}{j}(-R(u))^{j}x = R(u)\sum_{j = 1}^{m}\binom{m}{j}(-R(u))^{j-1}x. 
\end{align}
This shows $x \in \im R(u)$, and, substituted in \eqref{xru} repeatedly, this implies $x \in \cap_{j \geq 1}\im R(u)^{j} = \alg_{1}$. Consequently $\bar{\alg}_{0} \subset \alg_{1}$. Let $q \geq 1$ be such that $(R(u) - R(u)^{2})^{q} = 0$. If $m \geq q$ and $x = (I - R(u))^{m}y$ then $R(u)^{m}x = (R(u) - R(u)^{2})^{m}y = 0$, so $x \in \ker R(u)^{m} \subset \alg_{0}$. Hence $\bar{\alg}_{1} \subset \alg_{0}$. Since $\bar{\alg}_{0} \oplus \bar{\alg}_{1} = \alg = \alg_{0}\oplus \alg_{1}$, this suffices to show $\bar{\alg}_{1} = \alg_{0}$ and $\bar{\alg}_{0} = \alg_{1}$. 

Since the only eigenvalue of $R(u) - R(u)^{2}$ (over the algebraic closure of the base field) is $0$, the minimal polynomial of $R(u)$ divides some power of $x(1-x)$, and any eigenvalue of $R(u)$ is either $0$ or $1$.
Hence $k = \tr R(u)$ is an integer no greater than $n$. Since $R(u)$ is invertible on $\alg_{1}$ and nilpotent on $\alg_{0}$, $k = \tr R(u) = \dim \alg_{1}$. Since $R(u)u = u$, $k \geq 1$. Let $q \geq 1$ be the minimal integer such that $(R(u) - R(u)^{2})^{q} = 0$. Since $R(u)$ is nilpotent on $\alg_{0}$ and $I - R(u)$ is nilpotent on $\alg_{1}$ there is a minimal integer $t\geq 1$ such that $R(u)^{t}$ vanishes on $\alg_{0}$ and $(I-R(u))^{t}$ vanishes on $\alg_{1}$. Then $(R(u) - R(u)^{2})^{t} = R(u)^{t}(I - R(u))^{t} = 0$, so $t \geq q$. Since $t \leq \max\{\dim \alg_{0}, \dim \alg_{1}\} = \max\{n-k, k\}$, it follows that $q \leq \max\{n-k, k\}$.

Since, for all $x \in \alg$, $\la(R(u)x) = \la(x \mlt u) = h(x, u) = \la(x)$, there holds $\la(R(u)^{k}x) = \la(x)$ for all $k \geq 1$. If $x \in \alg_{0}$ then $x \in \ker R(u)^{p}$ for some $p \geq 0$, so $0 = \la(R(u)^{p}x) = \la(x)$, showing that $x \in \ker \la$. This shows \eqref{a0l}. Suppose $J$ is a nontrivial left ideal in $(\alg, \mlt)$ and $0\neq x \in J \cap \ker \la$. By the nondegeneracy of $h$ there is $y$ such that $\la(y\mlt x) = h(x, y) \neq 0$, and so $y \mlt x \in J$ but $y \mlt x \notin \ker \la$. This shows \eqref{rclsa6b}. Claim \eqref{kos8} is straightforward.
\end{proof}

Claims \eqref{lladj} and \eqref{rsa} are stated explicitly as  Proposition $2.6$ and Corollary $2.7$ of \cite{Choi-Chang}. When $h$ is positive definite the content of Lemma \ref{principalidempotentlemma} is due to Vinberg in \cite{Vinberg} (see pages $369-370$), and is also in \cite{Shima-homogeneoushessian}; the proofs for general $h$ are essentially the same, although the conclusions are not as strong as for positive definite $h$. If $h$ is positive definite then that $R(u) - R(u)^{2}$ be nilpotent and self-adjoint means it is zero, so $R(u)$ is a projection operator. In this case, $R(u)$ commutes with $L(u)$, so if the eigenvalues of $L(u)$ are real, so too are those of $R(u)$, in which case it follows that $L(u)$ is self-adjoint. From these observations Vinberg constructed a canonical decomposition of the LSA on which is based his structure theory for homogeneous convex cones. 
These results motivate much of what follows. When $h$ has mixed signature the conditions that $R(u) - R(u)^{2}$ be self-adjoint and nilpotent alone are insufficient to conclude its vanishing; in this case it is not clear what purely algebraic condition forces the conclusion that $R(u)$ be a projection.

The \textit{rank} of a Koszul form is the integer $\tr R(u)$. The rank is a basic invariant of the Koszul form. For the LSAs appearing in the classification of homogeneous convex cones, the idempotent $u$ is a unit, and so $\tr R(u) = n$, whereas for the LSAs giving rise to homogeneous improper affine spheres, $R(u)$ is a projection onto a one-dimensional ideal, so $\tr R(u) = 1$. Thus these two cases can be seen as opposite extremes.

The \emph{right principal idempotent} of an incomplete LSA with nondegenerate trace form $\tau$ is the idempotent $r$ associated with the Koszul form $\tr R$. 

\begin{lemma}\label{trrunilemma}
For an incomplete LSA $(\alg, \mlt)$ with nondegenerate trace form $\tau$ and right principal idempotent $r$, $(\ker \tr R, [\dum, \dum])$ is a unimodular Lie algebra if and only if $\tr R(r) \tr L = \tr L(r) \tr R$.
\end{lemma}
\begin{proof}
Clearly $\mu = \tr R(r) \tr L - \tr L(r) \tr R \in \alg^{\ast}$ satisfies $\mu(r) = 0$. Since $[\alg, \alg] \subset \ker \tr R$, there holds $\tr \ad_{\ker \tr R}(x) = \tr \ad(x) = \tr L(x) - \tr R(x) = \tr L(x)$ for $x \in \ker \tr R$. Hence $\mu(x) = \tr R(r)\tr\ad_{\ker \tr R}(x)$ for $x \in \ker \tr R$. Since by Lemma \ref{principalidempotentlemma}, $1 \leq \tr R(r) \leq \dim \alg$, it follows that $\mu$ vanishes if and only if $\ker \tr R$ is unimodular.
\end{proof}

Theorem \ref{lsacptheorem} shows that if the characteristic polynomial of an LSA satisfies certain conditions then its level sets are improper affine spheres. This motivates identifying algebraic properties of the LSA that guarantee that its characteristic polynomial satisfies these conditions.

\begin{theorem}\label{lsacptheorem}
If the characteristic polynomial $P(x)$ of an $n$-dimensional LSA $(\alg, \mlt)$ satisfies $P(x + tv) = P(x) + \la t$ for some $0 \neq v \in \alg$ and $\la \in \rea$ and $(\alg, \mlt)$ satisfies $2\tr L = (n+1)\tr R$, then $\H(e^{P}) = \kc e^{nP}$ for some constant $\kc \in \rea$. 
If, moreover, the trace form $\tau$ is nondegenerate, then 
\begin{enumerate}
\item\label{thom2} $\kc \neq 0$ and each level set of $P$ is an improper affine sphere with affine normal equal to a multiple of $v$;
\item\label{thom1} $\la \neq 0$, the element $r = \la^{-1}v$ is the right principal idempotent of $(\alg, \mlt)$, and $\tr R(r) = 1$.
\end{enumerate}
\end{theorem}
\begin{proof}
For any LSA, differentiating \eqref{holgp} yields
\begin{align}\label{hessholp}
\begin{split}
(\det &\lin(\hol(g)))^{-2} \hol(g)\cdot \det(\hess P + dP \tensor dP) \\&= \det(\hess(\hol(g)\cdot P) + d(\hol(g)\cdot P) \tensor d(\hol(g)\cdot P))\\
& = \det(\chi(g^{-1})\hess P + \chi(g^{-2})dP\tensor dP)= \chi(g)^{-n}\det(\hess P + \chi(g)^{-1}dP\tensor dP).
\end{split}
\end{align}
By Lemma \ref{twisteddetlemma}, if $e^{P}$ is translationally homogeneous in some direction then the last term in \eqref{hessholp} equals $\chi(g)^{-n-1}\det(\hess P + dP \tensor dP)$. In such a case taking $g = \exp_{G}(a)$ and noting $\det \lin(\hol(\exp_{G}(a))) = \det e^{L(a)} = e^{\tr L(a)}$ yields
\begin{align}
\hol(g)\cdot \det(\hess P + dP \tensor dP) = e^{2\tr L(a) - (n+1)\tr R(a)}\det (\hess P + dP \tensor dP).
\end{align}
If $2\tr L  = (n+1)\tr R$, it follows that $\det (\hess P + dP \tensor dP)$ is constant on an open set, and so constant, because it is a polynomial. In this case $e^{P}$ solves $\H(e^{P}) = \kc e^{n P}$ for some constant $\kc$. If $\kc \neq 0$ then the conclusion of claim \eqref{thom2} follows from Theorem \ref{ahtheorem}. Suppose $\tau$ is nondegenerate. To show $\kc \neq 0$ it suffices to show that $\H(e^{P})$ is nonzero at a single point of $\alg$. Since $P\hess(\log P) = P_{ij} - P_{i}P_{j}$ and $e^{-P}\hess(e^{P}) = P_{ij} + P_{i}P_{i}$, it follows from $P(x + tv) = P(x) + \la t$ and Lemma \ref{twisteddetlemma} that $P^{n}\H(\log P) = - e^{-nP}\H(e^{P})$. By \eqref{hesslogp} the nondegeneracy of $\tau$ implies that $\H(\log P) \neq 0$ at $0$, and so $\H(e^{P})$ is nonzero at $0$. There remains to show \eqref{thom1}. Differentiating $P(x + tv) = P(x) + \la t$ at $t = 0$ and using \eqref{dptra} shows that $\la = dP_{0}(v) = P(0)\tr R(v) = \tr R(v)$, and using this and \eqref{hesslogp}, shows that $-\la \tr R(a) = (\hess P)_{0}(a, v) - \la \tr R(a) = (\hess \log P)_{0}(a, v) = -\tau(a, v)$ for all $a \in \alg$. By the nondegeneracy of $\tau$, were $\la$ zero then $v$ would be zero, a contradiction. Hence $\la \neq 0$ and $r = \la^{-1}v$ satisfies $\tr R(r) = 1$ and $\tau(a, r) = \tr R(a)$.  
\end{proof}

\section{Notions of nilpotence for LSAs}\label{nilpotencesection}
Various notions of nilpotence, and their interrelations, that play a role in what follows are explained now. The reader should be aware that the terminology for various different notions of nilpotence is different in different papers on LSAs. For example, in \cite{Helmstetter} an LSA is called \textit{nilpotent} if it is what is here called \textit{right nil}, while here \textit{nilpotent} is used in a different sense.  

\begin{lemma}\label{solvablekernellemma}
For a finite-dimensional LSA $(\alg, \mlt)$ over a field of characteristic zero the kernel $\ker L$ of the left regular representation $L$ is a two-sided ideal and a trivial left-symmetric subalgebra of $(\alg, \mlt)$, for which $-R$ is a Lie algebra representation by commuting operators.
\end{lemma}
\begin{proof}
By the definition of an LSA, $\ker L$ is a Lie subalgebra of $(\alg, [\dum, \dum])$. If $x \in \alg$ and $n \in \ker L$ then $n\mlt x = L(n)x = 0 \in \ker L$, and $L(x\mlt n) = L(n\mlt x) + [L(x), L(n)] = 0$, showing that $\ker L$ is a two-sided ideal of $(\alg, \mlt)$. By \eqref{rlsa}, if $n, m \in \ker L$ then $0 = R(m \mlt n) = R(n)R(m)$, so $0 = R([m, n]) = -[R(m), R(n)]$. This shows $-R$ is a Lie algebra representation of $\ker L$ on $\alg$ by commuting operators. 
\end{proof}

A finite-dimensional LSA $(\alg, \mlt)$ defined over a field of characteristic zero is \textit{triangularizable} if there is a basis of $\alg$ with respect to which $L(x)$ is triangular for every $x \in \alg$. 
\begin{lemma}\label{cslemma}
For a finite-dimensional LSA $(\alg, \mlt)$ defined over a field $\fie$ of characteristic zero, the following are equivalent:
\begin{enumerate}
\item\label{csl1} $(\alg, \mlt)$ is triangularizable.
\item\label{csl2} The underlying Lie algebra $(\alg, [\dum, \dum])$ is solvable and for every $x \in \alg$ the eigenvalues of $L(x)$ are contained in $\fie$.
\end{enumerate}
\end{lemma}
\begin{proof}
Suppose $(\alg, \mlt)$ is triangularizable and fix a basis with respect to which $L(x)$ is triangular for all $x \in \alg$. This implies that the eigenvalues of $L(x)$ are contained in $\fie$. Then $L([x, y]) = [L(x), L(y)]$ is strictly triangular for all $x, y \in \alg$ and so $L(a)$ is strictly triangular for every $a \in [\alg, \alg]$. This implies $\tr L(a)L(b) = 0$ for all $a, b \in [\alg, \alg]$. Since, by Lemma \ref{solvablekernellemma}, $\ker L$ is an abelian Lie algebra,  this implies $(\alg, [\dum, \dum])$ is solvable, by Cartan's criterion (see, e.g., section III.$4$ of \cite{Jacobson}). On the other hand, if there holds \eqref{csl2} then, by one version of Lie's Theorem (see, e.g., Theorem $1.2$ in chapter $1$ of \cite{Gorbatsevich-Onishchik-Vinberg}), there is a complete flag in $\alg$ invariant under $L(\alg)$, so $(\alg, \mlt)$ is triangularizable.  
\end{proof}

Lemma \ref{completesolvablelemma} is due to Helmstetter.
\begin{lemma}[Proposition $(20)$ and Corollary $(21)$ of \cite{Helmstetter}]\label{completesolvablelemma} Let $\fie$ be a field of characteristic zero.
\noindent
\begin{enumerate}
\item\label{ulie1} The underlying Lie algebra of a complete finite-dimensional LSA over $\fie$ is solvable.
\item\label{ulie2} The underlying Lie algebra of a finite-dimensional LSA over $\fie$ is not perfect. In particular, the codimension of $[\alg, \alg]$ in $\alg$ is always at least one.
\end{enumerate}
\end{lemma}
\begin{proof}
If the underlying Lie algebra of a finite-dimensional LSA $(\alg, \mlt)$ over $\fie$ is not solvable, it contains a nontrivial semisimple Lie subalgebra $\S$ and $L$ is a representation of $\S$ on $\alg$. Since $\S$ is semisimple, the first Lie algebra cohomology of any finite-dimensional representation of $\S$ is trivial. View $\alg$ as an $\S$-module with the action given by $L$. Since the inclusion of $\S$ in $\alg$ is a Lie algebra cocycle of $\S$ with coefficients in $\alg$, there exists $a \in \S$ such that $x = L(x)a  = R(a)x$ for all $x \in \S$. This means that $I - R(a)$ is not invertible, so $P(-a) = 0$. This shows \eqref{ulie1}.
For a general finite-dimensional LSA $(\alg, \mlt)$, were $\alg = [\alg, \alg]$, then, by \eqref{rlsa}, $\alg$ would be right nil, and so complete. By \eqref{ulie1} this would imply that the underlying Lie algebra was solvable, contradicting $\alg = [\alg, \alg]$.
\end{proof}

\begin{corollary}
A complete finite-dimensional LSA over a field $\fie$ of characteristic zero is triangularizable if and only if for every $x \in \alg$ the eigenvalues of $L(x)$ are contained in $\fie$. In particular, a complete finite-dimensional LSA over an algebraically closed field of characteristic zero is triangularizable.
\end{corollary}
\begin{proof}
This follows from Lemmas \ref{completesolvablelemma} and \ref{cslemma}.
\end{proof}

\begin{lemma}[Proposition $(26)$ of \cite{Helmstetter}]\label{underlyingnilpotentlemma}
A finite-dimensional LSA over a field of characteristic zero is left nil if and only if it is right nil and its underlying Lie algebra is nilpotent.
\end{lemma}

\begin{proof}
This is Proposition $(26)$ of \cite{Helmstetter} (see also section $2$ of \cite{Kim-completeleftinvariant}).
\end{proof}

Associated with an LSA $(\alg, \mlt)$ there are the following descending series of subspaces of $\alg$. Define $\alg^{1} = \alg$, $\lnil^{1}(\alg) = \alg$ and $\rnil^{1}(\alg) = \alg$, and define recursively $\alg^{i+1} = [\alg, \alg^{i}]$, $\lnil^{i+1}(\alg) = \alg \mlt \lnil^{i}(\alg) = L(\alg)\lnil^{i}(\alg)$, and $\rnil^{i+1}(\alg) = \rnil^{i}(\alg)\mlt \alg = R(\alg)\rnil^{i}(\alg)$. It can be checked by induction on $i$ that $\alg^{i}\mlt \lnil^{j}(\alg) \subset \lnil^{i+j}(\alg)$; using this it can be checked by induction on $i$ that $\alg^{i} \subset \lnil^{i}(\alg)$. Using $\alg^{i} \subset \lnil^{i}(\alg)$ it can be checked by induction on $i$ that $\rnil^{i}(\alg)$ is a two-sided ideal of $(\alg, \mlt)$. This fact is contained in Proposition $23$ of \cite{Helmstetter} and the proof just sketched is indicated in \cite{Kim-completeleftinvariant}. The LSA $(\alg, \mlt)$ is \textit{right nilpotent} of \textit{length $k$} if there is $k \geq 1$ such that $\rnil^{k}(\alg) \neq \{0\}$ and $\rnil^{k+1}(\alg) = \{0\}$. A two-sided ideal in a right nilpotent LSA is right nilpotent, as is the quotient of the LSA by the ideal. A right nilpotent LSA is right nil, but a right nil LSA need not be right nilpotent.

Lemma \ref{rightnilpotentlemma} is a slight refinement of claim $(3)$ of Proposition $24$ of \cite{Helmstetter}.
\begin{lemma}\label{rightnilpotentlemma} 
For an $n$-dimensional LSA $(\alg, \mlt)$ the following are equivalent:
\begin{enumerate}
\item\label{jdes1} $(\alg, \mlt)$ is right nilpotent.
\item\label{jdes2} There is a sequence of two-sided ideals $\alg = \I^{1} \supset \I^{2} \supset \dots \supset \I^{r} = \{0\}$ such that $R(\alg)\I^{i} \subset \I^{i+1}$ for all $1 \leq i \leq r$. In particular each quotient $\I^{i}/\I^{i+1}$ is a trivial LSA.
\item\label{jdes3} There is a sequence of two-sided ideals $\alg = \J^{1} \supset \J^{2} \supset \dots \supset \J^{n} \supset \J^{n+1} = \{0\}$ such that $\dim \J^{i} = n+1-i$ and such that $R(\alg)\J^{i} \subset \J^{i+1}$ for all $1 \leq i \leq n$. In particular each quotient $\J^{i}/\J^{i+1}$ is a trivial one-dimensional LSA.
\end{enumerate}
In the case there hold \eqref{jdes1}-\eqref{jdes3}, the LSA $(\alg, \mlt)$ is triangularizable.
\end{lemma}
\begin{proof}
If $(\alg, \mlt)$ is right nilpotent, then 
\eqref{jdes1} implies \eqref{jdes2} with $\I^{i}=  \rnil^{i}(\alg)$. In a trivial LSA any subspace is a two-sided ideal, so any descending sequence of subspaces, each having codimension one in the preceding, is a descending sequence of two-sided ideals, each having codimension one in the preceding. Supposing given ideals as in \eqref{jdes2}, choosing such a sequence in each quotient $\I^{i}/\I^{i+1}$ and lifting the result to $\alg$ yields the desired sequence of ideals $\J^{j}$. Suppose given a sequence of ideals $\J^{i}$ as in \eqref{jdes3}. Then $R(a)\J^{i} \subset \J^{i+1}$ for all $a \in \alg$, so $R(a_{1})\dots R(a_{n})\J^{1} = \{0\}$ for all $a_{1}, \dots, a_{n} \in \alg$, which means $\alg$ is right nilpotent. If there holds \eqref{jdes3}, then there is a complete flag of subspaces in $\alg$ stable under $L(\alg)$, so $(\alg, \mlt)$ is triangularizable.
\end{proof}

In a finite-dimensional algebra $(\alg, \mlt)$ define $\mnil^{1}(\alg) = \alg$ and define $\mnil^{i}(\alg)$ to be the vector subspace of $\alg$ generated by all products of at least $i$ elements of $\alg$, however associated. Each term of the decreasing sequence $\alg = \mnil^{1}(\alg) \supset \mnil^{2}(\alg) \supset \dots \supset \{0\}$ is a two-sided ideal in $\alg$. If there is $k$ such that $\mnil^{k}(\alg) = \{0\}$, then $(\alg, \mlt)$ is \textit{nilpotent}. By a theorem of I.~M.~H. Etherington \cite{Etherington}, an algebra is nilpotent if and only if the associative multiplication algebra $\mult(\alg) \subset \eno(\alg)$ generated by $L(\alg)$ and $R(\alg)$ is nilpotent. The proof amounts to showing that 
\begin{align}\label{mnilmult}
\mnil^{i+1}(\alg) \supset \mult(\alg)^{i}\alg \supset \mnil^{2^{i}}(\alg). 
\end{align}
(Care is needed because although $\mult^{i}(\alg)\alg  = L(\alg)\mult^{i-1}(\alg)\alg + R(\alg)\mult^{i-1}(\alg)\alg$, by definition of $\mult(\alg)$, it need not be the case that $L(\alg)\mnil^{i}(\alg) + R(\alg)\mnil^{i}(\alg)$ equal $\mnil^{i+1}(\alg)$.)

By \eqref{mnilmult}, a nilpotent LSA is right nilpotent with nilpotent underlying Lie algebra because $\rnil^{k+1}(\alg) = R(\alg)^{k}\alg \subset \mult(\alg)^{k}\alg$ and $\alg^{k+1} = \ad(\alg)^{k}\alg \subset \mult(\alg)^{k}\alg$. Although, for a general not necessarily associative algebra, nilpotence is stronger than right nilpotence, Theorem \ref{trivalgtheorem} shows that a right nilpotent LSA with nilpotent underlying Lie algebra is nilpotent.

When considering multiple LSAs it is convenient to use subscripts indicating the dependence of left and right regular representations $L_{\mlt}$ and $R_{\mlt}$ corresponding to a given LSA $(\alg, \mlt)$. However, such subscripts will be used only when necessary, and when they are omitted, $L$ and $R$ are always defined with respect to the multiplication indicated by the symbol $\mlt$.

\begin{lemma}\label{hereditylemma}
Let $(\alg, \mlt)$ be a finite-dimensional LSA. If $(\alg, \mlt)$ is left nil, right nil, right nilpotent, nilpotent, or has nilpotent underlying Lie algebra, then any left-symmetric subalgebra or any homomorphic image of $(\alg, \mlt)$ has the same property. 
\end{lemma}

\begin{proof}
All the claims for a homomorphic image follow from the observation that, if $\Phi:(\alg, \mlt) \to (\balg, \circ)$ is a surjective LSA homormophism, then 
\begin{align}\label{rphir}
&R_{\circ}(\Phi(x))\circ \Phi = \Phi \circ R_{\mlt}(x), & &L_{\circ}(\Phi(x)) \circ \Phi = \Phi \circ L_{\mlt}(x),
\end{align}
for all $x \in \alg$. By \eqref{rphir} and its analogue for left multiplication operators it is immediate that a homomorphic image of a right nil or a left nil LSA has the same property. Similarly, by \eqref{rphir}, there hold $\rnil^{k}(\balg) = \Phi(\rnil^{k}(\alg))$ and $\mnil^{i}(\balg) = \Phi(\mnil^{i}(\alg))$, from which it follows that a homomorphic image of a right nilpotent or a nilpotent LSA is right nilpotent or nilpotent. The analogous claims for subalgebras are all straightforward.  
\end{proof}

Let $(\alg, \mlt)$ be a finite-dimensional LSA. 
Define a two-sided ideal  $\triv(\alg)$ by 
\begin{align}
\triv(\alg) = \ker L \cap \ker R = \{a \in \alg: a\mlt x = 0 = x\mlt a \,\,\text{for all}\,\, x \in \alg\}.
\end{align} 
Any vector subspace $\I \subset \triv(\alg)$ is also a two-sided ideal of $(\alg, \mlt)$. 
\begin{lemma}
The subspaces $\triv^{i}(\alg)$ of the LSA $(\alg, \mlt)$ defined by $\triv^{1}(\alg) = \triv(\alg)$ and
\begin{align}
\triv^{i+1}(\alg) = \{z \in \alg: L(z)\alg \subset \triv^{i}(\alg), R(z)\alg \subset \triv^{i}(\alg)\}
\end{align}
for $i \geq 1$ are two-sided ideals in $(\alg, \mlt)$ satisfying $\triv^{1}(\alg) \subset \triv^{2}(\alg) \subset \dots \subset \alg$.
\end{lemma}
\begin{proof}
The proof is by induction on $i$. The case $i = 1$ is clear. Suppose that for $1 \leq j \leq i$ it is known that $\triv^{j}(\alg)$ is a two-sided ideal of $(\alg, \mlt)$ and there holds $\triv^{1}(\alg) \subset \dots \subset \triv^{j}(\alg)$. If $z \in \triv^{i}(\alg)$ and $x \in \alg$ then $z \mlt x$ and $x \mlt z$ are contained in $\triv^{i}(\alg)$ by the inductive hypothesis, so $z \in \triv^{i+1}(\alg)$. Suppose $z \in \triv^{i+1}(\alg)$ and $x \in \alg$. Since $z \mlt x, x \mlt z \in \triv^{i}(\alg)$ and $\triv^{i}(\alg)$ is a two-sided ideal in $(\alg, \mlt)$ it follows that $L(z \mlt x)y$, $R(z\mlt x)y$, $L(x\mlt z)y$, and $R(x\mlt z)y$ are contained in $\triv^{i}(\alg)$, which shows that $x\mlt z$ and $z \mlt x$ are contained in $\triv^{i+1}(\alg)$. 
\end{proof}
From the definition it is immediate that $\triv(\alg/\triv^{i}(\alg)) = \triv^{i+1}(\alg)/\triv^{i}(\alg)$. 

\begin{lemma}\label{trivnilpotentlemma}
A finite-dimensional LSA $(\alg, \mlt)$ is nilpotent if and only if there is $m \geq 1$ such that $\triv^{m}(\alg) = \alg$.
\end{lemma}

\begin{proof}
An element $z \in \alg$ is contained in $\triv^{m}(\alg)$ if and only if any product of $m$ left and right multiplication operators annihilates $z$. Consequently, $\triv^{m}(\alg) = \alg$ if and only if $\mult(\alg)^{m}\alg = \{0\}$, where $\mult(\alg) \subset \eno(\alg)$ is the multiplication algebra generated by $L(\alg)$ and $R(\alg)$. By \eqref{mnilmult}, $\mult(\alg)^{m}\alg = \{0\}$ if and only if $\alg$ is nilpotent. 
\end{proof}

Theorem \ref{trivalgtheorem} generalizes Lemma $4.2$ of Choi and Kim's \cite{Choi-Kim}, which reaches similar conclusions in the case of a complete abelian LSA. 
\begin{theorem}\label{trivalgtheorem}
If a finite-dimensional LSA $(\alg, \mlt)$ is right nilpotent and has nilpotent underlying Lie algebra, then $\triv(\alg)$ has dimension at least one. In this case there is some $m \geq 1$ such that $\triv^{m}(\alg) = \alg$ and so $(\alg, \mlt)$ is nilpotent. If, moreover, $\alg \mlt \alg \neq \{0\}$ then $\triv(\alg) \cap (\alg \mlt \alg) \neq \{0\}$.
\end{theorem}
\begin{proof}
Since $(\alg, \mlt)$ is right nilpotent it is right nil, and since it has nilpotent underlying Lie algebra, it is left nil.
Let $\J^{i}$ be a descending sequence of ideals as in Lemma \ref{rightnilpotentlemma}. Then $\J^{n}$ is one-dimensional where $n = \dim \alg$. Let $z$ generate $\J^{n}$. Then $z \mlt x = 0$ for all $x \in \alg$, so $L(z) = 0$. Since $\J^{n} = \lb z \ra$ is a two-sided ideal, $R(z)x \in \lb z \ra$ for all $x \in \alg$. If there is $x \in \alg$ such that $R(z)x \neq 0$, then, after rescaling $x$, it can be supposed that $L(x)z = z$. Since this means $L(x)$ is not nilpotent, it contradicts that $(\alg, \mlt)$ is left nil. Hence $R(z) = 0$ and so $z \in \triv(\alg)$. 
Since, by Lemma \ref{hereditylemma}, $\alg/\triv^{i}(\alg)$ is again right nilpotent with nilpotent underlying Lie algebra, if $\triv^{i}(\alg) \neq \alg$ then $\triv^{i+1}(\alg)/\triv^{i}(\alg) = \triv(\alg/\triv^{i}(\alg))$ has dimension at least one, by the preceding. Since this implies that $\dim \triv^{i+1}(\alg) > \dim \triv^{i}(\alg)$ unless $\triv^{i}(\alg) = \alg$, it implies that there is some $m \geq 1$ such that $\triv^{m}(\alg) = \alg$. By Lemma \ref{trivnilpotentlemma}, $(\alg, \mlt)$ is nilpotent.

If $\alg \mlt \alg \neq \{0\}$, then the ideals $\J^{i}$ may be chosen so that $\J^{n} \subset \alg \mlt \alg$, and since $z \in \J^{n}$, this proves the final claim.
\end{proof}

\begin{lemma}
A right nilpotent LSA admits no Koszul form.
\end{lemma}

\begin{proof}
Since $(\alg, \mlt)$ is right nilpotent there is $k \geq 1$ such that $\rnil^{k+1}(\alg) = \{0\}$ and $\rnil^{k}(\alg) \neq \{0\}$. Then $\rnil^{k}(\alg)$ is a two-sided ideal contained in $\ker L$. However, if $(\alg, \mlt)$ were to admit a Koszul form, then $L$ would be injective.
\end{proof}

\section{LSAs with nondegenerate trace form and rank one right principal idempotent}
A derivation $D$ of the Hessian LSA $(\alg, \mlt, h)$ is \textit{conformal} if there is a constant $c \in \rea$ such that
\begin{align}\label{conformalder}
h(Dx, y) + h(x, Dy) = ch(x, y),
\end{align}
for all $x, y \in \alg$.  Equivalently, $D + D^{\ast} = cI$. If $c = 0$ in \eqref{conformalder} then $D$ is an \textit{infinitesimally isometric derivation}. The conformal derivations and the infinitesimally isometric derivations constitute Lie subalgebras of the Lie algebra of derivation of $(\alg, \mlt)$. Following Definition $2.8$ of \cite{Choi-domain}, a derivation $D$ of $(\alg, \mlt)$ satisfying \eqref{conformalder} with $c = 1$ is called a \textit{compatible derivation} of the Hessian LSA $(\alg, \mlt, h)$.

Following Choi in \cite{Choi-domain}, define the \textit{graph extension} $(\alg, \mlt, \hat{h})$ of the Hessian LSA $(\balg, \circ, h)$ with respect to the compatible derivation $D \in \eno(\balg)$ to be the vector space $\alg = \balg \oplus \lb D \ra$ equipped with the multiplication 
\begin{align}\label{gemlt}
(x + aD) \mlt (y + bD)  = x\circ y + aDy + (h(x, y) + ab)D
\end{align}
and the symmetric bilinear form 
\begin{align}\label{gehath}
\hat{h}(x + aD, y + bD) = h(x, y) + ab. 
\end{align}
\begin{lemma}\label{graphextensionlemma}
The graph extension $(\alg = \balg \oplus \lb D \ra, \mlt, \hat{h})$ of the Hessian LSA $(\balg, \circ, h)$ with respect to the compatible derivation $D \in \eno(\balg)$ satisfies:
\begin{enumerate}
\item\label{ge1} $(\alg, \mlt, \hat{h})$ is a Hessian LSA with Koszul form $\la(x + aD) = a$ generating $\hat{h}$.
\item\label{ge2} $D$ is the idempotent element of $(\alg, \mlt, \hat{h})$ associated with $\la$ and satisfies $\ker R_{\mlt}(D) = \ker \la = \balg \oplus \lb 0 \ra$, and $R_{\mlt}(D)^{2} = R_{\mlt}(D)$.
\item\label{ge3} The restriction to $\balg$ of $L_{\mlt}(D)$ equals $D$,
and
\begin{align}\label{geadj}
L_{\mlt}(D)^{\ast} + L_{\mlt}(D) = R_{\mlt}(D) + I.
\end{align}
\item\label{ge4} Any nontrivial two-sided ideal of $(\alg, \mlt)$ contains $D$.
\item\label{ge6b}The Lie normalizer $\n(\lb D\ra)$ of the linear span of $D$ in $\alg$ equals $\ker D \oplus \lb D \ra$. In particular $D$ is invertible if and only if $\n(\lb D\ra) = \lb D\ra$.
\item\label{ge6} If $D$ is invertible then $\balg = [\alg, \alg]$ and $(\alg, \mlt)$ is simple.
\item\label{ge5} For $x \in \balg \oplus \lb 0 \ra \subset \alg$ and $\hat{y} \in \alg$, 
\begin{align}\label{geder}
L_{\mlt}(D)(x \mlt \hat{y}) = L_{\mlt}(D)x \mlt \hat{y} + x \mlt L_{\mlt}(D)\hat{y}.
\end{align}
Consequently $e^{tL_{\mlt}(D)}(x \mlt \hat{y}) = e^{tL_{\mlt}(D)}x \mlt e^{tL_{\mlt}(D)}\hat{y}$, and $e^{tL_{\mlt}(D)}$ is an automorphism of $(\balg, \circ)$.
\end{enumerate}
Moreover, 
\begin{align}
\label{rmltrcirc1}\tr R_{\mlt}(x + aD) &= \tr R_{\circ}(x) + a,\\
\label{rmltrcirc2}\hat{h}(x + aD, y + bD) &= \tau_{\mlt}(x + aD, y + bD) - \tau_{\circ}(x, y),\\
\label{papb0}P_{\alg, \mlt}(x + aD) &= P_{\balg, \circ}(x)\left(1 + a - h(x,  (I + R_{\circ}(x))^{-1}Dx)\right),
\end{align} 
where $\tau_{\mlt}$ and $\tau_{\circ}$ are the right trace forms and $P_{\alg, \mlt}$ and $P_{\balg, \circ}$ are the characteristic polynomials of $(\alg, \mlt)$ and $(\balg, \circ)$. 
\end{lemma}

\begin{proof}
With respect to a basis $e_{1}, \dots, e_{n}$ of $\alg$ such that $e_{n} = D$ and $\balg \oplus \lb 0 \ra = \lb e_{1}, \dots, e_{n-1}\ra$, the left and right regular representations $L_{\circ}$, $R_{\circ}$, $L_{\mlt}$, and $R_{\mlt}$ of $(\balg, \circ)$ and $(\alg, \mlt)$ are related by
\begin{align}\label{mltcirc}
&L_{\mlt}(x + aD) = \begin{pmatrix} L_{\circ}(x) + aD & 0 \\ x^{\flat} & a \end{pmatrix},& &R_{\mlt}(x + aD) = \begin{pmatrix} R_{\circ}(x) & Dx \\ x^{\flat} & a \end{pmatrix}
\end{align}
where $x^{\flat} \in \balg^{\ast}$ is defined by $x^{\flat}(y) = h(x, y)$ for $y \in \balg$. Claims \eqref{ge1}-\eqref{ge3} can be verified by straightforward computations using the definitions and \eqref{mltcirc}.
If $J$ is a two-sided ideal in $(\alg, \mlt)$, by \eqref{rclsa6b} of Lemma \ref{principalidempotentlemma} there is $z \in J$ such that $\la(z) \neq 0$. Then $\la(z)D = z\mlt D \in J$, so $D \in J$. This shows \eqref{ge4}. 
If $x + aD \in \n(\lb D \ra)$ then $Dx = [D, x + aD] \in \lb D\ra$. As this holds if and only if $x \in \ker D$, $\n(\lb D \ra) = \ker D \oplus \lb D \ra$. This shows \eqref{ge6b}. 

If $D$ is invertible, then for $x \in \balg$ there exists $z \in \balg$ such that $x = Dz = D\mlt z - z\mlt D$. This shows that $\balg \subset [\alg, \alg]$. Since the opposite containment is clear from \eqref{gemlt}, this shows $\balg = [\alg, \alg]$. By \eqref{ge4}, a nontrivial two-sided ideal $\J$ in $(\alg, \mlt)$ contains $D$. From \eqref{mltcirc} it is apparent that $L_{\mlt}(D)$ is invertible if and only if $D$ is invertible. Consequently, if $D$ is invertible, $\alg = L_{\mlt}(D)\alg \subset \J$. This shows that $(\alg, \mlt)$ is simple and completes the proof of \eqref{ge6}.

Claim \eqref{ge5} is essentially Lemma $3.2$ of \cite{Shima-homogeneoushessian}. The proof is recalled for convenience.
Since $\ker R_{\mlt}(D) = \ker \la$, for $x \in \ker \la$ and $\hat{y} \in \alg$ there holds \eqref{geder}.
Since $L_{\mlt}(D)$ preserves $\ker \la$ there follows $L_{\mlt}(D)^{m}(x \mlt \hat{y}) = \sum_{j = 0}^{m}\binom{m}{j}L_{\mlt}(D)^{j}x \mlt L_{\mlt}(D)^{m-j}\hat{y}$, and so $e^{tL_{\mlt}(D)}(x \mlt \hat{y}) = \sum_{m \geq 0}\sum_{j = 0}^{m}t^{m}(j!(m-j)!)^{-1}L_{\mlt}(D)^{j}x \mlt L_{\mlt}(D)^{m-j}\hat{y} = e^{tL_{\mlt}(D)}x \mlt e^{tL_{\mlt}(D)}\hat{y}$. Differentiating $\hat{h}(e^{tL_{\mlt}(D)}x, e^{tL_{\mlt}(D)}\hat{y})$ in $t$ and simplifying the result using \eqref{hessianmetric} and \eqref{geadj} yields
\begin{align}
\begin{split}
\tfrac{d}{dt}\hat{h}(e^{tL_{\mlt}(D)}x, e^{tL_{\mlt}(D)}\hat{y}) &= \hat{h}(D \mlt e^{tL_{\mlt}(D)}x, e^{tL_{\mlt}(D)}\hat{y}) + \hat{h}(e^{tL_{\mlt}(D)}x, D\mlt e^{tL_{\mlt}(D)}\hat{y}) \\
&= \hat{h}(e^{tL_{\mlt}(D)}x \mlt D, e^{tL_{\mlt}(D)}\hat{y}) + \hat{h}(D, e^{tL_{\mlt}(D)}x \mlt e^{tL_{\mlt}(D)}\hat{y}) \\
&= \hat{h}(D, e^{tL_{\mlt}(D)}(x \mlt \hat{y})) = \hat{h}(e^{tL_{\mlt}(D)^{\ast}}D, x\mlt \hat{y}) = e^{t}\hat{h}(x, \hat{y}),
\end{split}
\end{align}
which implies $\hat{h}(e^{tL_{\mlt}(D)}x, e^{tL_{\mlt}(D)}\hat{y}) = e^{t}\hat{h}(x, \hat{y})$. Since for $x, y \in \ker \la$ there holds $x \mlt y = x \circ y$, there follows $e^{tL_{\mlt}(D)}(x \circ y) = e^{tL_{\mlt}(D)}x \circ e^{tL_{\mlt}(D)}y$, showing \eqref{ge5}.

The identities \eqref{rmltrcirc1}-\eqref{rmltrcirc2} follow from \eqref{mltcirc}. Writing vertical bars to indicate determinants,
\begin{align}
\begin{split}
P_{\alg, \mlt}(x + aD) 
& = \begin{vmatrix} I + R_{\circ}(x) & Dx\\ x^{\flat} & 1 + a \end{vmatrix} 
  = \begin{vmatrix} I + R_{\circ}(x) & Dx\\ x^{\flat} & 1 \end{vmatrix} + \begin{vmatrix} I + R_{\circ}(x) & 0\\ x^{\flat} & a \end{vmatrix}\\
& = \begin{vmatrix} I + R_{\circ}(x) - (Dx) \tensor x^{\flat}& 0 \\ x^{\flat} & 1 \end{vmatrix} + a| I + R_{0}(x)|\\
& = |I + R_{0}(x)|(1 - h(x, (I + R_{\circ}(x))^{-1}Dx)) +  a| I + R_{0}(x)| \\
&= P_{\balg, \circ}(x)\left(1 + a - h(x, (I + R_{\circ}(x))^{-1}Dx) \right),
\end{split}
\end{align}
where the penultimate equality follows from a standard identity for the determinant of a rank one perturbation.
\end{proof}

\begin{remark}
Claim \eqref{ge1} of Lemma \ref{graphextensionlemma} is a reformulation of Proposition $4.8$ of Choi's \cite{Choi-domain} (Choi assumes $(\balg, \mlt)$ is complete with unimodular underlying Lie algebra but, while this is necessary for the applications in \cite{Choi-domain}, it is irrelevant for the lemma). 
\end{remark}

\begin{remark}
Conclusion \eqref{ge6} of Lemma \ref{graphextensionlemma}, that the graph extension of a Hessian LSA by an invertible compatible derivation is simple, answers affirmatively a question posed by Choi in Remark $5.2$ of \cite{Choi-domain}.
\end{remark}

\begin{lemma}\label{completegraphextensionlemma}
Let $(\alg, \mlt, \hat{h})$ be the graph extension of the Hessian LSA $(\balg, \circ, h)$ with respect to the compatible derivation $D \in \eno(\balg)$. Then $(\balg, \circ)$ is complete if and only if $\ker \tr R_{\mlt} = \ker R_{\mlt}(D)$. In this case:
\begin{enumerate}
\item $\hat{h} = \tau_{\mlt}$.
\item\label{cge6} There is an integer $0 \leq m \leq \dim \balg - 1$ such that characteristic polynomial $P$ of $(\alg, \mlt)$ has the form
\begin{align}\label{papb}
P(x + aD) &= 1 + a - h(x, \sum_{l = 0}^{m}(-R_{\circ}(x))^{l}Dx).
\end{align}
In particular, $e^{P}$ is $1$-translationally homogeneous with axis $D$.
\item\label{cgenil} If $\deg P = k+1$, there exists $x \in \balg$ such that $R_{\circ}(x)^{k-1} \neq 0$; consequently, $\rnil^{k}(\balg, \circ) \neq \{0\}$.
\item\label{cge5b} The Lie algebras underlying $(\balg, \circ)$ and $(\alg, \mlt)$ are solvable.
\item\label{cgewt} $dP(E) = P$ where $E$ is defined by $E_{X} = (I + R_{\mlt}(X))D = D + D\mlt X$ for $X \in \alg$.
\end{enumerate}
If, additionally, the Lie algebra underlying $(\balg, \circ)$ is unimodular, then 
\begin{enumerate}
\setcounter{enumi}{5}
\item \label{cge10} there is a nonzero constant $\kc$ such that $\H(e^{P}) = \kc e^{nP}$, where $n = \dim \alg$, and the level sets of $P$ are improper affine spheres with affine normal a constant multiple of $D$.
\end{enumerate}
\end{lemma}
\begin{proof}
Note that $\balg \oplus \lb 0 \ra = \ker R_{\mlt}(D)$ by definition of a graph extension. If $(\balg, \circ)$ is complete, then $\tr R_{\circ}(x) = 0$ for all $x \in \balg$, so by \eqref{rmltrcirc1}, $\tr R_{\mlt}(x + aD) = a$ for all $x \in \balg$, from which it is apparent that $\ker \tr R_{\mlt} = \balg \oplus \lb 0\ra = \ker R_{\mlt}(D)$. On the other hand, if $\ker \tr R_{\mlt} = \ker R_{\mlt}(D)$, then $\ker \tr R_{\mlt} = \balg \oplus \lb 0 \ra$, so, by \eqref{rmltrcirc1}, $\tr R_{\circ}(x) = \tr R_{\mlt}(x) = 0$ for all $x \in \balg$.

If $(\balg, \circ)$ is complete, then $\tau_{\circ}$ vanishes identically, so \eqref{rmltrcirc2} yields $\tau_{\mlt} = \hat{h}$. 
As $(\balg, \circ)$ is complete it is right nil, so $R_{\circ}(x)$ is nilpotent, and there is a minimal $m \leq \dim \balg -1$ such that for all $x \in \balg$, $R_{\circ}(x)^{m+1} = 0$, while for some $x\in \balg$, $R_{\circ}(x)^{m} \neq 0$. Hence $(I + R_{\circ}(x))^{-1} = \sum_{l = 0}^{m}(-R_{\circ}(x))^{l}$. With $P_{\balg, \circ} = 1$, in \eqref{papb0} this yields \eqref{papb}. 
By \eqref{papb}, if $P$ has degree $k+1$ then $m \geq k-1$. This means there exists $x \in \balg$ such that $R_{\circ}(x)^{k-1} \neq 0$. In particular, this implies $\rnil^{k}(\balg, \circ) \neq \{0\}$. This shows \eqref{cgenil}. 
Since $(\balg, \circ)$ is complete, the action on $\balg$ of the simply-connected Lie group of $(\balg, [\dum, \dum])$ is simply transitive (see \cite{Helmstetter} or \cite{Kim-completeleftinvariant}). By Theorem I.$1$ of \cite{Auslander-affinemotions} this group is solvable, and so also $(\balg, [\dum, \dum])$ is solvable. Since $\alg/[\alg, \alg]$ is an abelian Lie algebra, $(\alg, [\dum, \dum])$ is solvable too. This proves \eqref{cge5b}.
Claim \eqref{cgewt} follows from \eqref{dptra} and \eqref{rmltrcirc1}. 

If the Lie algebra underlying $(\balg, \circ)$ is unimodular, then by Lemma \ref{trrunilemma} and \eqref{cge6}, $P$ satisfies the conditions of Theorem \ref{lsacptheorem}, and so \eqref{cge10} follows. 
\end{proof}

By Lemma \ref{graphextensionlemma}, the graph extension of a Hessian LSA with respect to a compatible derivation is an LSA equipped with a Koszul form. Lemma \ref{gecharlemma} characterizes when an LSA equipped with a Koszul form is isometrically isometric to a graph extension.

\begin{lemma}\label{gecharlemma}
For a finite-dimensional LSA $(\alg, \mlt)$ equipped with a Koszul form $\la \in \alg^{\ast}$ having associated idempotent $u$ and Hessian metric $h$ the following are equivalent:
\begin{enumerate}
\item \label{geclsa2}$\ker \la = \ker R(u)$.
\item \label{geclsa3}$h(u,u) = \la(u) \neq 0$, $R(u)^{2} = R(u)$, $\tr R(u) = 1$, and the linear span $\lb u \ra$ is a left ideal and $\la(u) \neq 0$.
\item \label{geclsa1}$h(u,u) = \la(u) \neq 0$ and the multiplication $x \circ y = x \mlt y - \la(u)^{-1}h(x, y)u$ makes $\ker \la$ an LSA for which the restriction of any nonzero multiple of $h$ is a Hessian metric and on which $L(u)$ acts as a compatible derivation. 
\end{enumerate}
In the case there hold \eqref{geclsa2}-\eqref{geclsa1}, $(\alg, \mlt, \la(u)^{-1}h)$ is isometrically isomorphic to the graph extension of $(\ker \la, \circ, \la(u)^{-1}h)$ with respect to $L(u)$.

Moreover, if there hold \eqref{geclsa2}-\eqref{geclsa1}, $L(u)$ is invertible if and only if the linear span $\lb u \ra$ equals its normalizer in $(\alg, [\dum, \dum])$.
In this case, $[\alg, \alg] = \ker \la = \ker R(u)$ and $(\alg, \mlt)$ is a simple LSA.
\end{lemma}

\begin{proof}
Suppose there holds \eqref{geclsa2} so that $\ker \la = \ker R(u)$. Then $u \in \ker \la = \ker R(u)$ implies the contradictory $u = R(u)u = 0$. Hence $h(u, u) = \la(u) \neq 0$. For $x \in \ker \la$, $R(u)x = 0$, while $R(u) u = u$, so $R(u)$ is the projection onto $\lb u \ra$ along $\ker \la$. This is equivalent to the conditions $R(u)^{2} = R(u)$ and $\tr R(u) = \text{codim}\ker R(u) = 1$. As $x \mlt u = R(u)x \in \lb u\ra$, $\lb u \ra$ is a left ideal in $(\alg, \mlt)$. Thus \eqref{geclsa2} implies \eqref{geclsa3}. If there holds \eqref{geclsa3}, then $R(u)$ is a projection with one-dimensional image, and since $R(u)u= u$ the image is $\lb u \ra$. For any $x \in \alg$ there is $c \in \rea$ such that $R(u)x = cu$. Then $c\la(u) = \la(cu) = \la(x\mlt u) = h(x, u)= \la(x)$, so $\la(u)R(u)x = \la(x)u$ from which it follows that $\ker \la \subset \ker R(u)$. This suffices to show $\ker R(u) = \ker \la$. Thus \eqref{geclsa3} implies \eqref{geclsa2}.

If $\la(u) = h(u, u) \neq 0$ then $\circ$ is defined and $x \circ y \in \ker \la$ for all $x, y \in \alg$. If $x, y, z \in \ker \la$ then
\begin{align}\label{circlsa}
(x \circ y - y\circ x)\circ z- x\circ(y\circ z) + y\circ(x \circ z) = \la(u)^{-1}R(u)(h(x, z)y - h(y, z)x).
\end{align}
From \eqref{circlsa} it is immediate that $\ker R(u) = \ker \la$ implies that $(\ker \la, \circ)$ is an LSA. Because, for $x, y, z \in \ker \la$, $h(x\circ y, z) = h(x\mlt y, z)$, the restriction to $\ker \la$ of $h$ is a Hessian metric for $(\ker \la, \circ)$. Since $\la(L(u)x) = h(u, x) = \la(x)$, $L(u)$ preserves $\ker\la$. If $x \in \ker \la$ then $h(L(u)x, y) + h(x, L(u)y) =  h(R(u)x, y) + h(u, x\mlt y) = h(x, y)$. Since $\ker R(u) = \ker \la$, if $x \in \ker \la$ then $L(u)(x\mlt y) - L(u)x \mlt y - x \mlt L(u)y = R(u)x \mlt y = 0$. Hence, for $x, y \in \ker \la$, 
\begin{align}
\la(u)(L(u)(x \circ y) - L(u)x \circ y - x \circ L(u)y) =  (h(L(u)x, y) + h(x, L(u)y - h(x, y))u  =0 ,    
\end{align}
showing that $L(u)$ is a compatible derivation of the Hessian LSA $(\ker \la, \circ, h)$. This shows that \eqref{geclsa2} implies \eqref{geclsa1}.

Suppose there holds \eqref{geclsa1}. Since $h(u, u) \neq 0$ the restriction of $h$ to $\ker \la$ is nondegenerate. If $\dim \alg > 2$ then $\dim \ker \la > 1$ and so given $0 \neq y \in \ker \la$ there can be found $x, z \in \ker \la$ such that $h(y, z) = 0$ and $h(x, z) = 1$. Susbtituting this into \eqref{circlsa} and supposing $\circ$ is left-symmetric on $\ker \la$ shows that $R(u)y = 0$. Hence $\ker \la \subset \ker R(u)$ which suffices to show \eqref{geclsa2}. If $\dim \alg = 2$ then $u$ is transverse to both $\ker \la$ and $\ker R(u)$ and so $\ker R(u) \subset \ker \la$. Since $L(u)$ is a compatible derivation of $(\ker \la, \circ)$, for all $x \in \ker \la$, $h(R(u)x, x) = h(x\mlt u, x) = h(u\mlt x, x) - h(u, x\mlt x) + h(x, u\mlt x) = h(x, x)- h(x, x) =0$. Since $h(R(u)x, u) = \la(R(u)x) = h(u, x) = \la(x) = 0$, it follows that $R(u)x = 0$ if $x \in \ker \la$, so $\ker R(u) = \ker \la$. Thus \eqref{geclsa1} implies \eqref{geclsa2}.

Suppose there hold \eqref{geclsa2}-\eqref{geclsa1}. Let $D = L(u)$ and define $\Psi: \ker \la \oplus \lb D \ra \to \alg$ by $\Psi(x + aD) = x + au$. Let $(\ker \la \oplus \lb D \ra, \hat{\mlt}, \hat{h})$ be the graph extension of $(\ker \la, \circ, \la(u)^{-1}h)$ with respect to $D = L(u)$. In particular, for $x, y \in \ker \la$, $\hat{h}(x + aD, y + bD) = ab + \la(u)^{-1}h(x, y)$, and
\begin{align}
\begin{split}
\Psi((x + aD)\hat{\mlt}(y + bD))& = \Psi(x\circ y + aDy + (ab + \la(u)^{-1}h(x, y))D )\\
&= x\circ y + aDy + (ab + \la(u)^{-1}h(x, y))u = x\mlt y + aL(u)y + ab u \\
&= (x + au) \mlt (y + bu) = \Psi(x + aD)\mlt \Psi(y + aD),
\end{split}
\end{align} 
so $\Psi$ is an algebra isomorphism. Moreover, $h(\Psi(x + aD), \Psi(y + bD)) = h(x, y) + ab\la(u) = \la(u)\hat{h}(x + aD, y + bD)$, so $\Psi$ is isometric with respect to $\hat{h}$ and $\la(u)^{-1}h$.

Suppose there hold \eqref{geclsa2}-\eqref{geclsa1}. If $L(u)$ is not invertible there exists $0 \neq x \in \alg$ such that $L(u)x = 0$. Then $\la(x) = h(u, x) = \la(L(u)x) = 0$, so $x \in \ker \la = \ker R(u)$. Hence $[x, u] = R(u)x - L(u)x = 0$ and $x$ is contained in the normalizer in $(\alg, [\dum, \dum])$ of $\lb u \ra$. It follows that if $\lb u \ra$ equals its normalizer in $(\alg, [\dum, \dum])$, then $L(u)$ is invertible. Suppose $L(u)$ is invertible and that $x \in \alg$ is contained in the normalizer of $\lb u \ra$ in $(\alg, [\dum, \dum])$. Then $\bar{x} = x - \la(u)^{-1}\la(x)u \in \ker \la = \ker R(u)$ is also contained in the normalizer of $\lb u \ra$ in $(\alg, [\dum, \dum])$, since $[\bar{x}, u] = [x, u]$. There is some $c \in \rea$ such that $L(u)\bar{x} = (L(u) - R(u))\bar{x} = [u, \bar{x}] = cu$. Since $\la(u) \neq 0$, that $0 = \la([u, \bar{x}]) = c\la(u)$ implies $c = 0$ so $L(u)\bar{x} = [u, \bar{x}] = 0$. Since $L(u)$ is invertible, this implies $\bar{x} = 0$, so $x \in \lb u \ra$. 

The conclusions that $[\alg, \alg] = \ker \la = \ker R(u)$ and $(\alg, \mlt)$ is a simple LSA when $L(u)$ is invertible follow from \eqref{ge6b} and \eqref{ge6} of Lemma \ref{graphextensionlemma} and the fact that $(\alg, \mlt, \la(u)^{-1}\hat{h})$ is isometrically isomorphic to a graph extension with respect to $L(u)$.
\end{proof}

If $\la$ is a Koszul form, then, for any nonzero $c$, $\tilde{\la} = c\la$ is also a Koszul form. The metric $\tilde{h}$ associated with $\tilde{\la}$ is $\tilde{h} = ch$. Both $\la$ and $\tilde{\la}$ determined the same idempotent, for if $\tilde{u}$ is the idempotent associated with $\tilde{\la}$, then $ch(u, x) = c\la(x) = \tilde{\la}(x) = \tilde{h}(\tilde{u}, x) = ch(\tilde{u}, x)$ for all $x \in \alg$, so $\tilde{u} = u$, by the nondegeneracy of $h$. If $\la(u) \neq 0$ then, replacing $\la$ by $\la(u)^{-1}\la$ it can be assumed $\la(u) = 1$. A Koszul form $\la$ is \textit{normalized} if $\la(u) \in \{0, 1\}$. If $\la$ is a normalized Koszul form and there hold \eqref{geclsa2}-\eqref{geclsa1} of Lemma \ref{gecharlemma}, then $(\alg, \mlt, h)$ is the graph extension of $(\ker \la, \circ, h)$ with respect to $L(u)$.

Given an LSA $(\alg, \mlt)$ and an endomorphism $D \in \eno(\alg)$ define $\Pi(\alg, D)$ to be the set of weights for the action of $D$ on $\alg$ and let 
\begin{align}
\alg^{\al} = \{x \in \alg: (D - \al I)^{p}x = 0 \,\,\text{for some}\,\, p \geq 1\}
\end{align}
be the weight space corresponding to $\al \in \Pi(\alg, D)$.

\begin{lemma}\label{weightlemma}
Let $D \in \eno(\balg)$ be a compatible derivation of the Hessian LSA $(\balg, \circ, h)$ and let $(\alg, \mlt, \hat{h})$ be the corresponding graph extension. Then
\begin{enumerate}
\item \label{gbwt1} If $\al\in \Pi(\alg, L_{\mlt}(D))$ and $\al \neq 1$ then $\al \in \Pi(\balg, D)$ and $\balg^{\al} = \alg^{\al} \subset \balg$. Consequently $\Pi(\balg, D) \subset \Pi(\alg, L_{\mlt}(D))$.
\item \label{gbwt2} $\al \in \Pi(\balg, D)$ if and only if $1 - \al \in \Pi(\balg, D)$, and $\dim \balg^{\al} = \dim \balg^{1-\al}$.
\item \label{gbwt2b} $\balg^{1} = \{0\}$ if and only if $D$ is invertible.
\item \label{gbwt3} For $\al, \be \in \Pi(\balg, D)$, $\balg^{\al}\mlt \balg^{\be} \subset \alg^{\al + \be}$ and $\balg^{\al}\circ \balg^{\be} \subset \balg^{\al + \be}$.
\item\label{gbwt4} For  $\al, \be \in \Pi(\balg, D)$, $h:\balg^{\al}\times \balg^{\be} \to \rea$ is nondegenerate if $\al + \be = 1$ and is the zero map otherwise.
\item\label{gbwt5} $[\alg^{\al}, \alg^{1-\al}] \subset \balg^{1}$. 
\end{enumerate}
If $\Pi(\balg, D) \subset (0, 1)$ and $D$ is triangularizable then $(\balg, \circ)$ is right nilpotent, so nilpotent.
\end{lemma}

\begin{proof}
Let $\lambda = \hat{h}(D, \dum)$. 
If $\al \in \Pi(\alg, L_{\mlt}(D))$ and $0 \neq a \in \alg^{\al}$, let $p \geq 1$ be the minimal integer such that $(L_{\mlt}(D) - \al I)^{p}a = 0$. By \eqref{geadj} of Lemma \ref{graphextensionlemma}, 
\begin{align}
\begin{split}
0 &= \hat{h}((L_{\mlt}(D) - \al I)^{p}a, D) = \hat{h}(a, (L_{\mlt}(D)^{\ast} - \al I)^{p}D) \\
&= \hat{h}(a, (R_{\mlt}(D) - L_{\mlt}(D) + (1-\al)I)^{p}D) = (1-\al)^{p}\la(a), 
\end{split}
\end{align}
so if $\al \neq 1$ then $a \in \balg$. This shows $\alg^{\al} \subset \balg$ if $\al \neq 1$. It follows that $\balg^{\al} = \alg^{\al}$ if $\al \neq 1$. Since $L_{\mlt}(D)$ preserves $\balg$, $\balg^{1} = \alg^{1} \cap \balg$. This proves \eqref{gbwt1}.

If $\al \in \Pi(\balg, D)$, $0 \neq a \in \balg^{\al}$, and $x \in \balg$, then there is a minimal $p \geq 1$ such that $(D - \al I)^{p}a = 0$, and so, since $D$ is a compatible derivation,
\begin{align}\label{hathax}
0 = h((D - \al I)^{p}a, x) = (-1)^{p}h(a, (D - (1 - \al)I)^{p}x).
\end{align}
Were $(D - (1 - \al)I)^{p}$ invertible on $\balg$, then, by the nondegeneracy of $h$, \eqref{hathax} would imply $a = 0$, so $(D - (1 - \al)I)^{p}$ has nonempty kernel on $\balg$, so $\balg^{1-\al}$ is nontrivial. This proves that $1 - \al \in \Pi(\balg, D)$ and $\dim \balg^{\al} \leq \dim \balg^{1-\al}$. By symmetry, $\dim \balg^{\al} = \dim \balg^{1-\al}$, and so \eqref{gbwt2} is proved. By \eqref{gbwt2}, $\balg^{1} = \{0\}$ if and only if $\balg^{0} = \{0\}$. As $D$ is invertible on $\balg$ if and only if $\balg^{0} = \{0\}$, there follows \eqref{gbwt2b}.

If $\al, \be \in \Pi(\balg, D)$, $0 \neq a \in \balg^{\al}$, and $0 \neq b \in \balg^{\be}$, then, by \eqref{ge5} of Lemma \ref{graphextensionlemma}, $(D - (\al + \be)I)(a \mlt b) = ((D - \al I)a) \mlt b + a \mlt (D - \be I)b$, so 
\begin{align}\label{halbeprod}
(D - (\al + \be)I)^{m}(a \mlt b) = \sum_{j = 0}^{m}\binom{m}{j}(D - \al I)^{j}a \mlt (D - \be I)^{m-j}b.
\end{align}
If $m \geq 2\max\{\dim \balg^{\al}, \dim \balg^{\be}\}$ then every term on the right-hand side of \eqref{halbeprod} vanishes, showing that $a \mlt b \in \alg^{\al + \be}$. Note that, although this shows $\balg^{\al}\mlt \balg^{\be} \subset \alg^{\al + \be}$, the possibility that $a\mlt b = 0$ is not excluded and this argument does not show that $\al + \be \in \Pi(\balg, D)$, and this need not be the case. 

Suppose $\al, \be \in \Pi(\balg, D)$, $0 \neq a \in \balg^{\al}$, and $0 \neq b \in \balg^{\be}$. If $\al + \be \neq 1$ then, by \eqref{gbwt1}, $\alg^{\al + \be} = \balg^{\al + \be} \subset \balg$ so $\hat{h}(a, b) = \la(a \mlt b) = 0$. Hence $a\circ b = a\mlt b$, so $\balg^{\al} \circ \balg^{\be} \subset \balg^{\al + \be}$. If $\be = 1 -\al$ then $a \mlt b \in \alg^{1}$ so $a\circ b = a \mlt b - \hat{h}(a, b)D \in \alg^{1} \cap \balg = \balg^{1}$. This proves \eqref{gbwt3}.
If $\al, \be \in \Pi(\balg, D)$ and $\al + \be \neq 1$ then $a \mlt b \in \balg^{\al + \be} \subset\ker \la$, so $h(a, b) = \hat{h}(a, b) = \la(a\mlt b) = 0$. On the other hand, if $\al \in \Pi(\balg, D)$ and $0\neq a \in \balg^{\al}$ by the nondegeneracy of $h$ on $\balg$ there exists $x \in \balg$ such that $h(a, x) \neq 0$. Since, for $\be \neq 1 - \al$, the projection of $x$ onto $\balg^{\be}$ is $h$ orthogonal to $a$, it can be assumed $x \in \balg^{1-\al}$. This proves \eqref{gbwt4}. 
Since $[\alg, \alg] \subset \ker \la$, if $a \in \alg^{\al}$ and $b \in \alg^{1-\al}$ then $[a, b] \in \alg^{1}\cap \balg = \balg^{1}$. This proves \eqref{gbwt5}.

Suppose $\Pi(\balg, D) \subset (0, 1)$ and let $0 < \al_{1} < \al_{2} < \dots < \al_{r} < 1$ be the distinct elements of $\Pi(\balg, D)$ arranged in increasing order. If $D$ is triangularizable, then $\balg$ equals the direct sum $\oplus_{i = 1}^{r}\balg^{\al_{i}}$ of the weight spaces of $D$. By \eqref{gbwt3}, each subspace $\I^{i} = \oplus_{j \geq i}\balg^{\al_{j}}$ is a two-sided ideal and the quotients $\I^{i}/\I^{i+1}$ are trivial LSAs, so, by Lemma \ref{rightnilpotentlemma}, $(\balg, \circ)$ is right nilpotent.
Since $\Pi(\balg, D) \subset (0, 1)$, $\balg = [\alg, \alg]$ is a nilpotent Lie algebra, so by Theorem \ref{trivalgtheorem}, $(\balg, \circ)$ is a nilpotent LSA.
\end{proof}

Lemma \ref{arithmeticrelationlemma} shows that the weights of a compatible triangularizable derivation acting on a complete Hessian LSA necessarily satisfy an arithmetic relation of a particular form.

\begin{lemma}\label{arithmeticrelationlemma}
Let $(\alg, \mlt, \hat{h})$ be the graph extension of the complete Hessian LSA $(\balg, \circ, h)$ with respect to a triangularizable compatible derivation $D \in \eno(\balg)$, and let $P$ be the characteristic polynomial of $(\alg, \mlt)$. Let $\Pi(\balg, D) = \{\al_{1}, \dots, \al_{r}\}$. For each integer $0 \leq l \leq \dim \balg - 1$ such that the degree $l + 2$ component of $P$ is not identically zero, there exists a partition $l+2 = \sum_{k = 1}^{r}i_{k}$ of $l+2$ as a sum of nonnegative integers $i_{1}, \dots, i_{r}$ such that $\sum_{k = 1}^{r}i_{k}\al_{k} = 1$. In particular, this is true for at least one nonnegative integer $l$, namely that for which $l + 2 = \deg P$.
\end{lemma}

\begin{proof}
Since $(\balg, \circ)$ is complete, by Lemma \ref{completegraphextensionlemma}, the characteristic polynomial $P_{\alg, \mlt}$ of its graph extension is given by \eqref{papb} and $\hat{h} = \tau_{\mlt}$, where $\tau_{\mlt}$ is the trace form of $(\alg, \mlt)$. Because $D$ is triangularizable, $\balg$ is the direct sum $\oplus_{\al_{i} \in \Pi(\balg, D)}\balg^{\al_{i}}$ of the weight spaces $\balg^{\al_{i}}$ of $D$. If $x \in \balg$ write $x = \sum_{i = 1}^{r}x_{\al_{i}}$ where $x_{\al_{i}} \in \balg^{\al_{i}}$. Consider an expression of the form $h(x, R_{\circ}(x)^{l}Dx)$. Each term $R_{\circ}(x)^{l}Dx$ is a linear combination of terms of the form $R_{\circ}(x_{\be_{1}})\dots R_{\circ}(x_{\be_{l}})Dx_{\be_{l+1}}$ where $\be_{1}, \dots, \be_{l+1}$ are not necessarily distinct elements of $\{\al_{1}, \dots, \al_{r}\} = \Pi(\balg, D)$. By \eqref{gbwt3} of Lemma \ref{weightlemma}, the term $R_{\circ}(x_{\be_{1}})\dots R_{\circ}(x_{\be_{l}})Dx_{\be_{l+1}}$ lies in $\balg^{\be_{1} + \dots + \be_{l+1}}$. Hence $\sum_{q = 1}^{l+1}\be_{q} = \sum_{i = 1}^{r}j_{i}\al_{i}$ where $j_{i}$ is the number of times $\al_{i}$ appears in the set $\{\be_{1}, \dots, \be_{l+1}\}$ and so $\sum_{q =1}^{r}j_{q} = l+1$. By \eqref{gbwt4} of Lemma \ref{weightlemma}, for the $h$-pairing of $R_{\circ}(x_{\be_{1}})\dots R_{\circ}(x_{\be_{l}})Dx_{\be_{l+1}}$ with some $x_{\al_{i}}$ to be nonzero, it must be that $1 - \al_{i} = \sum_{q = 1}^{r}j_{q}\al_{q}$. Equivalently, there are nonnegative integers $i_{1}, \dots, i_{r}$ such that $\sum_{q = 1}^{r}i_{q} = l+2$ and $1 = \sum_{q = 1}^{r}i_{q}\al_{q}$. By \eqref{papb} of Lemma \ref{completegraphextensionlemma}, the degree $l+2$ part of $P(x,a)$ is $(-1)^{l+1}h(x, R_{\circ}(x)^{l}Dx)$ for $(x, a) \in \alg$. If there is some $x \in \balg$ such that $(-1)^{l+1}h(x, R_{\circ}(x)^{l}Dx)$ is nonzero, then there must $\be_{1}, \dots, \be_{l+1}$ and $\al_{i}$ such that $h(x_{\al_{i}}, R_{\circ}(x_{\be_{1}})\dots R_{\circ}(x_{\be_{l}})Dx_{\be_{l+1}}) \neq 0$, so there are nonnegative integers $i_{1}, \dots, i_{r}$ such that $\sum_{q = 1}^{r}i_{q} = l+2$ and $1 = \sum_{q = 1}^{r}i_{q}\al_{q}$. If there is no such set of nonnegative integers $i_{1}, \dots, i_{r}$ for any $0 \leq l \leq \dim \balg -1$, then, by \eqref{papb} of Lemma \ref{completegraphextensionlemma}, $h(x,  R_{\circ}(x)^{l}Dx) = 0$ for all $0 \leq l \leq \dim\balg -1$. In this case $P_{\alg, \mlt}(x + aD) = 1 + a$. Then the Hessian $\hess P_{\alg, \mlt}$ is identically zero, which contradicts \eqref{hesslogp}, relating this Hessian to the nondegenerate bilinear form $\tau_{\mlt}$.
\end{proof}

\begin{lemma}\label{incompletelsalemma}
If the right principal idempotent $r$ of an incomplete $n$-dimensional LSA $(\alg, \mlt)$ with nondegenerate trace form $\tau$ and characteristic polynomial $P$ satisfies $\ker \tr R = \ker R(r)$ then:
\begin{enumerate}
\item\label{rclsa5} $\balg = \ker \tr R$ equipped with the multiplication $x\circ y = x \mlt y - \tau(x, y)r$ is a complete LSA.
\item\label{rclsa5b} The Lie algebra $(\alg, [\dum, \dum])$ is solvable.
\item\label{rclsa6}
$P(x + tr) = P(x) + t$ for all $t \in \rea$ and $x \in \alg$.
\item\label{rclsanil} For $x, y \in \balg$ write $R_{\circ}(x)y = y\circ x$. If $P$ has degree $k+1$ then there exists $x \in \balg$ such that $R_{\circ}(x)^{k-1} \neq 0$; consequently, $\rnil^{k}(\balg, \circ)\neq \{0\}$.
\item\label{rclsawt} $dP(E) = P$, where $E$ is the vector field defined by $E_{x} = (I + R(x))r = r + r\mlt x$.
\end{enumerate}
If, additionally, the Lie subalgebra $\balg = \ker \tr R$ is unimodular, then 
\begin{enumerate}
\setcounter{enumi}{5}
\item \label{rclsa10} there is a nonzero constant $\kc$ such that $\H(e^{P}) = \kc e^{nP}$ and the level sets of $P$ are improper affine spheres with affine normal a constant multiple of $r$.
\end{enumerate}
In particular, if, in addition to satisfying $\ker \tr R = \ker R(r)$, $L(r)$ is invertible, then there holds \eqref{rclsa10}, $[\alg, \alg] = \ker \tr R = \ker R(r)$, and $(\alg, \mlt)$ is a simple LSA.
\end{lemma}
\begin{proof}
This follows from Lemmas \ref{completegraphextensionlemma} and \ref{gecharlemma}.
\end{proof}

\begin{lemma}\label{triangularizablelemma}
Over a field of characteristic zero, a triangularizable LSA $(\alg, \mlt)$ having nondegenerate trace form $\tau$ contains an idempotent element $u$ such that the linear span $\lb u \ra$ is a left ideal and $R(u)$ is the operator of projection onto $\lb u \ra$.  That is, $R(u)^{2} = R(u)$ and $\tr R(u) = 1$. 
\end{lemma}
\begin{proof}
By assumption there exists a complete flag invariant under $L(\alg)$, so there exist $u \in \alg$ and $\la \in \alg^{\ast}$ such that $R(u)x = L(x)u = \la(x)u$ for all $x \in \alg$. Then $\tr R(u) = \la(u)$ and $\tau(x, u) = \tr R(x \mlt u) = \la(x)\tr R(u)$. Since $\tau$ is nondegenerate there is $v \in \alg$ such that $0 \neq \tau(v, u) = \la(v)\tr R(u)$, so $\la(u) = \tr R(u) \neq 0$. Replacing $u$ by $\la(u)^{-1}u$ it may be assumed that $\la(u) = \tr R(u) = 1$. Then $x\mlt u = \la(x) u$, so $\lb u \ra$ is a left ideal, and, in particular, $u\mlt u = u$. Finally, $R(u)^{2}x = \la(x)R(u)u = \la(x)u = R(u)x$.
\end{proof}
Note that Lemma \ref{triangularizablelemma} implies that a triangularizable LSA with nondegenerate trace form is incomplete. Suppose $(\alg, \mlt)$ is a triangularizable LSA with nondegenerate trace form $\tau$. It need not be the case that the right principal idempotent $r$ generate a left ideal. For example, in a clan, $r$ can be a right unit. Likewise it need not be the case that an idempotent $u$ as in Lemma \ref{triangularizablelemma} be the right principal idempotent (such a $u$ need not be uniquely determined).

Note that, for $u$ as in Lemma \ref{triangularizablelemma}, by Lemma \ref{gecharlemma} the hypothesis that $[\alg, \alg]$ have codimension one follows from the assumption that $L(u)$ be invertible.

Theorem \ref{triangularizabletheorem} is a slight refinement of Theorem \ref{triangularizabletheorem0} in the introduction. 

\begin{theorem}\label{triangularizabletheorem}
Let $(\alg, \mlt)$ be a triangularizable $n$-dimensional LSA over a field of characteristic zero and having nondegenerate trace form $\tau$ and codimension one derived Lie subalgebra $[\alg, \alg]$. Let $G$ be the simply-connected Lie group with Lie algebra $(\alg, [\dum, \dum])$. There are a nonzero constant $\kc$ and a closed unimodular subgroup $H \subset G$ having Lie algebra $[\alg, \alg]$, such that the characteristic polynomial $P$ of $(\alg, \mlt)$ solves $\H(e^{P}) = \kc e^{nP}$, and the level sets of $P$ are improper affine spheres homogeneous for the action of $H$ and having affine normals equal to a constant multiple of the right principal idempotent $r$. 
Moreover:
\begin{enumerate}
\item If $\Pi(\balg, L(r)) \subset (0, 1)$ then $(\balg, \circ)$ is right nilpotent, so nilpotent.
\item If $0 \notin \Pi(\balg, L(r))$, then $(\alg, \mlt)$ is simple.
\end{enumerate}
\end{theorem}

\begin{proof}
Let $u$ and $\lambda$ be as in the proof of Lemma \ref{triangularizablelemma}. Since $[\alg, \alg] \subset \ker \la \cap \ker \tr R$, if $[\alg, \alg]$ has codimension one then $[\alg, \alg] = \ker \la = \ker \tr R$. Since also $\la(u) = \tr R(u)$, it follows that $\la = \tr R$, and so $u$ is the right principal idempotent. The claims that the level sets of $P$ are affine spheres with affine normals a multiple of the right principal idempotent $r$ follow from Lemmas \ref{graphextensionlemma}, \ref{gecharlemma}, and \ref{incompletelsalemma}. If $P(x) = 1$ then $x \in \hol(G)0$, so $x = \hol(g)0$ for some $g \in G$. By \eqref{holgp}, $1 = P(x) = P(\hol(g)0) = \chi(g)P(0) = \chi(g)$, so $g$ is contained in the subgroup $H = \ker \chi$, which is unimodular, since the Lie algebra of $\ker \chi$ is $[\alg, \alg] = \ker \tr R$ is unimodular. It follows that the level set $\{x \in \alg: P(x) = 1\}$ is homogeneous for the action of $\ker \chi$. By the translational homogeneity of $P$ in the $r$ direction shown in \eqref{rclsa6} of Lemma \ref{principalidempotentlemma}, $\{x \in\alg: P(x) = 1 + t\} = \{x \in \alg: P(x) = 1\} + tr$ is an affine image of $\{x \in \alg: P(x) = 1\}$, so is also a homogeneous improper affine sphere. 

Suppose $\Pi(\balg, L(r)) \subset (0, 1)$ and let $0 < \al_{1} < \al_{2} < \dots < \al_{r} < 1$ be the distinct elements of $\Pi(\balg, L(r))$ arranged in increasing order. By \eqref{gbwt3} of Lemma \ref{weightlemma}, each subspace $\I^{i} = \oplus_{j \geq i}\balg^{\al_{j}}$ is a two-sided ideal and the quotients $\I^{i}/\I^{i+1}$ are trivial LSAs, so, by Lemma \ref{rightnilpotentlemma}, $(\balg, \circ)$ is right nilpotent. Since $(\alg, \mlt)$ is triangularizable, by Lemma \ref{cslemma} its underlying Lie algbera is solvable, so the underlying Lie algebra of $\balg = [\alg, \alg]$ is nilpotent. Hence, by Theorem \ref{trivalgtheorem}, $(\balg, \circ)$ is nilpotent.

That $0 \notin \Pi(\balg, L(r))$ means $L(r)$ is invertible. In this case $(\alg, \mlt)$ is simple by Lemma \ref{triangularizablelemma}.
\end{proof}

Note that the argument showing $(\balg, \circ)$ is right nilpotent if $\Pi(\balg, L(r)) \subset (0, 1)$ might fail were there supposed instead $\Pi(\balg, L(r)) \subset [0, 1]$. The problem is that one cannot conclude that the quotient $\I^{1}/\I^{2}$ is a trivial LSA.

\begin{remark}
In \cite{Mizuhara}, Mizuhara takes as hypotheses conditions corresponding to special cases of the conclusions \eqref{gbwt1}-\eqref{gbwt4} of Lemma \ref{weightlemma}. In particular, Mizuhara considers LSAs as in Theorem \ref{triangularizabletheorem} and for which $\Pi(\balg, L(r))$ is contained in $(0, 1)$ and $L(r)$ is diagonalizable. Such hypotheses exclude nontrivial LSAs. In example \ref{negeigexample} there are constructed LSA satisfying the hypotheses of Theorem \ref{triangularizabletheorem} and for which $\Pi(\balg, L(r))$ contains negative numbers or for which $L(r)$ is not diagonalizable.
\end{remark}

\begin{remark}
That a triangularizable LSA satisfying certain conditions must be simple shows a sense in which LSAs are very different from Lie algebras. 
\end{remark}

\begin{remark}
By \eqref{dptra} and \eqref{hesslogp} the characteristic polynomial $P$ of an LSA $(\alg, \mlt)$ determines the Koszul form $\tr R$ and the metric $\tau$, and, by \eqref{nablakdlogp} (with $k = 2$), $P$ determines the commutative part of the multiplication $\mlt$, but by itself $P$ does not determine in an obvious way the multiplication $\mlt$. To reconstruct $\mlt$ it seems necessary, but not sufficient without further hypotheses, to know also the underlying Lie algebra $(\alg, [\,,\,])$ and its relation to $L(r)$. This raises the question whether two triangularizable $n$-dimensional LSAs having nondegenerate trace form and codimension one derived Lie algebra and having the same characteristic polynomial must be isomorphic. Example \ref{negeigexample} shows that the answer is negative. There needs to be imposed some normalization condition situating the underlying Lie algebra of the LSA inside the symmetry algebra of $P$; it is planned to discuss this elsewhere.
\end{remark}

\section{Examples}\label{examplesection}
This section records some of the simplest illustrative examples.

\begin{example}\label{posdefsection}

Given an $(n-1)$-dimensional vector space $\ste$ with an nondegenerate inner product $g$ define an $n$-dimensional LSA $(\parab_{n}(g), \mlt)$ as follows. Equip $\alg = \ste \oplus \lb u \ra$ with the multiplication $\mlt$ defined, for $x, y \in \ste$, by $x\mlt y = g(x, y)u$, $x\mlt u = 0$, $u \mlt x = x/2$, and $u \mlt u = u$. It is straightforward to check that $(\parab_{n}, \mlt)$ is an LSA with trace form $\tau(x + au, y + bu) = ab + g(x, y)$ and characteristic polynomial $P(x + au)= 1 + a - g(x, x)/2$, whose level sets are paraboloids. In the case that $g = \delta$ is the standard Euclidean inner product, there will be written simply $\parab_{n} = \parab_{n}(\delta)$. In this case $\tau$ is positive definite. Moreover, if $g$ is positive definite then $\parab_{n}(g)$ is isomorphic to $\parab_{n}$.

\begin{lemma}\label{posdeflemma}
Suppose the $n$-dimensional LSA $(\alg, \mlt)$ carries a Koszul form $\la \in \alg^{\ast}$ for which the associated Hessian metric $h$ is positive definite. Suppose that the associated idempotent $u$ satisfies $\la(u) = 1$ and that $L(u)$ is invertible with real eigenvalues. Then $(\alg, \mlt)$ is isomorphic to $\parab_{n}$.
\end{lemma}

\begin{proof}
By the discussion following the proof of Lemma \ref{principalidempotentlemma}, $L(u)$ is self-adjoint and so diagonalizable. By the argument proving Proposition $4.1$ of \cite{Shima-homogeneoushessian}, the eigenvalues of the restriction of $L(u)$ to $\ker \la$ are all $1/2$, so that this restriction is half the identity, and, if $x, y \in \ker\la$, then $x \mlt y$ is an eigenvector of $L(u)$ with eigenvalue $1$, so equals $h(x, y)u$. It follows that $(\alg, \mlt)$ is isomorphic to $\parab_{n}(g)$ where $g$ is the restriction of $h$ to $\ker\la$.
\end{proof}

Lemma \ref{posdeflemma} can be deduced as a special case of Proposition II.$5$ of \cite{Vinberg}. Its relevance here is that an LSA yielding a homogeneous improper affine sphere that is not an elliptic paraboloid must have an indefinite signature trace form. 
\end{example}

\begin{example}\label{cayleyexample}
Let $\cayn$ be an $n$-dimensional real vector space equipped with a basis $e_{1}, \dots, e_{n}$. Let $\ep_{1}, \dots, \ep_{n}$ be the dual basis of $\cayn^{\ast}$ and write $e_{ij} = e_{i}\tensor \ep_{j}$ for the standard basis of $\eno(\cayn)$. For $x =  \sum_{i=1}^{n}x_{i}e_{i}$ there holds $e_{ij}(x) = x_{j}e_{i}$. For $x = \sum_{i=1}^{n}x_{i}e_{i}, y =  \sum_{i=1}^{n}y_{i}e_{i} \in \cayn$, define $L(x)$ and $R(x)$ by 
\begin{align}\label{cayl}
\begin{split}
L(x) &= x_{n}\sum_{i=1}^{n}ie_{ii} + \sum_{1 \leq j < i \leq n}x_{i-j}e_{ij} = 
\begin{pmatrix} 
x_{n} & 0 & 0 & \dots & 0 &0 & 0\\
x_{1} & 2x_{n} & 0 & \dots &0& 0 & 0\\
x_{2} & x_{1} & 3x_{n} & \dots & 0&0 & 0\\
\vdots & \vdots & \vdots & \dots& \vdots & \vdots & \vdots \\
x_{n-2} & x_{n-3} & x_{n-4} & \dots & x_{1} & (n-1)x_{n} & 0\\
x_{n-1} & x_{n-2} & x_{n-3} & \dots & x_{2} & x_{1} & nx_{n}
\end{pmatrix},\\
R(x) &= \sum_{i = 1}^{n}ix_{i}e_{in} + \sum_{1 \leq j < i \leq n}x_{i-j}e_{ij}= 
\begin{pmatrix} 
0 & 0 & 0 & \dots & 0 &0 & x_{1}\\
x_{1} & 0 & 0 & \dots &0& 0 & 2x_{2}\\
x_{2} & x_{1} & 0 & \dots & 0&0 & 3x_{3}\\
\vdots & \vdots & \vdots & \dots& \vdots & \vdots & \vdots \\
x_{n-2} & x_{n-3} & x_{n-4} & \dots & x_{1} & 0 & (n-1)x_{n-1}\\
x_{n-1} & x_{n-2} & x_{n-3} & \dots & x_{2} & x_{1} & nx_{n}
\end{pmatrix},
\end{split}
\end{align}
so that $x\mlt y = L(x)y = R(y)x$ and $[x, y]$ are given by
\begin{align}\label{cayl2}
\begin{split}
x\mlt y &= \sum_{i = 1}^{n}\left(ix_{n}y_{i} + \sum_{1 \leq j < i}x_{i-j}y_{j}\right)e_{i}, \\
 [x, y] &= \sum_{i = 1}^{n-1}i(y_{i}x_{n} - x_{i}y_{n})e_{i}.
\end{split}
\end{align}
(When $n = 1$, $(\cayn, \mlt)$ is just the algebra of real numbers.)
The \textit{Cayley algebra} $(\cayn, \mlt)$ is an LSA with underlying Lie bracket $[x, y]$. Since $(n+1)\tr R(x) = n(n+1)x_{n} = 2\tr L(x)$, $(\cayn, \mlt)$ is not complete. By \eqref{cayl2},
\begin{align}
\tau(x, y) = \tr R(x\mlt y) = n(x \mlt y)_{n} = n^{2}x_{n}y_{n} + n\sum_{1 \leq j < n}x_{n-j}y_{j},
\end{align}
where $(x\mlt y)_{n}$ is the coefficient of $e_{n}$ in $x \mlt y$. The right principal idempotent in $\cayn$ is $r = n^{-1}e_{n}$ and $L(r)$ is invertible with
\begin{align}
\Pi(\cayn, L(r)) = \{\tfrac{1}{n}, \tfrac{2}{n}, \dots, \tfrac{n-1}{n}, 1\}.
\end{align}
By \eqref{cayl2}, $\fili_{n-1} = [\cayn, \cayn] = \ker \tr R = \ker R(r)$ is a codimension one abelian Lie subalgebra, and $(\cayn, [\dum, \dum])$ is solvable. By Lemma \ref{incompletelsalemma} the multiplication
\begin{align}\label{fn1}
x \circ y = x\mlt y - \tfrac{1}{n}\tau(x, y)e_{n} = \sum_{i = 2}^{n-1}\left( \sum_{1 \leq j < i}x_{i-j}y_{j}\right)e_{i}
\end{align}
on $\fili_{n-1}$ makes $(\fili_{n-1}, \circ)$ a complete LSA for which the restriction of $h = \tfrac{1}{n}\tau$ to $(\fili_{n-1}, \circ)$ is a Hessian metric on which $L(r)$ acts as a compatible derivation and satisfying
\begin{align}\label{fn2}
&h(x, y) = \sum_{1 \leq j \leq n-1}x_{n-j}y_{j},&& x, y \in \fili_{n-1}.
\end{align}
In terms of the basis $\{e_{1}, \dots, e_{n-1}\}$ the relations \eqref{fn1} and \eqref{fn2} have the forms
\begin{align}
&e_{i}\circ e_{j} = \begin{cases} e_{i+j} & i + j \leq n-1\\ 0 & i+j > n - 1\end{cases},& &\tau(e_{i}, e_{j}) = \begin{cases} 1 & i + j = n\\ 0 & i+j \neq n\end{cases}, 
\end{align}
and $L(r)$ is the derivation defined by $De_{i} = \tfrac{i}{n}e_{i}$. The expression
\begin{align}
h(x\circ y, z) = \sum_{i + j + k = n, i\geq 1, j \geq 1, k\geq 1}x_{i}y_{j}z_{k},
\end{align}
is completely symmetric and $\circ$ is abelian.

From Lemma \ref{incompletelsalemma} it follows that the level sets of the characteristic polynomial $P_{n}$ are improper affine spheres homogeneous for the action of the simply connected abelian Lie group corresponding to $[\cayn, \cayn]$. As was mentioned in the introduction and will be proved now, these level sets are the hypersurfaces called \textit{Cayley hypersurfaces} in \cite{Eastwood-Ezhov}. 

\begin{lemma}
For $n \geq 1$, the characteristic polynomial $P_{n}(x_{1}, \dots, x_{n})$ of $(\cayn, \mlt)$ satisfies the recursion
\begin{align}\label{cayleyrecursion}
P_{n}(x_{1}, \dots, x_{n})  - 1 = \sum_{i = 1}^{n-1}x_{i}(1 - P_{n-i}(x_{1}, \dots, x_{n-i})) + nx_{n}, 
\end{align}
where $P_{1}(x) = 1 + x$ and, when $n = 1$, the sum in \eqref{cayleyrecursion} is understood to be trivial. 
\end{lemma}
\begin{proof}
Write 
\begin{tiny}
\begin{align}\label{cayrec}
\begin{split}&P_{n}(x) = \det(I + R(x))\\
 &= 
\begin{vmatrix} 
1 & 0 & 0 & \dots & 0 &0 & x_{1}\\
x_{1} & 1 & 0 & \dots &0& 0 & 2x_{2}\\
x_{2} & x_{1} & 1 & \dots & 0&0 & 3x_{3}\\
\vdots & \vdots & \vdots & \dots& \vdots & \vdots & \vdots \\
x_{n-3} & x_{n-4} & x_{n-5} & \dots & 1& 0 &  (n-2)x_{n-2}\\
x_{n-2} & x_{n-3} & x_{n-4} & \dots & x_{1} & 1 & (n-1)x_{n-1}\\
x_{n-1} & x_{n-2} & x_{n-3} & \dots & x_{2} & x_{1} & 1 + nx_{n}
\end{vmatrix} = \begin{vmatrix} 
1 & 0 & 0 & \dots &0 & 0 & x_{1}\\
x_{1} & 1 & 0 & \dots& 0 & 0 & 2x_{2}\\
x_{2} & x_{1} & 1 & \dots& 0 & 0 & 3x_{3}\\
\vdots & \vdots & \vdots & \dots& \vdots  & \vdots & \vdots \\
x_{n-3} & x_{n-4} & x_{n-5} & \dots & 1 & 0 &  (n-2)x_{n-2}\\
x_{n-2} & x_{n-3} & x_{n-4} & \dots & x_{1} & 1 &  1+ (n-1)x_{n-1}\\
x_{n-1} & x_{n-2} & x_{n-3} & \dots &x_{2} & x_{1} & 1 + nx_{n} + x_{1}
\end{vmatrix}\\
& =  \begin{vmatrix} 
1 & 0 & 0 & \dots & 0 & x_{1}\\
x_{1} & 1 & 0 & \dots & 0 & 2x_{2}\\
x_{2} & x_{1} & 1 & \dots & 0 & 3x_{3}\\
\vdots & \vdots & \vdots & \dots & \vdots & \vdots \\
x_{n-3} & x_{n-4} & x_{n-5} & \dots & 1 &  (n-2)x_{n-2}\\
x_{n-1} & x_{n-2} & x_{n-3} & \dots & x_{2} & 1 + nx_{n} + x_{1}
\end{vmatrix}-x_{1}P_{n-1}(x_{1}, \dots, x_{n-1}),
\end{split}
\end{align}
\end{tiny}

\noindent
where the first equality results from adding the penultimate column to the last column, and the second equality results upon 
evaluating the second determinant by expanding by cofactors down the penultimate column. Iterating the same procedure applied to the determinant appearing in the final expression of \eqref{cayrec} yields the recursion \eqref{cayleyrecursion}.
\end{proof}

\begin{lemma}\label{cayleypolynomiallemma}
The Eastwood-Ezhov polynomials $\Phi_{n}$ defined in \eqref{eepolynomials} are related to the polynomials $P_{n}$ by 
\begin{align}\label{eec}
P_{n} - 1= -n\Phi_{n}.
\end{align}
As a consequence the polynomials $\Phi_{n}$ satisfy the recursion \eqref{cayleyrecursion2}.
\end{lemma}

\begin{proof}
Write $P = P_{n}$ and $\Phi = \Phi_{n}$. By \eqref{dptra}, for $1 \leq i \leq n-1$, the vector field $E_{i} = (I + R(x))e_{i} = \pr_{i} + \sum_{j = i+1}^{n}x_{j-i}\pr_{j}$ satisfies $dP_{x}(E_{i}) = P(x)\tr R(e_{i}) = 0$. By Proposition $1$ of \cite{Eastwood-Ezhov}, $d\Phi(E_{i}) =0$. Examining the descriptions \eqref{cayleyrecursion} and \eqref{eepolynomials} of $P$ and $\Phi_{n}$, it is straightforward to see that that in $P$ and $-n\Phi$ the only monomial in which $x_{n}$ appears is $nx_{n}$, so that $\pr_{n}$ annihilates $P + n\Phi$. Hence the $n$ linearly independent vector fields $E_{1}, \dots, E_{n-1}$, and $\pr_{n}$ annihilate  $P + n\Phi$, so it is constant. As $P(0) = 1$ and $\Phi(0) = 0$, the relation \eqref{eec} follows. The recursion \eqref{cayleyrecursion2} for $\Phi_{n}$ follows from \eqref{cayleyrecursion} and \eqref{eec}.
\end{proof}

Since the polynomials \eqref{cayleyrecursion} are related to the polynomials $\Phi_{n}$ by $P_{n} - 1= -n\Phi_{n}$, the preceding discussion of $\cayn$ proves that the Cayley hypersurfaces are homogeneous improper affine spheres in a manner different from that in \cite{Eastwood-Ezhov}.

\begin{lemma}
The characteristic polynomial $P_{n}$ of the Cayley algebra $(\cayn, \mlt)$ solves
\begin{align}\label{hepn}
\H(e^{P_{n}}) =(-1)^{n(n-1)/2}n^{n+1}e^{nP_{n}}.
\end{align}
\end{lemma}

\begin{proof}
Define $E_{i} = (I + R(x))e_{i} = \pr_{i} + \sum_{j = i+1}^{n}x_{j-i}\pr_{j}$ for $1 \leq i \leq n-1$ and $E_{n} = \pr_{n}$. (Note that $E_{n}$ is not equal to $(I + R(x))e_{n}$.)
These vector fields satisfy $\pr_{E_{i}}E_{j} = E_{i+j}$ if $i + j \leq n$ and $\pr_{E_{i}}E_{j} = 0$ if $i + j> n$. By \eqref{cayleyrecursion}, $dP_{n}(E_{n}) = -n$. Since $dP_{n}(E_{i}) = 0$ if $i < n$, there results
\begin{align}
\left(\hess P_{n} + dP_{n} \tensor dP_{n}\right)(E_{i}, E_{j}) = \begin{cases}
-n & \text{if}\,\, i + j = n,\\
n^{2} & \text{if}\,\, i=n= j,\\
0 & \text{if} \,\, i+j \neq n,\, i+j < 2n,
\end{cases}
\end{align}
from which it follows that $\det(\hess P_{n} + dP_{n} \tensor dP_{n}) = (-1)^{n(n-1)/2}n^{n+1}\Psi^{2}$, where $\Psi$ is the standard volume form such that $\Psi(E_{1}, \dots, E_{n}) = 1$. The claim \eqref{hepn} follows from Lemma \ref{improperlemma}.
\end{proof}

Here is an alternative description of $\fili_{n-1}$. Consider an $n$-dimensional real vector space $\ste$ with basis $\{\ep_{1}, \dots, \ep_{n}\}$. Let $\{\ep^{1}, \dots, \ep^{n}\}$ be the dual basis of $\std$. View $\ep_{i}\,^{j} = \ep_{i}\tensor \ep^{j}$ as an element of $\eno(\ste)$, such that $\ep_{i}\,^{j}(x) = x^{j}\ep_{i}$ for $x = \sum_{i = 1}^{n}x^{p}\ep_{p}$. Let $\eno_{0}(\ste)$ be the subspace of trace-free endomorphisms. The associative algebra structure on $\eno(\ste)$ is given via composition by $\ep_{i}\,^{j}\ep_{k}\,^{l} = \delta_{k}\,^{j}\ep_{i}\,^{l}$, where $\delta_{k}\,^{j}$ is the Kronecker delta. The associated Lie bracket is the usual commutator of endomorphisms. The Lie centralizer $\cent(J)$ in $\eno_{0}(\ste)$ of the nilpotent endomorphism $J = \sum_{i = 1}^{n}\ep_{i}\,^{i+1}$ is the $(n-1)$-dimensional subspace generated by the nontrivial powers of $J$. The map $P:\fili_{n-1} \to \cent(J)$ associating with $x = \sum_{i = 1}^{n-1}x_{i}e_{i} \in \fili_{n-1}$ the polynomial $P(x) = \sum_{i = 1}^{n-1}x_{i}J^{i}$ is a linear isomorphism, and it is straightforward to check that $P(x)P(y) = P(x \circ y)$, where juxtaposition of elements of $\eno(\ste)$ indicates composition. 
Let $E = \sum_{i = 1}^{n}\tfrac{n+1 - 2i}{2n}\ep_{i}\,^{i} \in \eno_{0}(\ste)$ ($E$ is the sum of the standard positive roots of $\eno_{0}(\ste)$, divided by $2n$). Then $[E, J^{k}] = \tfrac{k}{n}J^{k}$, so $D = \ad(E)$ is the derivation of $\cent(J)$ corresponding to $L(r)$ in the sense that $\ad(E)P(x) = P(L(r)x)$. 

Define $\fili_{\infty}$ to be the vector space of infinite sequences $(x_{1}, x_{2}, \dots)$. It is straightforward to check that the multiplication $\circ$ defined on $\fili_{\infty}$ by
\begin{align}\label{filii}
(x \circ y)_{i} = \sum_{1 \leq j < i}x_{i-j}y_{j}
\end{align}
is associative (so left-symmetric) and commutative. The subspaces $\fili_{\infty}^{k} = \{x \in \fili_{\infty}: x_{i} = 0 \,\,\text{if}\,\, i \leq k\}$ constitute a decreasing filtration of $\fili_{\infty}$. If $x \in \fili_{\infty}^{p}$ and $y \in \fili_{\infty}^{q}$ and $i \leq p + q$, then for a summand in \eqref{filii} to be nonzero it must be that $i-j > p$ and $j > q$. These inequalities together yield the vacuous condition $q < j < i - p \leq q$. This shows that $\fili_{\infty}^{p}\circ \fili_{\infty}^{q} \subset \fili_{\infty}^{p+q}$. Define $\pi_{n}:\fili_{\infty} \to \fili_{n}$ to be the projection onto the first $n$ components, so that $\ker \pi_{n} = \fili_{\infty}^{n}$. The preceding implies that $\pi_{n}(x \circ y) = \pi_{n}(x) \circ \pi_{n}(y)$, so that $\pi_{n}$ is a left-symmetric homomorphism. Define $\la_{n}:\fili_{\infty} \to \rea$ by $\la_{n}(x) = x_{n}$. It is claimed that the metric $h$ on $\fili_{n-1}$ is defined by $h(x, y) = \la_{n}(\bar{x} \circ \bar{y})$ where $\bar{x}$ and $\bar{y}$ are any elements of $\fili_{\infty}$ such that $\pi_{n-1}(\bar{x}) = x$ and $\pi_{n-1}(\bar{y}) = y$. If $\pi_{n-1}(\hat{x}) = x$, then $\hat{x} - \bar{x} \in \fili_{\infty}^{n-1}$, and so its product with $\bar{y}$ is contained in $\fili_{\infty}^{n}$. Consequently, $\la_{n}(\bar{x}\circ \bar{y})$ does not depend on the choice of $\bar{x}$ and $\bar{y}$. It is given explicitly by \eqref{fn2}, and this shows that it is nondegenerate. 

The space $(\fili_{\infty}, \circ)$ can be viewed as the space of formal power series $\sum_{i \geq 1}x_{i}t^{i}$ with no constant term with its usual multiplication. It is straightforward to check that the formal Euler operator $E:\fili_{\infty} \to \fili_{\infty}$ defined by $E(x)_{i} = ix_{i}$ is a derivation of $\circ$, that is $E(x\circ y) = E(x)\circ y + x \circ E(y)$.
\end{example}

\begin{example}
Here is given an example of a $6$-dimensional incomplete LSA $(\alg, \mlt)$ with nondegenerate trace form and satisfying the conditions of Theorem \ref{triangularizabletheorem}. With respect to the standard basis on $\alg = \rea^{6}$ the left and right multiplication operators are given by
\begin{small}
\begin{align}
&L(x) = \begin{pmatrix} 
\tfrac{1}{4}x_{6} & 0 & 0 & 0 & 0 & 0\\
0 & \tfrac{1}{4}x_{6} & 0 & 0 & 0 & 0\\    
0 & 0 & \tfrac{1}{2}x_{6} & 0 & 0 & 0\\
0 & 0 & x_{2} & \tfrac{3}{4}x_{6} & 0 & 0\\
x_{3} & x_{3} & x_{1} + 2x_{2}& 0 & \tfrac{3}{4}x_{6} & 0 \\
6x_{4} & 6x_{5} & 6x_{3} & 6x_{1} & 6x_{2} & x_{6}
\end{pmatrix},&&
&R(x) = \begin{pmatrix} 
0 & 0 & 0 & 0 & 0 & \tfrac{1}{4}x_{1}\\
0 & 0 & 0 & 0 & 0 & \tfrac{1}{4}x_{2}\\    
0 & 0 & 0 & 0 & 0 & \tfrac{1}{2}x_{3}\\
0 & x_{3} & 0 & 0 & 0 & \tfrac{3}{4}x_{4}\\
x_{3} & 2x_{3} & x_{1} + x_{2}& 0 & 0 & \tfrac{3}{4}x_{5} \\
6x_{4} & 6x_{5} & 6x_{3} & 6x_{1} & 6x_{2} & x_{6}
\end{pmatrix}.
\end{align}
\end{small}
That this defines an LSA is tedious but straightforward. The trace form is
\begin{align}
\tau(x, y) = 6(x_{1}y_{4} + x_{4}y_{1} + x_{2}y_{5} + x_{5}y_{2} + x_{3}y_{3}) + x_{6}y_{6}.
\end{align}
The derived algebra $[\alg, \alg] = \ker \tr R = \ker R(r)$ has codimension one and is nilpotent but not abelian, for $[[\alg, \alg], [\alg, \alg]]$ is a two-dimensional abelian subalgebra. The characteristic polynomial is
\begin{align}
P(x) = 6x_1x_2x_3 + 6x_2^2x_3 - 6x_1x_4 - 6x_2x_5 - 3x_3^2 + x_6 + 1,
\end{align}
and $p(x_{1}, \dots, x_{5}) = P(x) - 1 - 6x_{6}$ solves $\H(p) = - 6^{5}$. By Lemma \ref{incompletelsalemma} the level sets of $P$ are improper affine spheres with affine normals parallel to $\pr_{x_{6}}$.
\end{example}

\begin{example}\label{negeigexample}
Here is described a class of examples of incomplete LSAs $(\alg, \mlt)$ having nondegenerate trace forms, satisfying the conditions of Theorem \ref{triangularizabletheorem}, and such that $L(r)$ has various interesting properties, such as a negative eigenvalue or a nontrivial Jordan block.

Consider a trivial LSA $(\balg, \circ)$. Any metric $h$ is necessarily Hessian, and any endomorphism $D \in \eno(\balg)$ is a derivation. Any invertible endomorphism of $\balg$ is an automorphism of $\circ$, so modulo automorphisms of $(\balg, \circ)$ it can be assumed that $D$ has its real Jordan normal form. Supposing a particular relation between $D$ and $h$ leads to the following examples.

 Work with respect to the standard basis $e_{1}, \dots, e_{n+1}$ of $\alg = \rea^{n+1}$ and write $x = \bar{x} + \hat{x}e_{n+1}$ where $\bar{x} = x_{1}e_{1} + \dots + x_{n}e_{n}$. Let $\balg$ be the subspace of $\alg$ spanned by $e_{1}, \dots, e_{n}$. Define $J \in \eno(\balg)$ by $J(e_{n+1-i}) = e_{i}$. Let $D \in \eno(\balg)$ be diagonal with respect to the standard basis and such that the diagonal entries $d_{i}$ defined by $D(e_{i}) = d_{i}e_{i}$ satisfy $d_{n+1-i} = 1- d_{i}$. This is equivalent to the relation $JD + DJ = J$. Let $N\in \eno(\balg)$ be a nilpotent endomorphism of $\balg$, strictly lower triangular with respect to the standard basis, and satisfying $[D, N] = 0$ and $N^{t}J = -JN$, where the transposed endomorphism $N^{t}$ is defined using the Euclidean structure on $\balg$ for which the standard basis is orthonormal. Examples showing that there exist $J$, $D$, and $N$ satisfying the conditions $DJ + JD = J$, $[D, N] = 0$, and $N^{t}J + JN = 0$ are given below. The conditions $DJ + JD = J$ and $N^{t}J + JN = 0$ correspond to requiring that the derivation of the trivial LSA structure on $\balg$ having matrix $D + N$ with respect to the standard basis be compatible with $h$. Using the mentioned Euclidean structure, $x \in \alg$ is identified with a column vector with components $x_{i}$. Let $\al$ be a nonzero real number. The left and right multiplication operators are defined by
\begin{align}
&L(x) = \begin{pmatrix*}[c] \hat{x}(D + N) & 0 \\ \al\bar{x}^{t}J & \hat{x}\end{pmatrix*},& &R(x) = \begin{pmatrix*}[c] 0 & (D + N)\bar{x} \\ \al\bar{x}^{t}J & \hat{x}\end{pmatrix*}.
\end{align}
That this defines an LSA is straightforward using the identities $DJ + JD = J$, $[D, N] = 0$, and $N^{t}J = -JN$. The multiplication and underlying Lie bracket are given explicitly by 
\begin{align}\label{dnjex}
&x \mlt y = \begin{pmatrix*}[c]\hat{x}(D+N)\bar{y} \\ \al\bar{x}^{t}J\bar{y} + \hat{x}\hat{y}\end{pmatrix*},& &[x, y] = \begin{pmatrix*}[c]\hat{x}(D+N)\bar{y}  - \hat{y}(D+N)\bar{x}\\ 0\end{pmatrix*}
\end{align}
The trace form is 
\begin{align}
\tau(x, y) = \hat{x}\hat{y} + \al \bar{x}^{t}J\bar{y}  = x_{n+1}y_{n+1} + \al\sum_{i = 1}^{n}x_{n+1-i}y_{i},
\end{align}
which is evidently nondegenerate. 
The induced multiplication $\circ$ on $\balg$ is trivial. The characteristic polynomial is $P(x) = 1 + \hat{x} - \al \bar{x}^{t}JD\bar{x}$. Note that neither $\tau$ nor $P$ depends on the choice of $N$. 

The derived algebra $[\alg, \alg]$ is abelian. If $D$ is invertible then, because $D$ and $N$ commute, $D^{-1}N$ is nilpotent, so $D + N = D(I + D^{-1}N)$ is invertible and it is apparent from \eqref{dnjex} that $[\alg, \alg]$ has codimension one in $\alg$. However, if $D$ is not invertible, then it can be that $[\alg, \alg]$ has codimension greater than one in $\alg$, as occurs for example for the LSA given by 
\begin{align}
& J = \begin{pmatrix*}[r]0  & 1 \\ 1 & 0 \end{pmatrix*},&& D = \begin{pmatrix*}[r]1  & 0 \\ 0 & 0 \end{pmatrix*},&
\end{align}
$N = 0$, and $\al = 1$.

A simple special case of the preceding showing that $\Pi(\alg, L(r))$ can contain nonpositive numbers and irrational numbers is the following example of dimension $n+1 = 4$. For any $\si \in \rea$, define
\begin{align}
& J = \begin{pmatrix*}[r]0 & 0 & 1 \\ 0 & 1 & 0  \\ 1 & 0 & 0 \end{pmatrix*},&&  D = \begin{pmatrix*}[r]  \si & 0 & 0  \\ 0 & 1/2 & 0  \\ 0 & 0 & 1- \si \end{pmatrix*},
\end{align}
and $N = 0$.
With respect to the standard basis on $\alg = \rea^{4}$ the left and right multiplication operators are
\begin{align}
&L(x) = \begin{pmatrix} 
\si x_{4} & 0 & 0 & 0 \\
0 & \tfrac{1}{2}x_{4} & 0 & 0 \\    
0 & 0 & (1-\si)x_{4} & 0 \\
x_{3} & x_{2} & x_{1} & x_{4}
\end{pmatrix},&&
&R(x) =  \begin{pmatrix} 
0 & 0 & 0 & \si x_{1} \\
0 & 0 & 0 & \tfrac{1}{2}x_{2} \\    
0 & 0 & 0 & (1-\si)x_{3} \\
x_{3} & x_{2} & x_{1} & x_{4}
\end{pmatrix}.
\end{align}
The trace form is
\begin{align}
\tau(x, y) = x_{1}y_{3} + x_{3}y_{1} + x_{2}y_{2} + x_{4}y_{4}.
\end{align}
The derived algebra $[\alg, \alg]$ is abelian and, if $\si \notin \{0, 1\}$, it has codimension one and equals $\ker \tr R = \ker R(r)$. The characteristic polynomial is $P(x) = 1 + x_{4} - x_{1}x_{3} - \tfrac{1}{2}x_{2}^{2}$. 

A simple special case of the preceding showing that $L(r)$ need not be diagonalizable is the following example of dimension $n+1 = 4$. Define
\begin{align}
& J = \begin{pmatrix*}[r]0 & 0 & 1 \\ 0 & 1 & 0  \\ 1 & 0 & 0 \end{pmatrix*},&&  N = \begin{pmatrix*}[r]  0 & 0 & 0  \\ t & 0 & 0  \\ 0 & -t & 0 \end{pmatrix*},
\end{align}
where $t \in \rea$, and let $D = (1/2)I$. The left and right multiplication operators are
\begin{align}
&L(x) = \begin{pmatrix*}[c] \tfrac{1}{2}x_{4}I + x_{4}N & 0 \\ \bar{x}^{t}J & x_{4}\end{pmatrix*},& &R(x) = \begin{pmatrix*}[c] 0 & \tfrac{1}{2}\bar{x} + N\bar{x} \\ \bar{x}^{t}J & x_{4}\end{pmatrix*}.
\end{align}
That this defines an LSA is straightforward using the identity $N^{t}J = -JN$. The trace form is 
\begin{align}
\tau(x, y) = \bar{x}^{t}J\bar{y} + x_{4}y_{4} = x_{1}y_{3} + x_{3}y_{1} + x_{2}y_{2} + x_{4}y_{4},
\end{align}
which is evidently nondegenerate. 
The derived algebra $[\alg, \alg] = \ker \tr R = \ker R(r)$ has codimension one and is abelian. The characteristic polynomial is $P(x) = 1 + x_{4} -x_{1}x_{3} - \tfrac{1}{2}x_{2}^{2}$. 
\end{example}

In the preceding examples $\tau$ and $P$ are the same, although the underlying LSAs $(\alg, \mlt)$ are not isomorphic because they have different $\Pi(\alg, L(r))$. This shows that a triangularizable LSA having nondegenerate trace form and codimension one derived Lie algebra is not determined up to isomorphism by its characteristic polynomial. On the other hand, the LSAs $(\balg, \circ)$ are in all cases trivial, so isomorphic.

\bibliographystyle{amsplain}
\def\polhk#1{\setbox0=\hbox{#1}{\ooalign{\hidewidth
  \lower1.5ex\hbox{`}\hidewidth\crcr\unhbox0}}} \def\cprime{$'$}
  \def\cprime{$'$} \def\cprime{$'$}
  \def\polhk#1{\setbox0=\hbox{#1}{\ooalign{\hidewidth
  \lower1.5ex\hbox{`}\hidewidth\crcr\unhbox0}}} \def\cprime{$'$}
  \def\cprime{$'$} \def\cprime{$'$} \def\cprime{$'$} \def\cprime{$'$}
  \def\cprime{$'$} \def\polhk#1{\setbox0=\hbox{#1}{\ooalign{\hidewidth
  \lower1.5ex\hbox{`}\hidewidth\crcr\unhbox0}}} \def\cprime{$'$}
  \def\Dbar{\leavevmode\lower.6ex\hbox to 0pt{\hskip-.23ex \accent"16\hss}D}
  \def\cprime{$'$} \def\cprime{$'$} \def\cprime{$'$} \def\cprime{$'$}
  \def\cprime{$'$} \def\cprime{$'$} \def\cprime{$'$} \def\cprime{$'$}
  \def\cprime{$'$} \def\cprime{$'$} \def\cprime{$'$} \def\dbar{\leavevmode\hbox
  to 0pt{\hskip.2ex \accent"16\hss}d} \def\cprime{$'$} \def\cprime{$'$}
  \def\cprime{$'$} \def\cprime{$'$} \def\cprime{$'$} \def\cprime{$'$}
  \def\cprime{$'$} \def\cprime{$'$} \def\cprime{$'$} \def\cprime{$'$}
  \def\cprime{$'$} \def\cprime{$'$} \def\cprime{$'$} \def\cprime{$'$}
  \def\cprime{$'$} \def\cprime{$'$} \def\cprime{$'$} \def\cprime{$'$}
  \def\cprime{$'$} \def\cprime{$'$} \def\cprime{$'$} \def\cprime{$'$}
  \def\cprime{$'$} \def\cprime{$'$} \def\cprime{$'$} \def\cprime{$'$}
  \def\cprime{$'$} \def\cprime{$'$} \def\cprime{$'$} \def\cprime{$'$}
  \def\cprime{$'$} \def\cprime{$'$} \def\cprime{$'$} \def\cprime{$'$}
  \def\cprime{$'$} \def\cprime{$'$}
\providecommand{\bysame}{\leavevmode\hbox to3em{\hrulefill}\thinspace}
\providecommand{\MR}{\relax\ifhmode\unskip\space\fi MR }
\providecommand{\MRhref}[2]{%
  \href{http://www.ams.org/mathscinet-getitem?mr=#1}{#2}
}
\providecommand{\href}[2]{#2}

\end{document}